\documentclass[11pt,reqno,a4paper]{amsart}
\pdfoutput=1

\usepackage{latexsym,amsthm,amssymb,amsmath,amscd,mathtools,hyperref,xpatch}
\mathtoolsset{showonlyrefs}
\usepackage{tikz-cd}
\tikzcdset{
  cells={font=\everymath\expandafter{\the\everymath\displaystyle}},
}
\usetikzlibrary{matrix,arrows,decorations.pathmorphing}
\usepackage{bm}
\usepackage{fancyhdr}
\usepackage{mathtools}
\usepackage{mathrsfs} 
\usepackage{tcolorbox}
\usepackage{bbm}

\usepackage[normalem]{ulem}

\usepackage[all]{xy}

\newcommand{\nospacepunct}[1]{\makebox[0pt][l]{\,#1}}

\usepackage[top=25truemm, bottom=20truemm,right=20truemm,left=20truemm]{geometry}

\usepackage{graphicx}

\usepackage[shortlabels]{enumitem}
\setlist[enumerate]{nosep}

\usepackage{xcolor}
\newcommand\redsout{\bgroup\markoverwith{\textcolor{red}{\rule[0.5ex]{2pt}{0.8pt}}}\ULon}

\newtheorem{theorem}{Theorem}[section]

\newtheorem*{conjecture*}{Conjecture}
\newtheorem*{theorem*}{Theorem}
\newtheorem{lemma}[theorem]{Lemma}
\newtheorem*{lemma*}{Lemma}
\newtheorem{corollary}[theorem]{Corollary}
\newtheorem*{corollary*}{Corollary}
\newtheorem{proposition}[theorem]{Proposition}
\newtheorem{remark}[theorem]{Remark}
\newtheorem{definition}[theorem]{Definition}
\newtheorem*{definition*}{Definition}
\newtheorem*{definitions*}{Definitions}

\newtheorem*{question*}{Question}

\newtheorem*{questions*}{Questions}

\newtheorem{example}[theorem]{Example}

\newtheorem{thm}{Theorem}[section]

\newcommand{\cA}{\mathcal A}
\newcommand{\cB}{\mathcal B}
\newcommand{\cC}{\mathcal C}
\newcommand{\cD}{\mathcal D}

\newcommand{\cG}{\mathcal G}
\newcommand{\cH}{\mathcal H}

\newcommand{\cK}{\mathcal K}

\newcommand{\cO}{\mathcal O}
\newcommand{\cP}{\mathcal P}

\def\Az{\mathbb{A}}
\def\Cz{\mathbb{C}}

\def\Mz{\mathbb{M}}
\def\Nz{\mathbb{N}}

\def\Iz{\mathbb{I}}

\def\Qz{\mathbb{Q}}
\def\Rz{\mathbb{R}}

\def\Zz{\mathbb{Z}}

\def\1z{\mathbb{1}}
\newcommand{\fA}{\mathfrak A}

\def\SEMI{\mbox{$\times\kern-2pt\vrule height5pt width.6pt \kern3pt $}}

\newcommand{\Hom}{{\rm Hom}\,}
\newcommand{\End}{{\rm End}\,}
\newcommand{\Aut}{{\rm Aut}\,}

\newcommand{\ind}{\textup{ind}}

\newcommand{\id}{{\rm id}}

\newcommand{\res}{{\rm res}}
\newcommand{\sign}{\textup{sign}}

\newcommand{\acts}{\curvearrowright} 

\newcommand{\fg}[1]{[\![#1]\!]}

\renewcommand{\Im}{\textup{Im}}
\renewcommand{\phi}{\varphi}

\newcommand{\pr}{\textup{pr}}

\newcommand{\ccor}{\textup{cor}}

\newcommand{\ol}{\overline}

\newcommand{\incl}{\textup{incl}}

\newcommand{\lwedge}{{\textstyle\bigwedge}}

\newcommand{\ev}{\textup{ev}}
\newcommand{\Ind}{\textup{Ind}}
\newcommand{\Ad}{\textup{Ad}}
\newcommand{\K}{\textup{K}}
\newcommand{\KK}{\textup{KK}}

\newcommand{\M}{\textup{M}}

\newcommand{\A}{\mathscr{A}}

\renewcommand{\H}{\textup{H}}

\renewcommand{\A}{\mathscr{A}}
\renewcommand{\S}{\mathscr{S}}
\newcommand{\Tor}{\textup{Tor}}
\newcommand{\gp}[1]{\langle #1\rangle}

\newcommand{\tor}{\textup{tor}}
\newcommand{\Homeo}{\textup{Homeo}}

\DeclareMathOperator{\supp}{supp}
\DeclareMathOperator{\coker}{coker}

\begin{document}

\title{Groupoid homology and K-theory for algebraic actions from number theory}

\thispagestyle{fancy}

\author{Chris Bruce}\thanks{C. Bruce has received funding from the European Union’s Horizon 2020 research and innovation programme under the Marie Sklodowska-Curie grant agreement No 101022531 and the European Research Council (ERC) under the European Union’s Horizon 2020 research and innovation programme (grant agreement No. 817597).}
\address[Chris Bruce]{School of Mathematics, Statistics and Physics, Herschel Building, Newcastle University, Newcastle upon Tyne, NE1 7RU, UK}
\email[Bruce]{chris.bruce@newcastle.ac.uk}

\author{Yosuke Kubota}\thanks{Y. Kubota is supported by RIKEN iTHEMS and JSPS KAKENHI Grant Numbers 22K13910, JPMJCR19T2.}
\address[Yosuke Kubota]{Graduate School of Science, Kyoto University, Kitashirakawa Oiwake-cho, Sakyo-ku, Kyoto 606-8502, Japan}
\email[Kubota]{ykubota@math.kyoto-u.ac.jp}

\author{Takuya Takeishi}\thanks{T. Takeishi is supported by JSPS KAKENHI Grant Number JP19K14551 and JP24K06780.}
\address[Takuya Takeishi]{Faculty of Arts and Sciences, Kyoto Institute of Technology, Matsugasaki, Sakyo-ku, Kyoto, Japan}
\email[Takeishi]{takeishi@kit.ac.jp}

\date{\today}

\begin{abstract}
 We compute the groupoid homology for the ample groupoids associated with algebraic actions from rings of algebraic integers and integral dynamics. We derive results for the homology of the topological full groups associated with rings of algebraic integers, and we use our groupoid homology calculation to compute the K-theory for ring C*-algebras of rings of algebraic integers, recovering the results of Cuntz and Li and of Li and L\"{uck} without using Cuntz--Li duality. Moreover, we compute the K-theory for C*-algebras attached to integral dynamics, resolving the conjecture by Barlak, Omland, and Stammeier in full generality.   
\end{abstract}

\maketitle

\section{Introduction}

\subsection{Context}
Renault pioneered the study of groupoid C*-algebras over 40 years ago \cite{Ren}. Today, groupoid C*-algebras play a central role in C*-algebra theory and many of the most important classes of C*-algebras have groupoid models, e.g., by recent work of Li, every classifiable simple  C*-algebra can be realized as the reduced (twisted) C*-algebra of an \'{e}tale groupoid \cite{Li:Invent}.
A homology theory for \'{e}tale groupoids was introduced by Crainic and Moerdijk \cite{CM}, which has seen a great deal of interest from operator algebraists since work of Matui, where a striking connection between groupoid homology and K-theory of C*-algebras was discovered \cite{Mat16}. Indeed, Matui made the following conjecture:
\begin{conjecture*}[Matui's HK conjecture]
If $\cG$ is a topologically principal minimal Hausdorff ample groupoid with compact unit space, then there are isomorphisms
\[
\bigoplus_{i=0}^\infty \H_{2i}(\cG)\cong \K_0(C_r^*(\cG)) \quad\text{ and } \quad \bigoplus_{i=0}^\infty \H_{2i+1}(\cG)\cong \K_1(C_r^*(\cG)).
\]
\end{conjecture*}
Scarparo showed that the HK conjecture does not hold in general \cite{Scar} (see also \cite{Deeley}), but it has led to a great deal of important work, and a refined rational version has recently been proposed, see \cite{DW}. It is an interesting question whether or not any given class of groupoids satisfies the conclusion of the HK conjecture. In this article, we shall say that a groupoid $\cG$ has the \emph{HK property} if the conclusion of Matui's HK conjecture holds for $\cG$. 

The topological full group of an ample groupoid is the group $\fg{\cG}$ of homeomorphisms of the unit space of the groupoid whose elements are locally given given by the action of some groupoid element. Such groups were introduced by Giordano, Putnam, and Skau for Cantor minimal systems \cite{GPS} and then for ample groupoids by Matui \cite{Mat12}.
Topological full groups turn out to be very interesting from the point of view of group theory, and they have provided several celebrated example classes of groups \cite{JNdlS, JM,NekJOT, Nek18,SWZ}.

Matui conjectured the existence of a three-term short exact sequence relating the groupoid homology of an ample groupoid $\cG$ with the first homology (i.e., abelianization) of the topological full group $\fg{\cG}$, see \cite[Conjecture~2.9]{Mat16}. This conjecture was recently resolved in a very strong sense for all groupoids with comparison by Li \cite{Li:TFG}, which means that one can access certain group-theoretic properties of topological full groups by computing groupoid homology. Moreover, topological full groups often witness strong rigidity phenomena in the sense that the abstract isomorphism class of the group characterizes the groupoid up to isomorphism, see \cite{Mat15} and \cite{Rubin}.

\subsection{Groupoids from algebraic actions}

A particularly interesting setting in which the aforementioned interactions between groupoids, groups, and C*-algebras play out is that of algebraic actions of semigroups, which have been recently studied by the first-named author and Xin Li \cite{BruceLi2,BruceLi3}. One of the more surprising theorems in this setting is the following complete rigidity result for algebraic actions from number theory, which served as the motivation for the present article. If $R$ is the ring of algebraic integers in a number field $K$, then the multiplicative monoid $R^\times\coloneqq R\setminus\{0\}$ acts on the additive group of $R$ by multiplication, giving rise to an  ample groupoid $\cG_{R^\times\acts R}$ which models the ring C*-algebra of $R$ as defined by Cuntz and Li \cite{CuntzLi1} (see also \cite{Li:ring}).
\begin{theorem*}[{\cite[Corollary E]{BruceLi3}}]
    Let $K_1$ and $K_2$ be algebraic number fields with rings of algebraic integers $R_1$ and $R_2$, respectively. Then, the following statements are equivalent: 
        \begin{enumerate}[\upshape(i)]      
        \item $K_1$ and $K_2$ are isomorphic (equivalently, $R_1\cong R_2$);
        \item the groupoids $\cG_{R_1^\times\acts R_1}$ and $\cG_{R_2^\times\acts R_2}$ are isomorphic;      
        \item the groups $\fg{\cG_{R_1^\times\acts R_1}}$ and $\fg{\cG_{R_2^\times\acts R_2}}$ are isomorphic;
        \item the commutator subgroups $\fg{\cG_{R_1^\times\acts R_1}}'$ and $\fg{\cG_{R_2^\times\acts R_2}}'$ are isomorphic;
         \item the ring C*-algebras $\fA[R_1]$ and $\fA[R_2]$ are isomorphic via a Cartan-preserving *-isomorphism.
    \end{enumerate}
\end{theorem*}

\subsection{Main results} The above rigidity theorem is grounds for a systematic study of the groups $\fg{\cG_{R^\times\acts R}}$. As explained above, we can use \cite{Li:TFG} to pass from homological information about the groupoid $\cG_{R^\times\acts R}$ to group-theoretic properties of $\fg{\cG_{R^\times\acts R}}$. Informed by this, our first result is on the homology of the groupoids $\cG_{R^\times\acts R}$:

\begin{thm}[Theorem~\ref{thm:gpdhomology}]
\label{thm:A}
For a number field $K$ of degree $d=[K:\Qz]$ with ring of integers $R$, we have $\H_n(\cG_{R^\times\acts R})=0$ for $0\leq n<d$. Moreover,
\begin{enumerate}[\upshape(i)]
\item If $K$ is totally imaginary, then 
\[ \H_n(\cG_{R^\times\acts R}) \cong \begin{cases} 
0 & \text{ if } 0\leq n<d,\\
\Zz & \mbox{ if } n=d, \\
\Zz^{\oplus\infty} \oplus \Zz/|\mu|\Zz & \mbox{ if }n=d+1, \\
\Zz^{\oplus\infty} \oplus (\Zz/|\mu|\Zz)^{\oplus\infty} & \mbox{ if }n \geq d+2,
\end{cases}\]
where $\mu=\mu_K$ is the group of roots of unity in $K$.
\item If $K$ has a real embedding, then 
\[ \H_n(\cG_{R^\times\acts R}) \cong \begin{cases} 
0 & \text{ if } 0\leq n<d,\\
\Zz/2\Zz & \mbox{ if }n=d, \\
(\Zz/2\Zz)^{\oplus\infty} & \mbox{ if }n \geq d+1. \\
\end{cases}\]
\end{enumerate}
\end{thm}

Our second theorem, building on Theorem A, is on the homology of $\fg{\cG_{R^\times\acts R}}$. In particular, we characterize when $\fg{\cG_{R^\times\acts R}}$ is simple; interestingly, this occurs for all number fields except $\Qz$:
\begin{thm}[Theorem~\ref{thm:TFGhomology}]
Let $K$ be a number field of degree $d=[K :\Qz]$ with ring of integers $R$.
    We have $\H_n(\fg{\cG_{R^\times\acts R}}) = 0$ for $0<n<d$,
    and 
    \[
    \H_d(\fg{\cG_{R^\times\acts R}}) = \begin{cases}
        \Zz & \mbox{if $K$ is totally imaginary}, \\
        \Zz/2\Zz & \mbox{if $K$ has a real embedding}.
    \end{cases}
    \]
    Moreover, $\fg{\cG_{R^\times\acts R}}$ is simple if and only if $K\neq \Qz$, and if $K=\Qz$, then $\fg{\cG_{R^\times\acts R}}^{\rm ab} \cong \Zz/2\Zz$. 
\end{thm}

The proof of Theorem~\ref{thm:A}---which is the groupoid homology analogue of the K-theory computation for ring C*-algebras---involves analysing the Lyndon--Hochschild--Serre (LHS) spectral sequence attached to the action $K\rtimes K^*\acts\Az_{K,f}$, where $\Az_{K,f}$ is the ring of finite adeles of $K$. This spectral sequence collapses at the $E^2$-page, which parallels the vanishing of the maps for iterated Pimsner--Voiculescu exact sequences in the K-theory setting \cite{CL,LiLuck}. However, in groupoid homology, we can show collapsing directly without using any Cuntz--Li duality type results.

Following the mind of the HK conjecture, this groupoid homology computation can be used as a reference to the computation of K-theory of the ring C*-algebra $\mathfrak{A}[R]$. 
The K-theory of the ring C*-algebras $\fA[R]$ from ring of algebraic integers was computed over a six year period in increasing levels of generalities by Cuntz \cite{Cuntz08}, Cuntz--Li \cite{CL,CL2}, and Li--L\"{u}ck \cite{LiLuck} in the following way. 

\begin{theorem*}[{\cite[Theorem 1.2]{LiLuck}}]
For a ring of integers $R$ in a number field $K$, we have 
\[
\K_*(\fA[R])\cong
    \begin{cases}
        \K_*(C^*(\mu)) \otimes_{\mathbb{Z}} \lwedge^* \Gamma & \text{ if $K$ is totally imaginary,}\\
        \lwedge^* \Gamma & \text{ if $K$ has a real embedding.}\\
    \end{cases}
\]
\end{theorem*}
A crucial ingredient in these proofs is Cuntz--Li duality \cite[Theorem~4.1]{CL}, which allows one to pass between finite and infinite adele rings. 
Our approach given in Theorem~\ref{thm:LuckLi}, which recovers this computation, uses only more common methods and simplifies the iterated Pimsner--Voiculescu arguments of \cite{CL,LiLuck} by using spectral sequences.
It may have applications to more general dynamical systems.

The strategy is as follows. 
First, decompose the multiplicative group $K^*$ as $\mu \times \Gamma$ by a free abelian subgroup $\Gamma \subseteq K^*$, and regard the groupoid C*-algebra $C^*_r((K \rtimes K^* ) \ltimes \Az_{K,f})$ as the crossed product $C^*_r((K \rtimes \mu ) \ltimes \Az_{K,f}) \rtimes \Gamma$. 
In general, there is a spectral sequence computing the K-groups of the crossed product C*-algebra with a free abelian group $\Gamma$. This, called (homological) Kasparov's spectral sequence \cite{Kasparov,SavinienBellissard,barlak,BOS} or the ABC spectral sequence \cite{Meyer2,PY} in the literature, comes from the fact that the Baum--Connes assembly map is an isomorphism for such $\Gamma$.
This spectral sequence is compared to the LHS spectral sequence in groupoid homology, via \emph{Raven's bivariant equivariant Chern character} \cite{Raven}, which is recently revisited by Deeley--Willett \cite{DW} for proving the rational HK-conjecture. 

Raven's Chern character is a vast generalization of the Chern character comparing topological K-theory and cohomology theory, whose domain group is identified with the K-group of the crossed product via the Baum--Connes isomorphism and the range group is identified with the groupoid homology with coefficients in $\Cz$. 
It is an isomorphism after tensoring with $\Cz$. 
As a consequence of its functoriality, it is natural to be able to compare the LHS spectral sequences of the domain and the range groups. The LHS spectral sequence of the domain is nothing but Kasparov's spectral sequence. 
This enables us to apply the concentration of the $E^2$-page of the LHS spectral sequence for groupoid homology in Theorem \ref{thm:A} to determine the higher differential of the K-theory LHS spectral sequence. Note that the K-theory LHS spectral sequence never concentrates in a single degree due to the Bott periodicity.

In this paper, we shall consider more general algebraic actions of semigroups on torsion-free finite rank abelian groups; by duality, these correspond to actions on compact connected abelian groups of finite dimension by local homeomorphisms.

Let $S$ be a cancellative reversible monoid. Assume we have an action $S\acts A$ be injective endomorphisms, where $A$ is a torsion-free finite rank abelian group (see Section~\ref{sec:algactions} for the precise definitions). 
The action $S\acts A$ naturally gives rise to an ample groupoid $\cG_{S\acts A}$ and to a concrete C*-algebra $\fA[S\acts A]$ (see \cite{BruceLi2}). By \cite{BruceLi2}, the groupoid $\cG_{S\acts A}$ is minimal, effective, and purely infinite.

Barlak, Omland, and Stammeier considered C*-algebras from \emph{integral dynamics}, which are certain algebraic actions on $\Zz$ arising from families of pairwise coprime natural numbers \cite{BOS}: Given a set $\Sigma\subseteq\Zz_{>1}$ of pairwise coprime integers, let $S$ be the multiplicative monoid generated by $\Sigma$. Then, the multiplication action $S\acts \Zz$ is an algebraic action. Our next result is the homology analogue of a K-theory conjecture by Barlak, Omland, and Stammeier:

\begin{thm}[Theorem~\ref{thm:homintdyn} and Theorem~\ref{thm:BOShomology}]
If $|\Sigma|=N<\infty$, then we have
    \[
    \H_n(\cG_{S\acts \Zz})\cong \begin{cases}
        \Zz/g_\Sigma \Zz & \text{ if } n=0,\\
        \Zz^{\oplus \binom{N}{n-1}} \oplus (\Zz/g_\Sigma \Zz)^{\oplus \binom{N-1}{n}} & \text{ if } 1\leq n\leq N-1,\\
        \Zz^N & \text{ if } n=N,\\
        \Zz & \text{ if } n=N+1,\\
        0 & \text{ if } n\geq N+2.
    \end{cases}
    \]
If $|\Sigma|=\infty$, then we have
    \[
    \H_n(\cG_{S\acts\Zz}) \cong \begin{cases}
    \Zz/g_\Sigma \Zz & \text{ if } n=0,\\
        \Zz^{\oplus\infty} \oplus (\Zz/g_\Sigma \Zz)^{\oplus \infty} & \text{ if } n\geq 1.
    \end{cases}
    \]
\end{thm}

Barlak, Omland, and Stammeier considered the so-called torsion subalgebra $C_r^*(\cG_\tor)$ of $C_r^*(\cG_{S\acts\Zz})$ (see Section~\ref{sssec:tor} for details), and they proved that the inclusion $C_r^*(\cG_\tor)\to C_r^*(\cG_{S\acts\Zz})$ splits in K-theory (\cite[Corollary~4.7]{BOS}). We prove an analogous result in groupoid homology using a LHS spectral sequence, which we then use to compute the homology of $\cG_{S\acts\Zz}$. This splitting result in homology requires several technical steps, including a functoriality result for LHS spectral sequences for groupoid homology (see Section~ \ref{sec:LHS}).

The homology of $\cG_\tor$ is relatively easy to compute (see Section~\ref{sssec:kunneth}), whereas Barlak, Omland, and Stammeier conjectured a formula for the K-theory of the C*-algebra $C_r^*(\cG_\tor)$, which they were able to prove only for a few special cases \cite[Conjecture~6.5]{BOS}.
This conjecture is particularly interesting for two reasons: first, it has an equivalent formulation in terms of the K-theory of $C^*(\cG_{S\acts\Zz})$; second, it is related to a conjecture for K-theory of higher-rank graph C*-algebras, see \cite[Section~5]{BOS}. Our final result resolves \cite[Conjecture~6.5]{BOS} in complete generality:
\begin{thm}[Theorem~\ref{thm:BOS}]
For any set $\Sigma\subseteq\Zz_{>1}$ of pairwise relatively coprime numbers with $|\Sigma|\geq 2$, the C*-algebra $C_r^*(\cG_\tor)$ is isomorphic to $\bigotimes_{s\in\Sigma}\cO_s$.
\end{thm}
Combining Theorems C and D shows that the groupoid $\cG_{S\acts\Zz}$ has the HK property.

Here we use the cohomological Kasparov spectral sequence, which is a dual of the one used in the study of ring C*-algebras. 
This spectral sequence is nothing but the Atiyah--Hirzebruch (AH) spectral sequence for the mapping torus continuous field $\mathcal{MT}(B,\beta)$ of C*-algebras (Definition \ref{def:mapping.torus}). 
Indeed, the mapping torus and the crossed product have the same K-theory via the Baum--Connes isomorphism.
Unfortunately, the comparison with Raven's Chern character no longer works here, since the $E_2$-terms are the torsion groups $\mathbb{Z}/g\mathbb{Z}$. 
Instead we use a projection-valued global section in $\mathcal{MT}(B,\beta) \otimes \mathcal{O}_{g+1}$ (which is indeed a nonunital continuous field), which is shown to exist in Lemma \ref{lem:unital.B}. 
It gives a $\ast$-homomorphism $C(\mathbb{T}^d) \to C(\mathbb{T}^d, \mathcal{MT}(B,\beta) \otimes \mathcal{O}_{g+1}) $, which compares their K-theory AH spectral sequences. This gives a strong constraint in the K-group of the $C(\mathbb{T}^d, \mathcal{MT}(B,\beta) \otimes \mathcal{O}_{g+1})$. 

\subsection{Structure of the paper}
Section~\ref{sec:prelims} contains preliminaries on groupoid homology and \'etale correspondences, and Section~\ref{sec:LHS} establishes a general functoriality result for LHS type spectral sequences for groupoid homology. Section~\ref{sec:gpdhom} contains general results on homology for the groupoids attached to algebraic actions. In Section~\ref{sec:rings}, we specialize to the case of rings, e.g., rings of algebraic integers, and give several complete homology and K-theory calculations. The final subsection of Section~\ref{sec:rings} contains the new K-theory computation for ring C*-algebras. The final Section~\ref{sec:BOS} contains the homology and K-theory computations for integral dynamics.

\textbf{Acknowledgement.} We thank Jamie Gabe and Xin Li for helpful comments. C. Bruce and T. Takeishi gratefully acknowledge support from a research-in-groups programme at the ICMS in Edinburgh. We also thank the anonymous referee for their careful and detailed reading of this article and for many helpful suggestions.

\section{Preliminaries}
\label{sec:prelims}

\subsection{Group homology}
For background on group homology, we refer the reader to \cite{Brown, Weibel}. Let $G$ and $H$ be groups, and suppose we have a $G$-module $M$ and an $H$-module $N$. If $j\colon G\to H$ and $\pi\colon M\to N$ are group homomorphisms such that $\pi(gm)=j(g)\pi(m)$ for all $g\in G$ and $m\in M$, then we let $(j,\pi)_*\colon \H_*(G,M)\to\H_*(H,N)$ be the induced map in homology, see \cite[Section~III.8]{Brown}.

We shall need the following well-known version of the K\"{u}nneth formula, which  we state here for convenience.
\begin{theorem}[cf. {\cite[Theorem~3.6.3]{Weibel}}]
\label{thm:Kunneth}
Let $G$ and $H$ be groups, $M$ a $G$-module, and $N$ an $H$-module. View $M\otimes N$ as a $G\times H$-module in the canonical way. If either $M$ or $N$ is torsion-free, then there is a short exact sequence 
\begin{equation*}
    0\to \bigoplus_{p+q=n}\H_p(G,M)\otimes\H_q(H,N)\to\H_n(G\times H,M\otimes N)\to\bigoplus_{p+q=n-1}\Tor(\H_p(G,M),\H_q(H,N))\to 0.
\end{equation*}
Moreover, this exact sequence splits if either $M$ or $N$ is $\Zz$-free.
\end{theorem}

\subsection{Ample groupoids and groupoid homology}
In this subsection, we collect some well-known definitions, notation, and terminology around ample groupoids and their homology. Groupoid homology was introduced in \cite{CM}. For ample groupoids -- as defined below -- groupoid homology is particularly interesting, as shown in \cite{Mat12}. 

Let $\cG$ be a topological groupoid with unit space $\cG^{(0)}$ and range and source maps $r,s\colon\cG\to\cG^{(0)}$, respectively. The groupoid $\cG$ is said to be \emph{\'{e}tale} if $r$ and $s$ are local homeomorphisms. An \emph{ample groupoid} is an \'{e}tale groupoid whose unit space is a totally disconnected locally compact Hausdorff space. In this paper, we will only consider Hausdorff ample groupoids.

For a locally compact Hausdorff space $X$, we let $\Zz[X]\coloneqq C_c(X,\Zz)$ denote the group of compactly supported $\Zz$-valued continuous functions on $X$. We write $\Zz X$ instead of $\Zz[X]$ when the meaning is clear.
Recall that a $\cG$-module is a nondegenerate $\Zz\cG$-module, where the ring structure on $\Zz\cG$ is given by the convolution product (see, e.g., \cite{BDGW,Miller}). Note that although $\Zz\cG$ is usually a nonunital ring, any free $\Zz\cG$-module is projective (\cite[Remark~2.9]{BDGW}).

In this paper, we shall work with the following definition of groupoid homology from \cite{BDGW,Li:TFG}. 

\begin{definition}
\label{def:homology}
 Let $\cG$ be a Hausdorff ample groupoid. The \emph{groupoid homology of $\cG$} is
 \begin{equation}
   \H_*(\cG) \coloneqq \Tor_*^{\Zz\cG}(\Zz[\cG^{(0)}], \Zz[\cG^{(0)}]). 
 \end{equation}   
\end{definition}
By \cite[Theorem~2.10]{BDGW} in the case where $\cG^{(0)}$ is $\sigma$-compact and \cite[Theorem~2.5]{Li:TFG} in general, $\H_n(\cG)$ is canonically isomorphic to the $n$-homology group of the chain complex defined in \cite[Section~3]{Mat12}, so Definition~\ref{def:homology} agrees with the definition in \cite[Definition~3.1]{Mat12}.

\subsection{Groupoid correspondences}
Let $\cG$ be a Hausdorff ample groupoid. For the terminology and notation related to $\cG$-modules and groupoid correspondences, we refer to \cite{Miller}, where it is shown that groupoid homology is functorial for groupoid correspondences.

A left action of $\cG$ on a space $X$ consists of a continuous map $r_X\colon X\to \cG^{(0)}$ -- called the \emph{anchor map} -- and a continuous map $\cG\times_{s,r_X} X\to X$, $(\alpha,x)\mapsto \alpha.x$, where
\[
\cG\times_{s,r_X} X\coloneqq \{(\alpha,x)\in \cG\times X: s(\alpha)=r_X(x)\}.
\]
Right actions are defined similarly. We write $\cG\acts X$ to denote an action of $\cG$ on a space $X$. The notions of free, basic, proper, and \'etale are defined in \cite[Section~2]{AKM}. 

\begin{definition}[cf. {\cite[Definition~3.1]{AKM}}]
 Let $\cG$ and $\cH$ be ample groupoids. A space $\Omega$ is called an \emph{\'etale correspondence from $\cG$ to $\cH$} if $\Omega$ carries right and left actions by $\cG$ and $\cH$, respectively, with anchor maps $r_\Omega$ and $s_\Omega$, respectively, such that 
 \begin{itemize}
     \item the actions of $\cG$ and $\cH$ commute;
     \item the right action of $\cH$ on $\Omega$ is free, proper, and \'etale (cf. \cite{AKM}).
 \end{itemize}
An \'etale correspondence $\Omega$ from $\cG$ to $\cH$ is \emph{proper} if the map $\Omega/\cG\to\cH^{(0)}$ induced from $r_\Omega$ is proper. 
\end{definition}

By \cite[Corollary~3.6]{Miller}, each proper \'etale correspondence $\Omega$ from $\cG$ to $\cH$ induces a homomrophism in homology $\H_*(\Omega)\colon\H_*(\cG)\to\H_*(\cH)$.

Let $\cD$ be the category whose objects are given by all triples $(\Gamma,X,\theta)$, where $\Gamma$ is a discrete group, $X$ is a compact Hausdorff space, and $\theta\colon\Gamma\to\Homeo(X)$ is a group homomorphism, and whose morphisms are defined as follows: A morphism from  $(\Gamma,X,\theta)$ to $(\Lambda,Y,\eta)$ is given by a pair $(j,p)$, where $p\colon Y\to X$ is a surjective continuous map and $j\colon\Gamma\to\Lambda$ is a group homomorphism such that $p(j(g)y)=gp(y)$ for all $g\in \Gamma$ and $y\in Y$. Composition of morphisms is defined via composition of functions.

There are two canonical functors from $\cD$ to the category of $\Zz$-graded abelian groups: The first is given on objects by $(\Gamma,X,\theta)\mapsto \H_*(\Gamma,\Zz X)$ and on morphisms by sending a pair $(j,p)$ as above to the map $\H_*(j,p^*)\colon \H_*(\Gamma,\Zz X)\to \H_*(\Lambda,\Zz Y)$, where $p^*$ is given via pre-composition. The second functor is given on objects by $(\Gamma,X,\theta)\mapsto \H_*(\Gamma\ltimes X)$ and on morphism by $(j,p)\mapsto \H_*(\Omega({j,p}))$, where $\Omega(j,p)\colon \Gamma\ltimes X\to \Lambda\ltimes Y$ is the correspondence given by $\Omega(j,p)\coloneqq \Lambda\ltimes Y$ with canonical right action by $\Lambda \ltimes Y$ and left action given by $(g,p(hy))(h,y)=(j(g)h,y)$ for $g \in \Gamma, h \in \Lambda, y \in Y$.
 
The next result is likely well-known to experts. Following the idea from \cite[Lemma~3.52]{Mad}, a proof can be given by combining Theorem~\ref{thm:LHSfunctorial} and Lemma~\ref{lem:technical}.

\begin{proposition}
\label{prop:isomfunctors}
The functors $(X,\Gamma,\theta) \mapsto \H_*(\Gamma, \Zz X)$ and $(X,\Gamma,\theta) \mapsto \H_*(\Gamma \ltimes X)$ are isomorphic. 
\end{proposition}

\subsection{Number fields}
\label{sec:ANT}
In this subsection, we briefly setup the notation and basic results on number fields that we will use in Section~\ref{sec:rings}. We refer the reader to \cite{Neu} for a comprehensive treatment or to \cite{Stevenhagen} for a concise overview.

Let $K$ be a number field, i.e., a finite degree field extension of $\Qz$. The degree of $K$ as a vector space over $\Qz$ will be denoted by $[K:\Qz]$. The ring of integers in $K$ is the integral closure of $\Zz$ in $K$, i.e., the subring $R\subseteq K$ defined by
\begin{equation*}
    R\coloneqq \{a\in K : f(a)=0\text{ for some monic, nonzero polynomial }f\in\Zz[x]\}.
\end{equation*}
The additive group of $R$ is isomorphic to $\Zz^{[K:\Qz]}$. 

The torsion subgroup of the multiplicative group $K^\times\coloneqq K\setminus\{0\}$ is the finite cyclic group consisting of the roots of unity in $R$. It will be denoted by $\mu$. By Dirichlet's Unit Theorem (\cite[Theorem~I.7.4]{Neu}), the unit group $R^*$ of $R$ is isomorphic to $\mu\times \Zz^{r+s-1}$, where $r$ is the number of field embeddings of $K$ into $\Rz$ and $s$ is the number of complex conjugate pairs of embeddings of $K$ into $\Cz$.

The quotient $K^*/\mu$ is a free abelian group of countably infinite rank, a well-known fact that we we use several times in this paper. This can be proven using the exact sequence appearing shortly after \cite[Corollary~I.3.9]{Neu} combined with Dirichlet's Unit Theorem.

\section{The Lyndon--Hochschild--Serre type spectral sequence for groupoid homology}
\label{sec:LHS}

Let $\cG$ be an ample groupoid, $\Gamma$ a discrete group, and $\varphi \colon \cG\to \Gamma$ a cocycle (i.e., a groupoid homomorphism). The skew product $\cG\times_\varphi\Gamma$ is the transformation groupoid $\cG \ltimes (\cG^{(0)} \times \Gamma)$ via the $\cG$-action given by the product of the multiplication and $\varphi$. More explicitly, the groupoid which is $\cG\times\Gamma$ as a topological space with groupoid operations defined as follows:
\begin{itemize}
    \item $(\alpha,g)$ and $(\beta,h)$ are composable if $\alpha$ and $\beta$ are composable and $h=g\varphi(\alpha)$, in which case the product is given by $(\alpha,g)(\beta,g\varphi(\alpha))\coloneqq(\alpha\beta,g)$.
    \item $(\alpha,g)^{-1}\coloneqq(\alpha^{-1},g\varphi(\alpha))$.
\end{itemize}
Note that $(\cG\times_\varphi\Gamma)^{(0)}=\cG^{(0)}\times\Gamma$, $r(\alpha,g)=(r(\alpha),g)$, and $s(\alpha,g)=(s(\alpha),g\varphi(\alpha))$.
The group $\Gamma$ acts on $\cG\times_\varphi\Gamma$ by $g.(\alpha,g')=(\alpha,gg')$.

By \cite[Theorem~3.8]{Mat12}, there is an \emph{LHS type spectral sequence}
    \[
E_{pq}^2(\cG,\varphi)=\H_p(\Gamma,\H_q(\cG\times_\varphi\Gamma))\Longrightarrow \H_{p+q}(\cG).
    \]
In this section, we show that the above spectral sequence is functorial for equivariant groupoid correspondences. 

\begin{definition}
    Let $\cG$ and $\cH$ be ample groupoids equipped with cocycles 
    $\varphi \colon \cG \to \Gamma$ and $\rho \colon \cH \to \Lambda$. Let $j \colon \Gamma \to \Lambda$ be a group homomorphism. 
    Let $\Omega \colon \cG \to \cH$ be an \'etale correspondence and let $\tau \colon \Omega \to \Lambda$ be a locally constant map. We say that the pair $(\Omega, \tau)$ is a \emph{($\Gamma$, $\Lambda$)-equivariant correspondence} if 
    \[ \tau(\alpha \omega) = j(\varphi(\alpha)) \tau(\omega) \mbox{ and } \tau(\omega \beta)=\tau(\omega)\rho(\beta)\]
    for all $\alpha \in \cG, \omega \in \Omega, \beta \in \cH$ such that they are composable appropriately. 
\end{definition}

We shall fix the following notation for the remainder of this section. Let $\cG$ and $\cH$ be ample $\sigma$-compact groupoids equipped with cocycles $\varphi \colon \cG \to \Gamma$ and $\rho \colon \cH \to \Lambda$, where $\Gamma$ and $\Lambda$ are discrete groups. Moreover, assume there is a group homomorphism $j \colon \Gamma \to \Lambda$.

We first define the functor $M \mapsto M \times_\varphi \Gamma$ from the category of left $\cG$-modules to $\cG \times_\varphi \Gamma$-modules. Let $M$ be a left $\cG$-module. 
Put $M \times_\varphi\Gamma \coloneqq M \otimes \Zz \Gamma$ as a $\Zz$-module. We define the action of $\Zz[\cG \times_\varphi \Gamma] = \Zz[\cG] \otimes \Zz\Gamma$ (this equality is just as $\Zz$-modules) by 
\[ (\xi \otimes 1_g)(m \otimes 1_h) = \xi|_{\varphi^{-1}(g^{-1}h)} m \otimes 1_g,\]
where $\xi \in \Zz\cG, m \in M$ and $g,h \in \Gamma$. Here, $1_g$ denotes the characteristic function on $\{g\}$. 

\begin{remark}
    We can see that the multiplication in $\Zz[\cG \times_\varphi \Gamma]$ is given by 
\[ (\xi \otimes 1_g)(\eta \otimes 1_h) = (\xi\vert_{\varphi^{-1}(g^{-1}h)}*\eta) \otimes 1_g\]
for $\xi, \eta \in \Zz[\cG]$ and $g,h \in \Gamma$.
\end{remark}

Let $M$ and $N$ be left $\cG$-modules and $\kappa \colon M \to N$ be a $\cG$-module map. Then, we can see that $\kappa \times_\varphi \Gamma \coloneqq \kappa \otimes \id \colon M \times_\varphi \Gamma \to N \times_\varphi \Gamma$ is a $\cG \times_\varphi \Gamma$-module map, so that $M \mapsto M \times_\varphi \Gamma$ is functorial. 
The functor $M \mapsto M \times_\rho \Gamma$ is exact because $\Zz \Gamma$ is $\Zz$-flat. 

We observe the following description for the canonical right $\cG \times_\varphi \Gamma$-module structure of $\Zz[\cG^{(0)} \times \Gamma]$. For $\xi \in \Zz[\cG], f \in \Zz[\cG^{(0)}], g,h \in \Gamma$, we have
\begin{equation}
 \label{eqn:rightmod}   
(f \otimes 1_g)(\xi \otimes 1_h) = 
\begin{cases}
\sum_{k \in \Gamma} f\xi|_{\varphi^{-1}(g^{-1}k)} \otimes 1_k & \mbox{ if } g=h, \\
0 & \mbox{ if } g \neq h.
\end{cases}
\end{equation}
Here, $f\xi|_{\varphi^{-1}(g^{-1}k)}$ denotes the element of $\Zz[\cG^{(0)}]$ resulting from the action of $\xi|_{\varphi^{-1}(g^{-1}k)}$ on $f$.

\begin{lemma} \label{lem:generator}
    Let $M$ be a left $\cG$-module. Then, the $\Zz$-module $\Zz[\cG^{(0)} \times \Gamma] \otimes_{\cG \times_\varphi \Gamma} (M \times_\varphi \Gamma)$ is spanned by 
    \[ \{(f \otimes 1_g) \otimes (m \otimes 1_g) \colon f \in \Zz[\cG^{(0)}], m \in M, g \in \Gamma\}.\]
    In addition, $\Zz[\cG^{(0)} \times \Gamma] \otimes_{\cG \times_\varphi \Gamma} (M \times_\varphi \Gamma)$ is a $\Gamma$-module with the $\Gamma$-action determined by 
    \[ k\cdot (f \otimes 1_g) \otimes (m \otimes 1_g) = (f \otimes 1_{kg}) \otimes (m \otimes 1_{kg})\]
    for $f \in \Zz[\cG^{(0)}], m \in M$ and $k,g \in \Gamma$. 
\end{lemma}
\begin{proof}
    Let $N = \Zz[\cG^{(0)} \times \Gamma] \otimes_{\cG \times_\varphi \Gamma} (M \times_\varphi \Gamma)$. By definition, $N$ is spanned by 
    \[ \{(f \otimes 1_g) \otimes (m \otimes 1_h) \colon f \in \Zz[\cG^{(0)}], m \in M, g,h \in \Gamma\}.\]
    For the first claim, it suffices to show that $(f \otimes 1_g) \otimes (m \otimes 1_h) = 0$ when $g \neq h$. To see this, use nondegeneracy to take an idempotent $e \in \Zz[\cG^{(0)}]$ such that $em=m$. Then, by Equation~\eqref{eqn:rightmod},
    \[ (f \otimes 1_g) \otimes (m \otimes 1_h)
    = (f \otimes 1_g) \otimes ((e \otimes 1_h)(m \otimes 1_h))
    = ((f \otimes 1_g)(e \otimes 1_h)) \otimes (m \otimes 1_h)
    =0, \]
    so that the first claim holds. To see the second claim, it suffices to show that the bilinear map 
    \[ (f \otimes 1_g, m \otimes 1_h) \mapsto (f \otimes 1_{kg}) \otimes (f \otimes 1_{kh})\]
    is balanced. This follows similarly.  
\end{proof}

Recall that a (possibly non-Hasudorff) topological space is said to be locally LCH if every point has a locally compact Hausdorff neighbourhood.
\begin{lemma} \label{lem:flat}
    Let $X$ be a totally disconnected locally LCH space equipped with a left $\cG$-action. 
    Suppose the action $\cG \acts X$ is basic and \'etale. Then, the action $\cG \times_\varphi \Gamma \acts X \times \Gamma$ is basic and \'etale, so that the left $\cG \times_\varphi \Gamma$-module $\Zz[X \times \Gamma] = \Zz[X] \times_\varphi \Gamma$ is flat by \cite[Proposition~2.8]{Miller}. 
\end{lemma}
\begin{proof}
    First, the action $\cG \times_\varphi \Gamma \acts X \times \Gamma$ is defined by 
    \[ (\gamma,g)(x,g\varphi(\gamma)) = (\gamma x, g),\]
    where $g \in \cG, x \in X, g \in \Gamma$ and $(\gamma,x)$ is a composable pair. 
    
    It is clear that the action $\cG \times_\varphi \Gamma \acts X \times \Gamma$ is free, and the anchor map $X \times \Gamma \to \cG^{(0)} \times \Gamma$ is a local homeomorphism. Hence, to show that action is basic, it suffices to show that the quotient map $q \colon  X \times \Gamma \to  (\cG \times_\varphi \Gamma) \backslash (X \times \Gamma)$ is a local homeomorphism. It is always open, so it suffices to show that $q$ is locally injective. Let $\bar{q} \colon X \to \cG \backslash X$ is the quotient map, which is a local homeomorphism by assumption. We have the following commutative diagram: 
    \[
    \begin{tikzcd}
        X \times \Gamma \ar[r,"q"] \ar[d] & \cG \times_\varphi \Gamma \backslash X \times \Gamma \ar[d]\\
        X \ar[r,"\bar{q}"] & \cG \backslash X\nospacepunct{.}
    \end{tikzcd}
    \]
    The projection $X \times \Gamma \to X$ and $\bar{q}$ are both local homeomorphisms, so that their composition is locally injective. Hence $q$ is locally injective. 
\end{proof}

Let $\cG$ and $\cH$ be ample groupoids equipped with cocycles $\varphi \colon \cG \to \Gamma$ and $\rho \colon \cH \to \Lambda$. 
Then, a ($\Gamma$, $\Lambda$)-equivariant \'etale correspondence $(\Omega,\tau) \colon \cG \to \cH$ induces the \'etale correspondence 
\[
\Omega \times_\tau \Gamma \colon \cG \times_\varphi \Gamma \to \cH \times_\rho \Lambda
\]
 in the following way: Put $\Omega \times_\tau \Gamma = \Omega \times \Gamma$ as a topological space. The left action is defined by 
 \[ (\alpha, g)(\omega, g\rho(\alpha)) = (\alpha \omega, g)\]
 for $\alpha \in \cG, \omega \in \Omega, g \in \Gamma$ such that $(\alpha,\omega)$ is a composable pair. Similarly, the right action is defined by 
 \[ (\omega,g)(\beta,j(g)\tau(\omega))=(\omega \beta, g)\]
 for $\beta \in \cH, \omega \in \Omega, g \in \Gamma$ such that $(\omega,\beta)$ is a composable pair. From the definition of equivariant correspondences, we see that $\Omega \times_\tau \Gamma$ is a $\cG$-$\cH$-bispace. The space $\Omega \times_\tau \Gamma$ is indeed an \'etale correspondence because of the following lemma: 
\begin{lemma}
    The right action $\Omega \times_\tau \Gamma \curvearrowleft \cH \times_\rho \Lambda$ is free, proper, and \'etale. If $\Omega$ is proper, then so is $\Omega \times_\tau \Gamma$. 
\end{lemma}
\begin{proof}
    By deifinition, the right action is free and the anchor map $\Omega \times_\tau \Gamma \to \cH^{(0)} \times \Lambda$ is a local homeomorphism. Moreover, there is a homeomorphism 
    \[ (\Omega \times_\tau \Gamma) / (\cH \times_\rho \Lambda) \to (\Omega / \cH) \times \Gamma,\ [\omega,g]_{\cH \times_\rho \Lambda} \mapsto ([\omega]_\cH,g)\]
    which makes the following diagram commute:
    \[
    \begin{tikzcd}
        & \Omega \times_\tau \Gamma \ar[ld,"q"] \ar[rd, "\bar{q} \times \id"]& \\
        (\Omega \times_\tau \Gamma) / (\cH \times_\rho \Lambda) \ar[rr,"\sim"] & &(\Omega / \cH) \times \Gamma\nospacepunct{,}
    \end{tikzcd}
    \]
    where $q$ and $\bar{q}$ are quotient maps. Hence, $(\Omega \times_\tau \Gamma) / (\cH \times_\rho \Lambda)$ is Hausdorff and $q$ is a local homeomorphism, which implies that the right action is proper. Finally, suppose $\Omega$ is proper, i.e., $f\colon \Omega/ \cH \to \cG^{(0)}$ is a proper map. Then, the map 
    \[(\Omega \times_\tau \Gamma) / (\cH \times_\rho \Lambda) \cong (\Omega / \cH) \times \Gamma \to \cG^{(0)} \times \Gamma\]
    is equal to $f \times \id$, which is a proper map. Hence $\Omega \times_\tau \Gamma$ is proper. 
\end{proof}

For a $(\Gamma, \Lambda)$-equivariant correspondence $(\Omega, \tau)$, let $\Omega_h \coloneqq \tau^{-1}(h)$ for $h \in \Lambda$. Note that $\Zz \Omega = \bigoplus_h \Zz \Omega_h$. 

\begin{lemma} \label{lem:natural}
    Let $M$ be a left $\cH$-module. Then, there is an isomorphism of $\cG \times_\varphi \Gamma$-modules
     \[ 
     (\Ind_\Omega M) \times_\varphi \Gamma \cong \Ind_{\Omega \times_\tau \Gamma} (M \times_\rho \Lambda),\quad  
     (f \otimes m) \otimes 1_g \mapsto (f \otimes 1_g) \otimes (m \otimes 1_{j(g)h})
     \]
     for all $f \in \Zz \Omega_h, m \in M, g \in \Gamma, h \in \Lambda$. 
\end{lemma}
\begin{proof}
    For $g \in \Gamma$, define a bilinear map
    \[ \Phi_g \colon \Zz \Omega \times M \to \Ind_{\Omega \times_\tau \Gamma} (M \times_\rho \Lambda),\ (f, m) \mapsto (f \otimes 1_g) \otimes (m \otimes 1_{j(g)h}) \]
    for $f \in \Zz \Omega_h, m \in M, h \in \Lambda$. Then, for $\xi \in \Zz \cH$ and $k \in \Lambda$ with $\supp \xi \subseteq \rho^{-1}(k)$, we have $f\xi \in \Zz \Omega_{hk}$, so that 
    \begin{align*}
        &\Phi_g(f\xi, m) = (f \xi \otimes 1_g) \otimes (m \otimes 1_{j(g)hk}) = ((f \otimes 1_g)(\xi \otimes 1_{j(g)h})) \otimes (m \otimes 1_{j(g)hk}), \mbox{ and} \\
        &\Phi_g (f, \xi m) = (f \otimes 1_g) \otimes (\xi m \otimes 1_{j(g)h})
    = (f \otimes 1_g) \otimes ((\xi \otimes 1_{j(g)h})( m \otimes 1_{j(g)hk})). 
    \end{align*}
    Hence, the map $\Phi_g \colon \Zz \Omega \otimes_\cH M \to \Ind_{\Omega \times_\tau \Gamma} (M \times_\rho \Lambda)$
    is well-defined. Now, let 
    \[
    \Phi  \colon (\Ind_\Omega M) \times_\varphi \Gamma \to \Ind_{\Omega \times_\tau \Gamma} (M \times_\rho \Lambda).
    \] be the direct sum of the maps induced from the maps $\Phi_g$ for all $g\in\Gamma$.
    A calculation shows that $\Phi$ is a $\cG \times_\varphi \Gamma$-module map. A similar argument shows that the bilinear map 
    \[
  (\Zz \Omega \otimes \Zz\Gamma) \times  (M \otimes \Zz \Lambda) \to (\Ind_\Omega M) \times_\varphi \Gamma, \quad (f \otimes 1_g, m \otimes 1_l) \mapsto \delta_{j(g)h,l} (f \otimes m) \otimes 1_g 
    \]
    for $f \in \Zz \Omega_h$ is balanced, which induces the inverse of $\Phi$. 
\end{proof}

\begin{definition}
    A left $\cG$-module $M$ is said to be \emph{$\Gamma$-graded} if there is a grading $M = \bigoplus_{g \in \Gamma} M_g$ as $\Zz$-modules and satisfy 
    $fM_h \subseteq M_{gh}$ 
    for $g,h \in \Gamma$ and $f \in \Zz[\cG]$ with $\supp f \subseteq \varphi^{-1}(g)$.  
\end{definition}

Let us mention an example class of graded modules that will play a role in our proof below. An equivariant correspondence $\Omega$ gives rise to a $\Lambda$-graded left $\cG$-module $M=\Zz[\Omega]$ with grading $M_h =\Zz[\Omega_h]$ with respect to the cocycle $j \circ \varphi \colon \cG \to \Lambda$. In particular, $\Zz[\cG^{(n)}]$ is a $\Gamma$-graded left $\cG$-module for $n \geq 1$, where $\cG^{(n)}$ denotes the space of composable $n$-tuples, since $\cG^{(n)}$ is a $\cG$-$\cG$ equivariant correspondence with $\tau(g_1,\dots,g_n)=\varphi(g_1\cdots g_n)$. On the other hand, $\Zz[\cG^{(0)}]$ does not have a natural $\Gamma$-grading.

Fix a projective resolution $P_\bullet \to \Zz[\cG^{(0)}]$ over $\Zz\cG$ and a projective resolution $Q_\bullet \to \Zz[\cH^{(0)}]$ over $\Zz\cH$, where $P_n$ is of the form $\Zz[X_n]$ and is $\Gamma$-graded, and $Q_n$ is of the form $\Zz[Y_n]$ and is $\Lambda$-graded. We can always choose such a resolution because the bar resolution is a projective resolution by \cite[Section~2]{BDGW} (note that we do not require that differentials preserve the grading). Then, $\Ind_\Omega Q_\bullet$ is a flat resolution of $\Ind_\Omega \cH^{(0)}$ by \cite[Proposition~2.8]{Miller}. Now suppose $\Omega$ is proper, 
and let $\kappa \colon \Zz[\cG^{(0)}] \to \Zz[\Omega/\cH] \cong \Ind_\Omega \Zz[\cH^{(0)}]$ be the map induced by the proper map $\Omega/\cH \to \cG^{(0)}$. By projectivity of $P_\bullet$, we can choose a lift $\tilde{\kappa}_\bullet \colon P_\bullet \to \Ind_\Omega Q_\bullet$. 
By applying the skew product functor (which is exact), we obtain the resolutions 
$P_\bullet \times_\varphi \Gamma \to \Zz[\cG^{(0)} \times \Gamma]$ and $\Ind_\Omega Q_\bullet \times_\varphi \Gamma \to \Ind_\Omega \Zz[\cH^{(0)}] \times_\varphi \Gamma \cong \Ind_{\Omega \times_\tau \Gamma} \Zz[ \cH^{(0)} \times_\rho \Lambda]$ by Lemma~\ref{lem:natural}, together with the chain map 
\[ 
\tilde{\kappa}_\bullet \times_\varphi \Gamma \colon P_\bullet \times_\varphi \Gamma 
 \to \Ind_\Omega Q_\bullet \times_\varphi \Gamma \cong \Ind_{\Omega \times \Gamma} (Q_\bullet \times_\rho \Lambda)
 \]
 which is a lift of $\kappa \times_\varphi \Gamma \colon \Zz[\cG^{(0)} \times_\varphi \Gamma] \to \Ind_{\Omega \times_\tau \Gamma} \Zz[ \cH^{(0)} \times_\rho \Lambda]$. 
 They are flat resolutions by Lemma~\ref{lem:flat}. Now we can apply \cite[Theorem~3.5]{Miller} to obtain the homomorphism 
 \[ \H_*(\Omega \times_\tau \Gamma) \colon \H_*(\cG \times_\varphi \Gamma) \to \H_*(\cH \times_\rho \Lambda). \]

In order to obtain the functoriality of the LHS type spectral sequence, let $C_\bullet = (P_\bullet \times_\varphi \Gamma)_{\cG \times_\varphi \Gamma}$ and $C_\bullet'=(Q_\bullet \times_\rho \Lambda)_{\cH \times_\rho \Lambda}$. 
By Lemma~\ref{lem:generator}, $C$ and $C'$ are complexes of $\Gamma$-modules, where $\Gamma$ acts diagonally. 
The map $\H_*(\Omega \times_\tau \Gamma)$ is induced by the 
chain map

\begin{equation}
\label{eqn:kappabar}
\begin{aligned}   
 \bar{\kappa} \colon (P_\bullet \times_\varphi \Gamma)_{\cG \times_\varphi \Gamma} \xrightarrow{(\tilde{\kappa} \times_\varphi \Gamma)_{\cG\times_\varphi\Gamma}} 
(\Ind_\Omega Q_\bullet \times_\varphi \Gamma)_{\cG \times_\varphi \Gamma} 
&\xrightarrow{\sim} (\Ind_{\Omega \times_\tau \Gamma} (Q_\bullet \times_\rho \Lambda))_{\cG \times_\varphi \Gamma}\\
&\xrightarrow{\delta \otimes \id} 
(Q_\bullet \times_\rho \Lambda)_{\cH \times_\rho \Lambda},  
\end{aligned}
\end{equation}
where
\[\delta = \delta_{\Omega \times_\tau \Gamma} \colon \Zz[\cG^{(0)} \times \Gamma] \otimes_{\cG \times_\varphi \Gamma} \Zz[\Omega \times_\tau \Gamma] \to \Zz[\cH^{(0)} \times \Lambda]
\]
is the map from \cite[Proposition~3.3]{Miller}. 
The next lemma implies that $\bar{\kappa}$ is ($\Gamma$, $\Lambda$)-equivariant:
\begin{lemma} \label{lem:delta-equiv}
    The canonical map $\delta$
    is ($\Gamma$, $\Lambda$)-equivariant. 
\end{lemma}
\begin{proof}
    By definition, we have $s_{\Omega \times_\tau \Gamma}(\omega, g) = (s_\Omega(\omega), j(g)\tau(\omega))$ for $(\omega, g) \in \Omega \times_\tau \Gamma$. 
    Hence, for $g\in \Gamma, h \in \Lambda$ and $\xi \in \Zz \Omega_h$, we have
    \begin{equation} 
        (s_{\Omega \times_\tau \Gamma})_*(\xi \otimes 1_g) = (s_\Omega)_*(\xi) \otimes 1_{j(g)h}.
    \end{equation} 
    By Lemma~\ref{lem:generator}, the $\Zz$-module $\Zz[\cG^{(0)} \times \Gamma] \otimes_{\cG \times_\varphi \Gamma} \Zz[\Omega \times_\tau \Gamma]$ is spanned by 
    \[ 
    \{(f \otimes 1_g) \otimes (\xi \otimes 1_g) \colon f \in \Zz[\cG^{(0)}], \xi \in \Zz \Omega_h, g \in \Gamma, h \in \Lambda\}.
    \] 
    By definition of $\delta$, we have 
    \begin{equation}
        \label{eqn:delta}
    \delta((f \otimes 1_{kg}) \otimes (\xi \otimes 1_{kg}))
    = (s_{\Omega \times_\tau \Gamma})_*((f|_{\varphi^{-1}(e)}\xi) \otimes 1_{kg})
    = (s_\Omega)_*(f\xi) \otimes 1_{j(k)j(g)h}
        \end{equation}

    for $k \in \Gamma$, since $f\xi \in \Zz \Omega_h$. Thus we have equivariance. 
\end{proof}

\begin{lemma} \label{lem:coinv}
    Let $M=\bigoplus_g M_g$ be a $\Gamma$-graded left $\cG$-module. Then, 
    the map
    \begin{align*}
        &\Phi \colon \Zz \Gamma \otimes (\Zz[\cG^{(0)}] \otimes_{\cG} M) \to \Zz[\cG^{(0)} \times \Gamma] \otimes_{\cG \times_\varphi \Gamma} (M \times_\varphi \Gamma), \\
        &1_g \otimes f \otimes m \mapsto (f \otimes 1_{gh^{-1}}) \otimes (m \otimes 1_{gh^{-1}})
    \end{align*}
    for $f\in \Zz[\cG^{(0)}]$, $g,h\in\Gamma$, and $m\in M_h$ is an isomorphism of $\Gamma$-modules. In particular, the $\Gamma$-module $\Zz[\cG^{(0)} \times \Gamma] \otimes_{\cG \times_\varphi \Gamma} (M \times_\varphi \Gamma)$ is $\Gamma$-acyclic and
    \[ ((M \times_\varphi \Gamma)_{\cG \times_\varphi \Gamma})_\Gamma \cong M_\cG. \]
\end{lemma}
\begin{proof}
    First, we show that $\Phi$ is well-defined. Fix $g \in \Gamma$. Then, the bilinear map 
\[\Phi_g \colon \Zz[\cG^{(0)}] \times M \to \Zz[\cG^{(0)} \times \Gamma] \otimes_{\cG \times_\varphi \Gamma} (M \times_\varphi \Gamma), \quad
        (f, m) \mapsto (f \otimes 1_{gh^{-1}}) \otimes (m \otimes 1_{gh^{-1}})
\]
    for $f\in \Zz[\cG^{(0)}]$, $h\in\Gamma$, and $m\in M_h$ is well-defined because $M = \bigoplus_h M_h$. 
    Let $h,k \in \Gamma$, $f \in \Zz[\cG^{(0)}]$, $m \in M_h$, and $\eta \in \Zz[\cG]$ with $\supp \eta \subseteq \varphi^{-1}(k)$. Then, by Equation~\eqref{eqn:rightmod},
    \begin{align*}
        &\Phi_g(f\eta, m) = (f\eta \otimes 1_{gh^{-1}}) \otimes (m \otimes 1_{gh^{-1}})
        = ((f \otimes 1_{gh^{-1}k^{-1}})(\eta \otimes 1_{gh^{-1}k^{-1}})) \otimes (m \otimes 1_{gh^{-1}}), \mbox{ and} \\
        &\Phi_g(f, \eta m) = 
        (f \otimes 1_{gh^{-1}k^{-1}}) \otimes (\eta m \otimes 1_{gh^{-1}k^{-1}})
        =(f \otimes 1_{gh^{-1}k^{-1}}) \otimes ((\eta \otimes 1_{gh^{-1}k^{-1}})(m \otimes 1_{gh^{-1}})), 
    \end{align*}
    which implies that $\Phi_g$ is balanced and thus induces a map from $\Zz \Gamma \otimes (\Zz[\cG^{(0)}] \otimes_{\cG} M)$. Hence, $\Phi$ is the direct sum of the maps induced from $\Phi_g$ for all $g\in\Gamma$. Exactly the same argument shows that the bilinear map 
\[\Zz[\cG^{(0)} \times \Gamma] \times (M \times_\varphi \Gamma) \to \Zz \Gamma \otimes (\Zz[\cG^{(0)}] \otimes_{\cG} M), \quad (f \otimes 1_g, m \otimes 1_{g'}) \mapsto \delta_{g,g'} 1_{gh} \otimes f \otimes m
\]
    for $f\in \Zz[\cG^{(0)}]$, $g,g',h\in\Gamma$, and $m \in M_h$
    is balanced, so that it induces a map from the tensor product $\Zz[\cG^{(0)} \times \Gamma] \otimes_{\cG \times_\varphi \Gamma} M \times_\varphi \Gamma$, which is the inverse of $\Phi$. The latter claim follows from \cite[Corollary~III.6.6]{Brown}.
\end{proof}

The isomorphism $M_\cG \to ((M \times_\varphi \Gamma)_{\cG \times_\varphi \Gamma})_\Gamma$ in Lemma~\ref{lem:coinv} is given by 
\[ f \otimes m \in \Zz[\cG^{(0)}] \otimes_\cG M \mapsto 
1 \otimes (f \otimes 1_e) \otimes (m \otimes 1_e) \in \Zz \otimes_\Gamma (\Zz[\cG^{(0)} \times \Gamma] \otimes_{\cG \times_\varphi \Gamma} (M \times_\varphi \Gamma)). 
\]
Hence, if $F \colon M \to N$ is a $\cG$-module map between $\Gamma$-graded left $\cG$-modules (we do not assume that $F$ preserves the gradings), then the following diagram is commutative: 
\[
\begin{tikzcd}
    M_\cG\ar[r,"\sim"] \ar[d, "F_\cG"] & ((M \times_\varphi \Gamma)_{\cG \times_\varphi \Gamma})_\Gamma \ar[d, "((F\times_\varphi \Gamma)_{\cG \times_\varphi \Gamma})_\Gamma"]\\
    N_\cG\ar[r,"\sim"] & ((N \times_\varphi \Gamma)_{\cG \times_\varphi \Gamma})_\Gamma\nospacepunct{.} 
\end{tikzcd}
\]
In particular, we have $(C_\bullet)_\Gamma \cong (P_\bullet)_\cG$ and $(C_\bullet')_\Gamma = (Q_\bullet)_\cH$. 

\begin{theorem} \label{thm:LHSfunctorial}
    An equivariant correspondence $(\Omega, \tau) \colon \cG \to \cH$ induces a morphism of LHS type spectral sequences $E(\cG,\varphi) \to E(\cH,\rho)$ such that
    \begin{enumerate}[\upshape(i)]
        \item  the map between $E^2$-pages
    \[ \H_p(\Gamma, \H_q(\cG \times_\varphi \Gamma)) \to \H_p(\Lambda, \H_q(\cH \times_\rho \Lambda))\]
    is induced from the $(\Gamma,\Lambda)$-equivariant map 
    \[ \H_*(\Omega \times_\tau \Gamma) \colon \H_*(\cG \times_\varphi \Gamma) \to \H_*(\cH \times_\rho \Lambda); \]
        \item the map between $E^\infty$-pages is $\H_{*}(\Omega) \colon \H_*(\cG) \to \H_*(\cH)$. 
    \end{enumerate}
\end{theorem}
\begin{proof}
    The LHS type spectral sequence is obtained by applying \cite[Proposition~VII.5.6]{Brown} to $C_\bullet$ and $C_\bullet'$. Since the map $\bar{\kappa}$ from \eqref{eqn:kappabar} is $(\Gamma,\Lambda)$-equivariant, we see from the proof of this proposition that $\bar{\kappa}$ induces a map between the spectral sequences in \cite[Proposition~VII.5.6]{Brown} for $C_\bullet$ and $C_\bullet'$. Moreover, the maps $\H_q(C_\bullet)\to \H_q(C_\bullet')$ between coefficient modules of the homology groups on the $E^2$-page are induced from $\bar{\kappa}$. Since $\H_*(\Omega \times_\tau \Gamma)$ is induced from $\bar{\kappa}$, this shows (i).
    
    It remains to show that the map $\H_*((C_\bullet)_\Gamma) \to \H_*((C'_\bullet)_\Lambda)$ induced by $\bar{\kappa}$ is identified with $\H_*(\Omega)$ via the natural transformation in Lemma~\ref{lem:coinv}. To see this, it suffices to show the map 
    \[ (\Ind_\Omega Q_\bullet)_\cG \to (Q_\bullet)_\cH\]
    induced by $\id_{\Zz} \otimes \delta \otimes \id_{Q_\bullet\times_\rho\Lambda}$ is equal to $\delta_\Omega\otimes\id_{Q_\bullet}$. It is the composition 
    \begin{align*}
        &(\Ind_\Omega Q_\bullet)_\cG  \xrightarrow{\mbox{Lemma~\ref{lem:coinv}}} 
        (((\Ind_\Omega Q_\bullet) \times_\varphi \Gamma)_{\cG \times_\varphi \Gamma})_\Gamma \xrightarrow{\mbox{Lemma~\ref{lem:natural}}} 
        ((\Ind_{\Omega\times_\tau\Gamma}(Q_\bullet \times_\rho \Lambda))_{\cG \times_\varphi \Gamma})_\Gamma \\
        &\xrightarrow{\id \otimes \delta \otimes \id} ((Q_\bullet \times_\rho \Lambda)_{\cH \times_\rho \Lambda})_\Lambda \xrightarrow{\mbox{Lemma~\ref{lem:coinv}}} 
        (Q_\bullet)_\cH. 
    \end{align*}
    For $f \in \Zz[\cG^{(0)}]$, $\xi \in \Zz \Omega_h$, $h \in \Lambda$, and $m \in Q_\bullet$, we have 
    \begin{align*}
        &f \otimes (\xi \otimes m) \mapsto 
        1 \otimes (f \otimes 1_e) \otimes ((\xi \otimes m) \otimes 1_e) \mapsto
        1 \otimes (f \otimes 1_e) \otimes ((\xi \otimes 1_e) \otimes (m \otimes 1_{h}) \\ 
        &\mapsto 1 \otimes ((s_\Omega)_*(f \xi) \otimes 1_{h}) \otimes (m \otimes 1_{h})
        \mapsto (s_\Omega)_*(f \xi) \otimes m, 
    \end{align*}
    where we have used Equation~\eqref{eqn:delta}. Hence, this map is equal to $\delta_\Omega\otimes\id_{Q_\bullet}$. 
\end{proof}

We conclude this section with an example setting where we shall utilize Theorem~\ref{thm:LHSfunctorial} and a technical lemma that will be used in our proof of Theorem~\ref{thm:BOShomology}.

\begin{example}[\'etale homomorphisms]
\label{ex:etalehom}
    Let $f \colon \cG \to \cH$ be an \'etale homomorphism. The correspondence $\Omega_f$ associated to $f$ is $\Omega_f = \cG^{(0)} \times_{\cH^{(0)}} \cH$, and the map $\H_*(\cG) \to \H_*(\cH)$ induced by $f$ coincides with $\H_*(\Omega_f)$, see \cite[Example~3.8]{Miller}. Suppose $f$ is equivariant in the sense that $j \circ \varphi = \rho \circ f$. Then, $\Omega_f$ is an equivariant correspondence via
    \[ \tau \colon \Omega_f \to \Lambda,\ (x,\beta) \mapsto \rho(\beta). \]
    Then, we can see that $\Omega_f \times_\tau \Gamma \cong \Omega_{f \times j}$. In particular, the map $\H_*(\Omega_f \times_\tau \Gamma) \colon \H_*(\cG \times_\varphi \Gamma) \to \H_*(\cH \times_\rho \Lambda)$, which gives the map between $E^2$-pages, is induced by the \'etale homomorphism $f \times j$. 
\end{example}

\begin{definition}
\label{def:equiv}
    Let $X$ and $Y$ be compact spaces and suppose we have partial actions $\Gamma\acts X$ and $\Lambda\acts Y$ (in the sense of, e.g., \cite[Definition~2.1]{Li:IMRN}). Let $j \colon \Gamma \to \Lambda$ and $p \colon Y \to X$. We say that $(p,j)$ is \emph{equivariant} if for all $y\in Y$ and $g\in\Gamma$, if $p(y)$ is contained in the domain of $g$, then $y$ is contained in the domain of $j(g)$ and $p(j(g)y)=gp(y)$. 
\end{definition}

\begin{lemma}
\label{lem:technical}
    Let $X$ and $Y$ be compact spaces and suppose we have partial actions $\Gamma\acts X$ and $\Lambda\acts Y$ (in the sense of, e.g., \cite[Definition~2.1]{Li:IMRN}). Assume there is a morphism between short exact sequences of groups
    \begin{equation}
        \begin{tikzcd}
            1\arrow[r] & \ker\bar{\varphi}\arrow[r]\arrow[d] & \Gamma \arrow[r,"\bar{\varphi}"]\arrow[d,"j"] &\Gamma''\arrow[d,"j''"] \arrow[r] &1 \\
            1\arrow[r]& \ker\bar{\rho}\arrow[r] & \Lambda \arrow[r,"\bar{\rho}"] &\Lambda'' \arrow[r] &1\nospacepunct{,}
        \end{tikzcd}
    \end{equation}
    and there is a continuous surjective map $p\colon Y\to X$ 
    such that $(p,j)$ is equivariant in the sense of Definition \ref{def:equiv}. 
Define cocycles
    \[ \varphi\colon \Gamma\ltimes X\to\Gamma \xrightarrow{\bar{\varphi}} \Gamma'' \mbox{ and } \rho\colon\Lambda\ltimes Y\to\Lambda \xrightarrow{\bar{\rho}} \Lambda''.\]
     Assume that there are subgroups $\Gamma' \leq \ker \bar{\varphi}$ and $\Lambda' \leq \ker\bar{\rho}$ such that the restricted actions $\Gamma'\acts X$ and $\Lambda'\acts Y$ are globally defined and 
     \[ \ker \bar{\varphi} \ltimes X = \Gamma'\ltimes X \quad \text{ and }\quad \ker \bar{\rho} \ltimes Y = \Lambda'\ltimes Y.\]
     Suppose we have $j(\Gamma') \subseteq \Lambda'$. 
     
    Let $\Omega = \Lambda \ltimes Y$ be the $(\Gamma'',\Lambda'')$-equivariant correspondence from $\Gamma\ltimes X$ to $\Lambda \ltimes Y$ with canonical right action by $\Lambda \ltimes Y$, left action given by
    $(g,p(hy))(h,y)=(j(g)h,y)$
    for $g \in \Gamma, h \in \Lambda, y \in Y$, and locally constant map given by $\tau=\rho\colon \Omega\to \Lambda''$. 
    Let $\Omega' = \Omega(j,p) \colon \Gamma' \ltimes X \to \Lambda' \ltimes Y$. 

    Then, the following diagram commutes: 
    \begin{equation}
        \begin{tikzcd}
            \H_*((\Gamma\ltimes X)\times_\varphi\Gamma'') \arrow[rr,"\H_*(\Omega\times_\tau\Gamma'')"] & &\H_*((\Lambda\ltimes Y)\times_\rho\Lambda'')\\
            \H_*(\Gamma'\ltimes X) \arrow[rr,"\H_*(\Omega')"]\arrow[u]&& \H_*(\Lambda'\ltimes Y)\nospacepunct{.}\arrow[u]
        \end{tikzcd}
    \end{equation}
\end{lemma}
\begin{proof}
    Let $\iota_1 \colon \Gamma' \ltimes X \to (\Gamma \ltimes X) \times_\varphi \Gamma''$ and $\iota_2 \colon \Lambda' \ltimes Y \to (\Lambda \ltimes Y) \times_\rho \Lambda''$ be the inclusion maps. Let $\Omega_{\iota_1}$ and $\Omega_{\iota_2}$ denote the correspondences associated to the \'etale maps $\iota_1$ and $\iota_2$, respectively. By \cite[Proposition~3.7]{Miller}, it suffices to show $\Omega_{\iota_2} \circ \Omega' = (\Omega \times_\tau \Gamma'') \circ \Omega_{\iota_1}$. We can directly see that the both sides are isomorphic to the correspondence $\Lambda \ltimes Y \colon \Gamma' \ltimes X \to (\Lambda \ltimes Y) \times_\rho \Lambda''$. 
\end{proof}

We shall use the following continuity property for groupoid homology:
\begin{proposition}
\label{prop:continuity}
Let $X=\varprojlim_{k\geq 1} \{X_k;\pi_{k,k-1}\}$ be the projective limit of compact totally disconnected spaces $X_k$, and suppose we have discrete groups $\Gamma_k$ for $k\geq 1$ together with injective group homomorphisms $j_{k,k+1}\colon\Gamma_k\to\Gamma_{k+1}$ such that $\{\Gamma_k;j_{k,k+1}\}$ is a directed system. Let $\Gamma\coloneqq\varinjlim_k\{\Gamma_k; j_{k,k+1}\}$. Moreover, assume that for each $k$ we have a partial action $\Gamma_k\acts X_k$ and $\Gamma\acts X$ a partial action such that the pairs of maps $(\pi_{k+1,k}, j_{k,k+1})$ is equivariant in the sense of Definition~\ref{def:equiv}. Let $\Omega_k \colon \Gamma_k\ltimes X_k \to \Gamma_{k+1}\ltimes X_{k+1}$ be the correspondence from Lemma~\ref{lem:technical}. Then, there is a canonical isomorphism
\[
\H_*(\Gamma\ltimes X)=\varinjlim_k\{\H_*(\Gamma_k\ltimes X_k);\H_*(\Omega_k)\}.
\]
\end{proposition}
\begin{proof}
    Let $\cH_k := \Gamma_k\ltimes X_k$ and 
    let $\Phi_k \colon \Zz \cH_k \to \Zz\cH_{k+1}$ be the $\Zz$-module map defined by 
    $\xi \mapsto \xi1_{X_{k+1}}$. Here, the product is defined via the left action $\Zz \cH_k \acts \Zz\cH_{k+1} = \Zz \Omega_k$. Concretely, 
    for $\xi \in \Zz \cH_k$ and $(h,y) \in \cH_{k+1}$, we have 
    \[ (\Phi_k(\xi))(h,y) = \begin{cases}
        \xi(g,\pi_{k+1,k}(y)) & \mbox{if $h=j(g)$ for some $g \in \Gamma_{k}$}, \\
        0 & \mbox{otherwise}.
    \end{cases}\]
    Here, we used equivariance in the sense of Definition~\ref{def:equiv}. 
    Let $C_\bullet^k = \Zz[\cH_k^\bullet] = \Zz[\cH_k] \otimes_{X_k}\dots \otimes_{X_k} \Zz[\cH_k]$  be the chain complex obtained by taking coinvariants of the bar resolution. Then, we have $\H_*(\cH_k)=\H_*(C_\bullet^k)$. We first claim that 
    $\H_*(\Omega_k) \colon \H_*(\cH_k) \to \H_*(\cH_{k+1})$ is induced from the chain map
    \[ \Phi_k^{\otimes n} =\Phi_k \otimes \dots \otimes \Phi_k \colon \Zz[\cH_k^n] \to \Zz[\cH_{k+1}^n]\]
    for each $n$. Fix $k \in \Nz$ and let $j=j_{k,k+1}$, $\pi=\pi_{k+1,k}$. 
    Let $\cK=\cH_k \ltimes X_{k+1}$ be the action groupoid, 
    where the action $\cH_k \acts X_{k+1}$ is defined by 
    \[ (g,\pi(y))\cdot y = j(g)y\]
    for $y \in X_{k+1}$ and $g \in \Gamma_k$. 
    Note that this groupoid action is well defined because of the equivariance condition from Definition~\ref{def:equiv}. We identify $\cK \cong \Gamma_k \ltimes X_{k+1}$,    
    where the partial action $\Gamma_k \acts X_{k+1}$ is defined using $j$, via the isomorphism given by
    \[ \cH_k \ltimes X_{k+1} \to \Gamma_k \ltimes X_{k+1},\quad 
    (g,\pi(y),y) \mapsto (g,y).\]
    Let $\varphi \colon \cK \to \cH_{k+1}$ be the \'etale homomorphism defined by 
    $(g,y) \mapsto (j(g),y)$ for $(g,y) \in \Gamma_k \ltimes X_{k+1}$. 
    Let $\Omega_\varphi$ be the \'etale correspondence associated to $\varphi$. Then, we can see that $\Omega_\varphi \circ \cK = \Omega_k$, where $\cK$ is considered as an \'etale correspondence $\cH_k \to \cK$. By \cite[Example~3.9]{Miller}, $\H_*(\cK)$ is induced from $\tau_n^* \colon \Zz[\cH_k^n] \to \Zz[\cK^n]$, where 
    \[\tau_n \colon \cK^n \to \cH_k^n,\quad( (g_1,j(g_2\cdots g_n) y), \dots, (g_n,y)) \mapsto 
    ((g_1,g_2\cdots g_n \pi(y)), \dots, (g_n,\pi(y))). \]
    On the other hand, by \cite[Example~3.8]{Miller}, $H_*(\Omega_\varphi)$ is induced from 
    $\varphi_*^n \colon \Zz[\cK^n] \to \Zz[\cH_{k+1}^{n}]$ defined by 
    \[\varphi_*^n(f)((h_1,h_2\cdots h_n y), \dots, (h_n,y)) = \begin{cases}
        f((h_1,h_2\cdots h_n y), \dots, (h_n,y)) & \mbox{if $h_i \in j(\Gamma_k)$ for all $i$}, \\
        0 & \mbox{otherwise}.
    \end{cases}\]
    Then, we can see that $\phi_*^n \circ \tau_n^* = \Phi_k^{\otimes n}$, so that $\H_*(\Omega_k)$ is induced from $\Phi_k^{\otimes \bullet}$. Note that $\Phi_k$ is multiplicative, and we can directly see that $\Phi_k^{\otimes \bullet}$ is indeed a well-defined chain map by using this property. 

    Next, we claim that $\Zz\cH=\varinjlim (\Zz\cH_k, \Phi_k)$, where the inductive limit is taken in the category of abelian groups. To see this, it suffices to show that $\Zz\cH$ is spanned by $\{\eta1_\cH : \eta\in \Zz\cH_k\}$. Clearly $\Zz\cH$ is spanned by the functions supported on $\{j_k(g)\} \times X$ for some $k$ and $g \in \Gamma_k$, where $j_k\colon\Gamma_k\to\Gamma$ is the canonical map. Let $\xi \in \Zz[\cH]$ be such a function. Then, there exist $k\in \Nz$, $g \in \Gamma_k$, and $\tilde{\xi} \in \Zz[X]$ such that 
    \[ \xi(h,x) = \begin{cases}
    \tilde{\xi}(x) & \mbox{ if } h = j_k(g), \\
    0 & \mbox{ if } h \neq j_k(g). 
    \end{cases}\]
    Then, there exist $k' \in \Nz$ and $\tilde{\eta} \in X_{k'}$ such that $\tilde{\xi} = \tilde{\eta} \circ \pi_k$. We may assume that $k=k'$.
    Let $\eta \coloneqq 1_{g} \times \tilde{\eta} \in \Zz \cH_k$, 
    and then we have $\xi = \eta 1_\cH$. 

    Let $C_\bullet = \Zz[\cH^\bullet]$. Then, we have $C_\bullet = \varinjlim (C_\bullet^k, \Phi_k^{\otimes \bullet})$ because tensor products preserve inductive limits. Therefore, continuity for homology of chain complexes yields the desired property.

\end{proof}

\section{Groupoid homology for algebraic actions} 
\label{sec:gpdhom}

\subsection{Groupoids from algebraic actions}
\label{sec:algactions}
In this subsection, we recall some ideas from \cite{BruceLi2}. Let $S$ be a nontrivial cancellative reversible monoid. Recall that $S$ is said to be left (resp. right) reversible if $sS\cap tS\neq\emptyset$ (resp. $Ss\cap St\neq\emptyset$) for all $s,t\in S$, and $S$ is reversible if it is both left and right reversible.

Let $A$ a torsion-free abelian group of finite rank $d\in\Zz_{>0}$ (i.e., $A$ is isomorphic to a subgroup of $\Qz^d$ and $A\otimes_\Zz\Qz\cong \Qz^d$). Let $\sigma\colon S\acts A$ be a faithful algebraic action in the sense of \cite[Definition~2.1]{BruceLi2}, i.e., $\sigma\colon S\to \End_\Zz(A)$ is an injective (unital) monoid homomorphism such that $\sigma_s$ is injective for all $s$.

Since $S$ is right reversible, we can form the group of quotients $\S\coloneqq S^{-1}S$, which is unique up to canonical isomorphism. Moreover, since $S$ is reversible, it is directed with respect to the partial orders defined by $s\mid_r t$ if $t\in Ss$ and $s\mid_l t$ if $t\in sS$ (see \cite[Section~2.3]{CEL} for a discussion of this). 

\begin{remark}
    The partial order $\mid_r$ is used when one take inductive limits over $S$ and the partial order $\mid_l$ is used when one takes projective limits over $S$.
\end{remark}

The action of $S$ on $A$ generates a semilattice of subgroups: Let
\begin{equation*}
    \cC\coloneqq\{\sigma_{s_1}^{-1}\sigma_{t_1}\cdots\sigma_{s_m}^{-1}\sigma_{t_m}A : s_i,t_i\in S, m\in\Zz_{>0}\},
\end{equation*}
where, given $s\in S$ and a subset $X\subseteq A$, we write $\sigma_s^{-1}X\coloneqq\{y\in A : \sigma_s(y)\in X\}$ for the set-theoretic inverse image of $X$ under $\sigma_s$.
Members of $\cC$ are called \emph{$S$-constructible subgroups}, see \cite[Definition~3.6]{BruceLi2}. By \cite[Proposition~3.9]{BruceLi2}, $\cC$ is a closed under taking finite intersections. 

\begin{example}
\label{ex:RxactsR}
Let $R$ be the ring of integers in a number field $K$, and let $R^\times\coloneqq R\setminus\{0\}$ be the multiplicative monoid of $R$. For the multiplication action $R^\times\acts R$, the semilattice $\cC$ is the family of all nonzero ideals of $R$.
\end{example}

We shall assume that our action $\sigma\colon S\acts A$ is exact in the sense of \cite[Definition~4.11]{BruceLi2}, i.e., $\bigcap_{C\in\cC}C=\{0\}$. Since $A$ is torsion-free and of finite rank, it satisfies the finite index property from \cite[Definition~7.1]{BruceLi2}, i.e., $[A:\sigma_sA]<\infty$ for all $s\in S$ (see, e.g., \cite[Exercise~92.5]{Fuchs2}). By \cite[Proposition~7.2]{BruceLi2}, this implies that $[A:C]<\infty$ for all $C\in\cC$, so that $\ol{A}\coloneqq\varprojlim_{C\in\cC}A/C$ is a compact abelian group. Since $S$ is left reversible, the family $\{\sigma_sA : s\in S\}$ is  cofinal in $\cC$, so that $\ol{A}\cong \varprojlim_{s\in S}A/\sigma_sA$. The canonical map $A\to\ol{A}$ is injective by exactness of $\sigma\colon S\acts A$, and we shall identify $A$ with its image in $\ol{A}$. 
Each $\sigma_s$ uniquely extends to a continuous injective group endomorphism $\ol{A}\to\ol{A}$, which we also denote by $\sigma_s$.

Let $\ol{\A}\coloneqq\varinjlim_{s\in S}\{\ol{A};\sigma_s\}$. Then, $\ol{\A}$ is a locally compact abelian group. For each $s\in S$, let $\iota_s\colon \ol{A}\to\ol{\A}$ be the canonical (continuous, injective) map associated with the direct limit, so that $\ol{\A}=\bigcup_{s\in S}\iota_s(\ol{A})$. The maps $\iota_s$ satisfy the following: If $s\mid_r t$, i.e., $t=rs$ for some $r\in S$, then the following diagram commutes:
\begin{equation*}
    \begin{tikzcd}
       \ol{A}\arrow{dr}{\iota_s}\arrow{rr}{\sigma_r} &  &\ol{A}\arrow{dl}{\iota_t} \\
       & \ol{\A}\nospacepunct{.} &
    \end{tikzcd}
\end{equation*}
It is easy to see that $\iota_1(\ol{A})$ is a compact open subgroup of $\ol{\A}$. We shall identify $\ol{A}$ with $\iota_1(\ol{A})$.

Define $\A\coloneqq\varinjlim_{s\in S}\{A;\sigma_s\}$. There is a canonical injective homomorphism $\A\to \ol{\A}$ with dense image equal to $\bigcup_{s\in S}\iota_s(A)$. 
We shall identify $\A$ with its dense image in $\ol{\A}$. By universality of inductive limits, the action $\sigma\colon S\acts\ol{A}$ extends to an action $S\acts\ol{\A}$ by automorphisms. Since $\S$ is the group of quotients of $S$, this gives rise to an action $\tilde{\sigma}\colon\S\acts\ol{\A}$ by automorphisms. This action leaves $\A$ invariant, and restricts to an action by automorphisms $\S\acts\A$, which we also denote by $\tilde{\sigma}$. Note that $\tilde{\sigma}\colon \S\acts\A$ is a globalization of $\sigma\colon S\acts A$ in the sense of \cite[Definition~2.3]{BruceLi2} and $\A=\bigcup_{s\in S}\tilde{\sigma}_s^{-1}(A)$.

Let $A\rtimes S$ be the semi-direct product monoid associated with the action $\sigma\colon S\acts A$. Then, there is a canonical action $A\rtimes S\acts \ol{A}$, which extends to a partial action $\A\rtimes\S\acts\ol{A}$ satisfying $(a,s).x=a+\sigma_s(x)$ for all $x\in A$ and $(a,s)\in A\rtimes S\subseteq \A\rtimes\S$ (cf. \cite[Section~3]{BruceLi2}). Since $A\rtimes S$ generates the group $\A\rtimes\S$, the partial action $\A\rtimes\S\acts\ol{A}$ is characterized by the action of the submonoid $A\rtimes S$.

\begin{definition}[cf. {\cite[Definition~3.29]{BruceLi2}}]
    The groupoid attached to the algebraic action $\sigma\colon S\acts A$ is the partial transformation groupoid $\cG_{S\acts A}\coloneqq(\A\rtimes\S)\ltimes\ol{A}$.
\end{definition}

\begin{example}
Let $S=R^\times\acts A=R$ be the algebraic action from Example~\ref{ex:RxactsR}. Then, $\ol{A}$ is simply the profinite completion of $R$, which can be canonically identified with the compact ring $\ol{R}$ of integral adelels over $K$, and $\ol{\A}$ can be identified with $\Az_{K,f}$, the locally compact ring of finite adeles over $K$. Further, $\S=K^*$ is the multiplicative group of the field $K$, and $\A=K$ (viewed as an additive group). The action $K\rtimes K^*\acts \ol{R}$ is simply the restriction of the canonical action $K\rtimes K^*\acts \Az_{K,f}$, and the groupoid $\cG_{R^\times\acts R}$ is the groupoid model for the ring C*-algebra of $R$ (cf. \cite[Section~7.2.2]{BruceLi2}).
\end{example}

Before we turn to groupoid homology, we need to make some observations.
\begin{remark}
\label{rmk:full}
Since $\ol{\A}=\bigcup_{s\in S}\iota_s(\ol{A})$, the compact open subset $\ol{A}\subseteq \ol{\A}$ is full for the groupoid $(\A\rtimes\S)\ltimes\ol{\A}$, i.e., $(\A\rtimes\S).\ol{A}=\ol{\A}$. 
\end{remark}

We now turn to fullness of $\ol{A}$ viewed as a subset of (the unit space of) the groupoid $\A\ltimes\ol{\A}$.

\begin{lemma}
\label{lem:scrAcapbarA=A}
We have $\A\cap\ol{A}=A$ (intersection taken inside $\ol{\A}$).
\end{lemma}
\begin{proof}
Clearly, we have $A\subseteq \A\cap\ol{A}$. Let $a\in \A\cap\ol{A}$. Then, there exists $r\in S$ and $b\in A$ such that $a=\tilde{\sigma}_r^{-1}(b)$. Hence, $b=\sigma_r(a)\in\sigma_r\ol{A}=\overline{\sigma_rA}$ (where $\overline{\sigma_rA}$ denotes the closure of $\sigma_rA$). Thus, $b\in A\cap \overline{\sigma_rA}=\sigma_rA$. Now we can find $c\in A$ such that $b=\sigma_r(c)$, which gives $a=c\in A$.
\end{proof}

\begin{lemma}
\label{lem:scrA+barA=barscrA}
We have $\A+\ol{A}=\ol{\A}$ (sum taken inside $\ol{\A}$).
\end{lemma}
\begin{proof}
Consider the inclusion $\A/A\to\ol{\A}/\ol{A}$ induced from the inclusion $\A\to\ol{\A}$. It has dense image because $\A$ is dense in $\ol{\A}$. Since $\ol{A}$ is compact open in $\ol{\A}$, the quotient $\ol{\A}/\ol{A}$ is discrete. Hence, $\A/A\to\ol{\A}/\ol{A}$ is surjective; this map injective by Lemma~\ref{lem:scrAcapbarA=A} and is thus an isomorphism. The claim now follows.
\end{proof}

\begin{proposition}
\label{prop:equivAdditive}
The inclusion $A\ltimes\ol{A}\hookrightarrow\A\ltimes\ol{\A}$ is an equivalence of groupoids. Consequently, the canonical map $\H_*(\incl,\incl)\colon \H_*(A,\Zz\ol{A})\to \H_*(\A,\Zz\ol{\A})$ is an isomophism.
\end{proposition}
\begin{proof}
By Lemma~\ref{lem:scrA+barA=barscrA}, $\ol{A}$ is a full subset of the unit space of $\A\ltimes\ol{\A}$. Thus, it suffices to show that $A\ltimes\ol{A}=(\A\ltimes\ol{\A})_{\ol{A}}^{\ol{A}}$, which follows from Lemma~\ref{lem:scrAcapbarA=A}. The second claim follows from Proposition~\ref{prop:isomfunctors} combined with \cite[Proposition~3.5]{Mat12} (see also \cite[Lemma~4.3.]{FKPS}).
\end{proof}

\subsection{Homology computations}

We want to compute the groupoid homology $\H_*(\cG_{S\acts A})$, which is canonically identified with the group homology $\H_*(\A\rtimes \S,\Zz\ol{\A})$ (see Proposition~\ref{prop:isomfunctors}). The general strategy is to compute $\H_q(\A,\Zz\ol{\A})$ and $\H_p(\S,\H_q(\A,\Zz\ol{\A}))$ and then use the Lyndon--Hochschild--Serre spectral sequence (see \cite[Theorem~VII.6.3]{Brown}). Here, the $\S$-module structure on $\H_q(\A,\Zz\ol{\A})$ is defined as follows (see \cite[Section~III.8]{Brown}): for $s\in \S$ let $\alpha_s\colon \Zz\ol{\A}\to\Zz\ol{\A}$ by $\alpha_s(f)\coloneqq f\circ \tilde{\sigma}_s^{-1}$; then, 
\begin{equation*}
    \S\ni s\mapsto \H_*(\tilde{\sigma}_s,\alpha_s)\in\Aut_\Zz(\H_q(\A,\Zz\ol{\A})).
\end{equation*}

Let $\Iz\coloneqq\{[A:C] : C\in\cC\}$, and let $\Zz[\Iz^{-1}]$ be the unital subring of $\Qz$ generated by $\{\frac{1}{k} : k\mid n\text{ for some }n\in\Iz\}$. Note that since $S$ is reversible, it suffices to invert the indices $[A:\sigma_sA]$ for $s\in S$. 
By \cite[Lemma~3.21]{BruceLi3}, the canonical map $A\otimes \Zz[\Iz^{-1}]\to \A\otimes \Zz[\Iz^{-1}]$ is an isomorphism.

The main result of this subsection is the following.
\begin{theorem}
\label{thm1}
For every $q\geq 0$, there is a canonical embedding of $\Upsilon \colon \H_q(\A,\Zz\ol{\A}) \to (\lwedge^qA)\otimes\Zz[\Iz^{-1}]$ such that 
\begin{equation*}
\lwedge^qA\leq \Upsilon(\H_q(\A,\Zz\ol{\A}))\leq (\lwedge^qA)\otimes\Zz[\Iz^{-1}].
\end{equation*}
In particular, we have a canonical isomorphism 
\[ \Upsilon \otimes \id_{\Zz[\Iz^{-1}]} \colon \H_q(\A,\Zz\ol{\A}) \otimes \Zz[\Iz^{-1}]\xrightarrow{\cong} (\lwedge^qA)\otimes\Zz[\Iz^{-1}].\]
Under this identification, the action $\theta\colon \S\acts \H_q(\A,\Zz\ol{\A})$ is characterized by the property that 
\begin{equation*}
(\Upsilon \otimes \id_{\Zz[\Iz^{-1}]})\circ (\theta_s\otimes\id_{\Zz[\Iz^{-1}]})=\left( \lwedge^q(\sigma_s)\otimes\frac{1}{[A:\sigma_sA]}\right) \circ (\Upsilon \otimes \id_{\Zz[\Iz^{-1}]})\quad\text{ for all }s\in S.
\end{equation*}
\end{theorem}
Note that $\lwedge^qA=0$ for $q>d$.

For each $s\in S$, let $M_s := \Zz[A/\sigma_sA]$. The projection map $\pi_s\colon \ol{A}\to A/\sigma_sA$ induces an inclusion map $\pi_s^*\colon M_s\to\Zz \ol{A}$ by $\pi_s^*(f)=f\circ\pi_s$. We shall often identify $M_s$ with its image in $\Zz \ol{A}$.
For $s,t\in S$ with $s\mid_l t$, let $\pi_{s,t}\colon A/\sigma_tA\to A/\sigma_sA$ be the canonical projection map, and let $\pi_{s,t}^*\colon M_s\to M_t$ by $\pi_{s,t}^*(f)\coloneqq f\circ\pi_{s,t}$. Then, 
\begin{equation}
 \label{eqn:lim1}   
\H_*(A,\Zz\ol{A})\cong \varinjlim_{s\in S}\{\H_*(A,M_s);\H_*(\id_A,\pi_{s,t}^*)\}.
\end{equation}
For $s\in S$, the canonical map $\H_*(A,M_s)\to \H_*(A,\Zz\ol{A})$ associated with the direct limit decomposition in \eqref{eqn:lim1} is given by $\H_*(\id_A,\pi_s^*)$. In order to understand the limit in \eqref{eqn:lim1} better, we shall use Shapiro's lemma as is done in a more general setting in \cite{Scar}.

For each $s\in S$, let $\ind_s\colon \Zz\to \Zz\ol{A}$ by $\ind_s(1)=1_{\sigma_s\ol{A}}$. Note that the image of $\ind_s$ is contained in $M_s$. 
Given $s,t\in S$ with $s\mid_l t$, we also let $\rho_{s,t}\colon \H_*(A)\to \H_*(A)$ be the unique group homomorphism such that the following diagram commutes:
\begin{equation}
\label{eqn:rhodef}
    \begin{tikzcd}
    \H_*(A)\arrow[rr,"\rho_{s,t}"] \arrow{d}[left]{\cong}[right]{\H_*(\sigma_s \vert^{\sigma_sA},\id)} & &\H_*(A)\arrow{d}[left]{\cong}[right]{\H_*(\sigma_t\vert^{\sigma_tA},\id)}\\
        \H_*(\sigma_sA) \arrow{d}{\H_*(\incl,\ind_s\vert^{M_s})} & &\H_*(\sigma_tA)\arrow{d}{\H_*(\incl,\ind_t\vert^{M_t})} \\
        \H_*(A,M_s)\arrow{rr}{\H_*(\id,\pi_{s,t}^*)} & & \H_*(A,M_t).
    \end{tikzcd}
\end{equation}
Here, the vertical maps in the bottom half are the isomorphisms from Shapiro's lemma (see, e.g., \cite[Proposition~III.6.2~and~Exercise~2~of~Section~III.8]{Brown}).
Then, there is a canonical isomorphism
\begin{equation}
\label{eqn:isomoflimits}
    \H_*(A,\Zz\ol{A})\cong\varinjlim_{s\in S}\{\H_*(A);\rho_{s,t}\}.
\end{equation}

For each $s\in S$, the canonical map $\rho_{s,\infty}\colon \H_*(A)\to \H_*(A,\Zz\ol{A})$ arising from the direct limit decomposition in \eqref{eqn:isomoflimits} is given by 
\begin{equation}
\label{eqn:rhoinfty}
    \rho_{s,\infty}=\H_*(\id,\pi_s^*)\circ\H_*(\incl,\ind_s\vert^{M_s})\circ\H_*(\sigma_s\vert^{\sigma_sA},\id)=\H_*(\sigma_s,\ind_s).
\end{equation}

\begin{lemma}
\label{lem:rho}
For all $s\in S$, we have $\H_*(\sigma_s,\id)\circ\rho_{1,s}=[A:\sigma_sA]\id_{\H_*(A)}$. 
\end{lemma}
\begin{proof}
Let $\res^{A}_{\sigma_sA}\colon\H_*(A)\to\H_*(\sigma_sA)$ be the transfer map (see, e.g., \cite[Section~III.9]{Brown}). By \eqref{eqn:rhodef}, we have 
\begin{equation*}   \H_*(\incl,\ind_s\vert^{M_s})\circ\H_*(\sigma_s\vert^{\sigma_sA},\id)\circ\rho_{1,s}=\H_*(\id,\pi_{1,s}^*)=\H_*(\incl,\ind_s\vert^{M_s})\circ\res^{A}_{\sigma_sA},
\end{equation*}
where the second equality is from the proof of \cite[Proposition~2.4]{Scar}. Since $\H_*(\incl,\ind_s\vert^{M_s})$ is invertible, the above equation implies that 
\begin{equation}
\label{eqn:rho}
  \H_*(\sigma_s\vert^{\sigma_sA},\id)\circ\rho_{1,s}=\res^{A}_{\sigma_sA}. 
\end{equation}

We let $\ccor_{\sigma_sA}^{A}\colon\H_*(\sigma_sA)\to\H_*(A)$ denote the corestriction map, i.e., the map induced from the inclusion $\sigma_sA\to A$. We have $\ccor_{\sigma_sA}^{A}\circ \H_*(\sigma_s\vert^{\sigma_sA},\id)=\H_*(\sigma_s,\id)$, so applying $\ccor_{\sigma_sA}^{A}$ to both sides of \eqref{eqn:rho} and using \cite[Proposition~III.9.5(ii)]{Brown} gives us $\H_*(\sigma_s,\id)\circ\rho_{1,s}=[A:\sigma_sA]\id_{\H_*(A)}$.
\end{proof}

\begin{remark}
For the case $S=\Nz$, Lemma~\ref{lem:rho} is a homology analogue of \cite[Lemma~3.2]{CV}. In addition, up to this point, we have not used that $A$ is torsion-free in any essential way.
\end{remark}

For each $s\in S$, put $\tilde{\rho}_{s,\infty}\coloneqq\rho_{s,\infty}\otimes\id_{\Zz[\Iz^{-1}]}$, and for $s,t\in S$ with $s\mid_l t$, define $\tilde{\rho}_{s,t}\coloneqq\rho_{s,t}\otimes\id_{\Zz[\Iz^{-1}]}$.
 
\begin{proposition}
\label{prop:invertingrho}
 The map $\tilde{\rho}_{1,\infty}\colon \H_*(A)\otimes\Zz[\Iz^{-1}] \to\H_*(A,\Zz\ol{A}) \otimes\Zz[\Iz^{-1}]$ is an isomorphism.
\end{proposition}
\begin{proof}
Since tensor products commute with direct limits, we have a canonical isomorphism
\begin{equation*}
    \H_*(A,\Zz\ol{A}) \otimes\Zz[\Iz^{-1}]\cong\varinjlim_{s\in S}\left\{\H_*(A)\otimes\Zz[\Iz^{-1}];\tilde{\rho}_{s,t}\right\}.
\end{equation*}
By Lemma~\ref{lem:rho}, $\H_*(\sigma_s, \id)\otimes\id_{\Zz[\Iz^{-1}]}$ and $\tilde{\rho}_{1,s}$ are invertible with 
\begin{equation} \label{eqn:invertrho}
    \tilde{\rho}_{1,s}^{-1}=\frac{1}{[A:\sigma_sA]}\H_*(\sigma_s, \id)\otimes\id_{\Zz[\Iz^{-1}]}
\end{equation}
 for every $s\in S$. Hence, every connecting map $\tilde{\rho}_{s,t}$ is invertible, so that $\tilde{\rho}_{1,\infty}$ must be invertible.
\end{proof}

\begin{proof}[Proof of Theorem~\ref{thm1}]
Let $\iota\coloneqq \H_*(\incl,\incl)\colon \H_*(A,\Zz\ol{A})\to \H_*(\A,\Zz\ol{\A})$ be the the isomorphism from Proposition~\ref{prop:equivAdditive}.
The canonical inclusion $\H_*(A)\hookrightarrow \H_*(\A,\Zz\ol{\A})$ is given by the composition
\begin{equation*}
    \lwedge^*(A) \cong \H_*(A)\xrightarrow{\rho_{1,\infty}}\H_*(A,\Zz\ol{A})\xrightarrow{\iota} \H_*(\A,\Zz\ol{\A}),
\end{equation*}
where the first isomorphism is from \cite[Theorem~V.6.4(i)]{Brown}.

Using Proposition~\ref{prop:invertingrho}, we obtain a canonical inclusion
\begin{align*}
  \Upsilon \colon \H_*(\A,\Zz\ol{\A})\xrightarrow{\iota^{-1}}\H_*(A,\Zz\ol{A})&  \xrightarrow{\incl}\H_*(A,\Zz\ol{A}) \otimes\Zz[\Iz^{-1}]
  \xrightarrow{\tilde{\rho}_{1,\infty}^{-1}} \H_*(A) \otimes\Zz[\Iz^{-1}]\cong \lwedge^*(A) \otimes\Zz[\Iz^{-1}].
\end{align*}

It remains to compute the $S$-action on $\H_*(\A,\Zz\ol{\A})$ given by $\theta_s = \H_*(\tilde{\sigma}_s,\alpha_s)$ for $s \in S$.
For each $s\in S$, we have the following commutative diagram:
\begin{equation*}
    \begin{tikzcd}
    \H_*(\A,\Zz\ol{\A})\otimes\Zz[\Iz^{-1}]\arrow{rrr}{\H_*(\tilde{\sigma}_s,\alpha_s)\otimes\id_{\Zz[\Iz^{-1}]}}  & & & \H_*(\A,\Zz\ol{\A})\otimes\Zz[\Iz^{-1}] \\
       \H_*(A,\Zz\ol{A})\otimes\Zz[\Iz^{-1}]\arrow[u,"{\iota\otimes \id_{\Zz[\Iz^{-1}]}}"]  & & & \H_*(A,\Zz\ol{A})\otimes\Zz[\Iz^{-1}] \arrow[u,"{\iota\otimes\id_{\Zz[\Iz^{-1}]}}"]\\
       \H_*(A)\otimes\Zz[\Iz^{-1}]\arrow{u}{\tilde{\rho}_{1,\infty}}\arrow{rrr}{\theta_s\otimes\id} & & & \H_*(A)\otimes\Zz[\Iz^{-1}]\arrow{u}{\tilde{\rho}_{1,\infty}}\nospacepunct{.}
    \end{tikzcd}
\end{equation*}
Observe that
\[
 \H_*(\tilde{\sigma}_s, \alpha_s) \circ \iota \circ \H_*(\sigma_1, \ind_1)
    = \iota \circ \H_*(\sigma_s, \ind_s).
\]
Since $\rho_{t,\infty}=\H_*(\sigma_t, \ind_t)$ for all $t\in S$ by \eqref{eqn:rhoinfty}, we have   
\begin{align*}
\theta_s\otimes\id&=\tilde{\rho}_{1,\infty}^{-1}\circ(\iota\otimes \id_{\Zz[\Iz^{-1}]})^{-1}\circ\H_*(\tilde{\sigma}_s,\alpha_s)\otimes\id_{\Zz[\Iz^{-1}]}\circ (\iota\otimes \id_{\Zz[\Iz^{-1}]})\circ \tilde{\rho}_{1,\infty}\\
&=(\tilde{\rho}_{s,\infty}\circ\tilde{\rho}_{1,s})^{-1}\circ (\iota\otimes \id_{\Zz[\Iz^{-1}]})^{-1}\circ (\iota \circ \H_*(\sigma_s, \ind_s)\otimes \id_{\Zz[\Iz^{-1}]}) \\
&=(\tilde{\rho}_{s,\infty}\circ\tilde{\rho}_{1,s})^{-1}\circ\tilde{\rho}_{s,\infty}=\tilde{\rho}_{1,s}^{-1}=\frac{1}{[A:\sigma_sA]}\H_*(\sigma_s,\id)\otimes\id_{\Zz[\Iz^{-1}]}.
\end{align*}
For the last equality, we have used Equation~\eqref{eqn:invertrho}. 
It remains to observe that under the identification $\H_*(A)\cong \lwedge^*(A)$, $\H_*(\sigma_s,\id)$ is taken to $\lwedge^*(\sigma_s)$, see \cite[Theorem~V.6.4(i)]{Brown}.
\end{proof}

\subsection{The case \texorpdfstring{$A \cong \Zz^d$}{A coong Zd}}

Let us specialize to the case $A \cong \Zz^d$ to prepare for Sections \ref{sec:rings} and \ref{sec:BOS}. We identify $S$ with its image in $\End_\Zz(A)\cong \M_d(\Zz)$, and we identify $\Zz$ with $\Zz1_d\subseteq \M_d(\Zz)$.

\begin{proposition}
\label{prop:Z[1/n]}
In addition to the standing assumptions of this section, assume $A \cong \Zz^d$ for some $d\geq 1$. Recall that $\H_*(A)\cong\lwedge^*A$. Then,
\begin{enumerate}[\upshape(i)]
\item For each $0\leq q<d$, we have
\[
\lwedge^qA\otimes\Zz[1/n]\leq \H_q(A,\Zz\ol{A})=\H_q(\A,\Zz\ol{\A})
\] 
for all $n\in S\cap \Zz_{>0}$. 
In particular, if $n\in S$ for all $n\in\Iz$, then $\H_q(\A,\Zz\ol{\A})=(\lwedge^q A)\otimes\Zz[\Iz^{-1}]$.
\item For $q>d$, we have $\H_q(\A,\Zz\ol{\A})=0$.
\item If $S\cap \Zz_{>0}$ is cofinal in $S$ with respect to $\mid_r$, then $\rho_{1,\infty}\colon \H_d(A)\to\H_d(\A, \Zz\ol{\A})$ is an isomorphism.
\end{enumerate}
\end{proposition}
\begin{proof}
The equality $\H_q(A,\Zz\ol{A})=\H_q(\A,\Zz\ol{\A})$ follows from Proposition~\ref{prop:equivAdditive}.

(i): We have $[A:nA]=n^d$ and $\lwedge^q(n)=n^q$, so by Theorem~\ref{thm1}, we have $\theta_s \otimes\id_{\Zz[\Iz^{-1}]} = n^{q-d}$, so that $\theta_s = n^{q-d}$. This implies that $\H_q(\A,\Zz\ol{\A})$ is divisible by $n$ since $q<d$, so that we have $\lwedge^qA\otimes\Zz[1/n]\leq \H_q(\A,\Zz\ol{\A})$. 

(ii) is immediate because $\lwedge^q A = 0$. 

(iii): By Lemma~\ref{lem:rho}, we have $\rho_{1,n}=n^{d-d}\id_{\H_d(A)}=\id_{\H_d(A)}$ for all $n\in S\cap \Zz_{>0}$. Thus, $\H_d(\A,\Zz\ol{\A})\cong \varinjlim_{s\in S}\{\H_d(A);\rho_{s,t}\}\cong\varinjlim_{n\in S\cap \Zz_{>0}}\{\H_d(A);\rho_{n,m}\} =\H_d(A)=\Zz$.
\end{proof}

Let us close this section by showing that the conclusion in part (iii) of Proposition~\ref{prop:Z[1/n]} holds whenever $S$ contains $\Zz_{>0}$ by the proof of Proposition~\ref{prop:Z[1/n]}.

\begin{proposition}   
\label{prop:cofinal}
In addition to the standing assumptions of this section, assume $A\cong\Zz^d$ for some $d\geq 1$. If $\Zz_{>0}\subseteq S$, then the canonical map $\varinjlim_{n\in \Zz_{>0}}\{\H_q(A); \rho_{n,m}\}\to\varinjlim_{s\in S}\{\H_q(A); \rho_{s,t}\}$ is an isomorphism for all $q \geq 0$.
\end{proposition}
\begin{proof}
For $a\in \End_\Zz(A)\cong \M_d(\Zz)$ with $\det(a)\neq 0$, let $\tilde{a}\in\End_\Zz(A)$ be such that $a\tilde{a}=\tilde{a}a=|\det(a)|$.

Put $T\coloneqq \S\cap\End_\Zz(A)$. 
First, we claim that for any $a \in T$, we have $\tilde{a} \in T$. To see this, we see that $\tilde{a} \in \End_\Zz(A)$ by definition, and we have $\tilde{a} = a^{-1}|\det(a)| \in \S$ by assumption.  
In particular, $T$ is a reversible monoid: Given $a,b\in T$, we have 
$|\det(a)| \in aT$ and $|\det(b)| \in bT$, so that $|\det(ab)| \in aT \cap bT$. Similarly, we can show $|\det(ab)| \in Ta \cap Tb$. 

Since $|\det(a)| \in Ta$ for any $a \in T$, the set of natural numbers $\{|\det(t)| : t\in T\}$ is cofinal in $T$, so that  $\Zz_{>0}$ is cofinal in $T$. 
Consider the following diagram:
\[
\begin{tikzcd}
   \varinjlim_{s\in S}\{\H_q(A); \rho_{s,t}\} \arrow[rr]& & \varinjlim_{s\in T}\{\H_q(A); \rho_{s,t}\} \\
   &\varinjlim_{n\in \Zz_{>0}}\{\H_q(A); \rho_{n,m}\}\nospacepunct{.}\arrow[lu]\arrow[ru] & 
\end{tikzcd}
\]
The arrow on the lower right is an isomorphism by cofinality of $\Zz_{>0}$ in $T$. Since $\Zz_{>0}\subseteq S\subseteq T$, $S$ is cofinal in $T$, so that the top arrow is also an isomorphism. 
\end{proof}
Note that $\Zz_{>0}$ is not necessarily cofinal in $S$, which is why we work with the monoid $T$ in the proof of Proposition~\ref{prop:cofinal}.

\section{Algebraic actions from finite rank rings}
\label{sec:rings}

\subsection{Rings of finite rank}
Let $R$ be a unital ring whose additive group is isomorphic to $\Zz^d$ for some $d\in\Zz_{>0}$. For instance, $R$ could be the ring of algebraic integers in a number field, matrices over such a ring, or the integral group ring of a finite group.
Let
\begin{equation*}
    R^\times\coloneqq\{a\in R : ax=x \text{ implies } x=0 \text{ for all } x\in R\}
\end{equation*}
be the muliplicative monoid of left regular elements in $R$. Note that $R^\times$ is automatically cancellative because $R$ is of finite rank. Let $S\subseteq R^\times$ be a reversible submonoid such that $\Zz_{>0}\subseteq S$. Then, $S\acts R$ is a faithful, exact algebraic action satisfying the finite index property. To see exactness, note that $nR$ is an $R^\times$-constructible subgroup for every $n\in \Zz_{>0}$, and $\bigcap_{n\geq 1}nR=\{0\}$.

We shall identify $R$ with its image in the $\Qz$-algebra $\A\coloneqq R\otimes \Qz$. Under this identification, $S$ is a submonoid of the group of units $\A^*$ of $\A$. Put $\S\coloneqq S^{-1}S$. For every $a\in \A$, there exists $n\in\Zz_{>0}$ such that $na\in R$. Thus, $(R^\times)^{-1}R^\times=R^\times(R^\times)^{-1}=\A^*$. In particular, the monoid $R^\times$ is reversible. 

We shall apply our general results to $A=R$ and our action $S\acts R$. For $a\in \S$, let $N_\Qz(a)\in\Qz^*$ be the determinant of the $\Qz$-linear isomorphism $\A\to\A$ given by $x\mapsto ax$. For $a\in S$, we have $N_\Qz(a)=[R:aR]$ (this can be shown by considering the Smith Normal Form).

Theorem~\ref{thm1}, Proposition~\ref{prop:Z[1/n]}, 
and Proposition~\ref{prop:cofinal} give the following:
\begin{theorem}
\label{thm:rings}
For every $q\geq 0$, we have a canonical identification
\begin{equation*}
  \H_q(\A,\Zz\ol{\A})\cong \begin{cases} \lwedge_\Qz ^q \A & \text{ if } q <d, \\
  \lwedge_\Zz ^d R \cong \Zz & \text{ if } q=d, \\
  0 & \text{ if }q > d, 
  \end{cases}
 \end{equation*}
  and under this identification, the action $\theta\colon \S\acts \H_q(\A,\Zz\ol{\A})$ is characterized by
\begin{equation*}
    \theta_a=\frac{1}{|N_\Qz(a)|}\lwedge^q(a)\quad (a\in\S),
\end{equation*} 
where $\lwedge^q(a)$ is the exterior power of the multiplication-by-$a$ map on $\A$.
\end{theorem}


Let $\sign\colon \S\to\{\pm 1\}$ be given by $\sign(a)\coloneqq\frac{N_\Qz(a)}{|N_\Qz(a)|}$, and denote by $\Zz_{\sign}$ the $\S$-module $\Zz$ on which $\S$ acts via $\sign$.

\begin{proposition}\label{prop:group.homology.LHS}
We have 
\begin{equation*}
    \H_p(\S,\H_q(\A,\Zz\ol{\A}))=\begin{cases}
    0 & \text{ if } 0\leq q<d,\\
    \H_p(\S,\Zz_{\sign}) & \text{ if } q=d,\\
    0 & \text{ if } q>d.
    \end{cases}
\end{equation*}
\end{proposition}
\begin{proof}
First, suppose $0\leq q<d$. Choose $n\in\S\cap\Zz_{>1}$. Then, $\theta_n(x)=n^{q-d}x$ for all $x\in\lwedge_\Qz^q\A$. Note that $0<n^{q-d}<1$ because $q<d$. We have 
\begin{equation*}
\H_k(\gp{n},\lwedge_\Qz^q\A)=
\begin{cases}
(\lwedge_\Qz^q\A)_{\gp{n}}& \text{ if } k=0,\\
(\lwedge_\Qz^q\A)^{\gp{n}}& \text{ if } k=1,\\
0 & \text{ if } k>1,
\end{cases}
\end{equation*}
(see, for instance, \cite[Example~6.1.4]{Weibel}). Here, $\gp{n}=n^\Zz$. 
The group 
\[
(\lwedge_\Qz^q\A)_{\gp{n}}=\lwedge_\Qz^q\A/\langle(1-n^{k(q-d)})x : x\in \lwedge_\Qz^q\A, k\in\Zz\rangle
\]
is trivial because $\lwedge_\Qz^q\A$ is divisible and $(1-n^{q-d})\lwedge_\Qz^q\A\leq \langle(1-n^{k(q-d)})x : x\in \lwedge_\Qz^q\A, k\in\Zz\rangle$. 
The group 
\[
(\lwedge_\Qz^q\A)^{\gp{n}}=\{x\in \lwedge_\Qz^q\A : \theta_n(x)=x\}
\]
is trivial because $\theta_n$ has no fixed points except $0$. Since $\gp{n}\leq \S$ is a normal subgroup, the Lyndon--Hochschild--Serre spectral sequence now implies that $\H_*(\S,\lwedge_\Qz^q\A)=0$ for all $0\leq q<d$.

By Theorem~\ref{thm:rings}, we have $\H_d(\A,\Zz\ol{\A})=\lwedge_\Zz^d\A=\Zz$ and $\theta_a=\sign(a)$.
\end{proof}

\begin{corollary} \label{cor:numberfields}
The Lyndon--Hochschild--Serre spectral sequence for $\H_*(\A\rtimes\S,\Zz\ol{\A})$ collapses at the $E^2$-page, and for all $n\geq 0$, we have
\begin{equation*}   \H_n(\A\rtimes\S,\Zz\ol{\A})\cong \H_{n-d}(\S,\Zz_{\sign}).
\end{equation*}
\end{corollary}

\subsection{Rings of algebraic integers}
Let us now specialize to the case of rings of algebraic integers in number fields.

\begin{lemma}
\label{lem:sign1}
Let $K$ be a number field. 
For $a\in K^*$, we have 
\begin{equation}
    \sign(a)=\frac{N_\Qz(a)}{|N_\Qz(a)|}=
        \prod_{w\in V_{K,\Rz}}\sign(w(a)),
\end{equation}
where $V_{K,\Rz}$ is the set of real embeddings of $K$.
\end{lemma}
\begin{proof}
By \cite[Proposition~I.2.6]{Neu}, we have $N_\Qz(a)=\prod_{w\in V_K}w(a)$, where $V_K$ is the set of field embeddings of $K$ into $\Cz$. Let $V_{K,\Cz}$ denote the set of field embeddings $w$ of $K$ into $\Cz$ such that $w(K)\not\subseteq\Rz$. Now we have 
\[
  \frac{N_\Qz(a)}{|N_\Qz(a)|}=\prod_{w\in V_{K,\Cz}}\frac{w(a)}{|w(a)|} \prod_{w\in V_{K,\Rz}}\frac{w(a)}{|w(a)|}=\prod_{w\in V_{K,\Rz}}\frac{w(a)}{|w(a)|},
\]
where we define the empty product to be $1$. Here, we used that for each complex embedding $w\in V_{K,\Cz}$, there is a unique conjugate embedding $\ol{w}\in V_{K,\Cz}$ satisfying $w(a)\ol{w}(a)=|w(a)|^2$ for all $a\in K$.
\end{proof}

\begin{lemma}
\label{lem:sign2}
 Given a number field $K$ that has a real embeddings, there exists $a\in K^*$ such that $\sign(a)=-1$ and $a^\Zz$ is a summand of $K^*$.
\end{lemma}
\begin{proof}
Pick $w \in V_{K,\mathbb{R}}$. By \cite[Theorem I.3.4]{Neu}, there is $a \in K^*$ such that $w(a) <0$ and $w'(a) >0$ for any $w' \in V_{K,\mathbb{R}}$. By Lemma~\ref{lem:sign1}, this means that $\mathrm{sign}(a) = -1$, and hence $\mathrm{sign}$ is surjective. Recall that the multiplicative group $K^*$ of $K$ is isomorphic to the direct product $\mu \times \Gamma$, where $\mu$ is the group of roots of unity and $\Gamma$ is a free abelian group (Subsection \ref{sec:ANT}). If $|V_{K,\mathbb{R}}|$ is even, then $\mathrm{sign}(\mu)=\{ + 1\}$, and hence $\mathrm{sign}|_{\Gamma}$ is surjective. If $|V_{K,\mathbb{R}}|$ is odd, in which case $\mathrm{sign}(\mu) = \{ \pm 1\}$, we may assume that $\mathrm{sign}|_{\Gamma}$ is surjective by replacing $\Gamma = \langle a_1,a_2,a_3,\cdots \rangle$ with $\Gamma ' \coloneqq \langle -a_1,a_2,a_3,\cdots \rangle$ if necessary. 

Pick $b \in \Gamma$ such that $\mathrm{sign}(b) = -1$. Then, since $\Gamma$ is free abelian, we can take the primitive part $a \in \Gamma$ of $b$, i.e., the unique element $a \in \Gamma$ such that $b = m a$ for some integer $m \ge 1$ and $a$ is not divisible by any integer greater than $1$. Then $\mathrm{sign}(a)$ must be $-1$, and $a^{\mathbb{Z}}$ is a direct summand.
\end{proof}

\begin{theorem} 
\label{thm:gpdhomology}
For a number field $K$ of degree $d=[K:\Qz]$ with ring of integers $R$, we have the following:
\begin{enumerate}[\upshape(i)]
\item If $K$ is totally imaginary, then 
\[ \H_n(\cG_{R^\times\acts R}) \cong \begin{cases} 
0 & \text{ if } 0\leq n<d,\\
\Zz & \mbox{ if } n=d, \\
\Zz^{\oplus\infty} \oplus \Zz/|\mu|\Zz & \mbox{ if }n=d+1, \\
\Zz^{\oplus\infty} \oplus (\Zz/|\mu|\Zz)^{\oplus\infty} & \mbox{ if }n \geq d+2,
\end{cases}\]
where $\mu=\mu_K$ is the group of roots of unity in $K$.
\item If $K$ has a real embedding, then 
\[ \H_n(\cG_{R^\times\acts R}) \cong \begin{cases} 
0 & \text{ if } 0\leq n<d,\\
\Zz/2\Zz & \mbox{ if }n=d, \\
(\Zz/2\Zz)^{\oplus\infty} & \mbox{ if }n \geq d+1. \\
\end{cases}\]
\end{enumerate}
\end{theorem}
\begin{proof}
First, suppose $K$ is totally imaginary. In this case, the action $K^* \acts \Zz_\sign$ is the trivial action by Lemma~\ref{lem:sign1}. 
Since $K^*/\mu$ is free abelian, we have $K^*\cong \mu\times (K^*/\mu)$ and $\H_*(K^*/\mu)\cong\bigwedge^*(K^*/\mu)$, where the homology group here denotes group homology with $\Zz$-coefficients (see \cite[Theorem~V.6.4(i)]{Brown} for the latter claim). Further, note that $K^*/\mu$ is abstractly isomorphic to $\Zz^{\oplus\infty}$, see Section~\ref{sec:ANT}. By the K\"unneth formula (Theorem~\ref{thm:Kunneth}), we have 
\begin{align*} 
\H_n(K^*, \Zz_\sign) &= \H_n(\mu \times (K^*/\mu))=\bigoplus_{p+q=n} \H_p(\mu) \otimes \H_q(K^*/\mu) \\
&= \H_n(K^*/\mu) \oplus \bigoplus_{p \colon\mbox{odd}} (\Zz/|\mu|\Zz) \otimes \H_{n-p}(K^*/\mu), 
\end{align*}
so that the claim (i) holds by Corollary~\ref{cor:numberfields}. Next, suppose the number of real embeddings of $K$ is odd. In this case, $\mu=\{\pm 1\}$ and $\sign(-1)=-1$. 
Let $\Gamma = \ker \sign \subseteq K^*$. Then, $\Gamma \cong K^*/\mu$ and $K^*=\mu \times \Gamma$. 
By K\"unneth's formula (Theorem~\ref{thm:Kunneth}), we have 
\[ \H_n(K^*, \Zz_\sign) = \bigoplus_{p+q=n} \H_p(\mu, \Zz_\sign) \otimes \H_q(K^*/\mu). \]
By \cite[Theorem~6.2.2]{Weibel}, we can see that 
\[ 
\H_p(\mu, \Zz_\sign) = 
\begin{cases} \Zz/2\Zz & \mbox{if $p$ is even}, \\
0 & \mbox{if $p$ is odd}, 
\end{cases} 
\]
so that the claim (ii) is true in this case by Corollary~\ref{cor:numberfields}. Finally, suppose the number of real embeddings of $K$ is even and nonzero. 
In this case, $\mu=\{\pm 1\}$ and $\sign(-1)=1$. 
By Lemma~\ref{lem:sign2}, there exists $a \in K^*$ with $\sign(a)=-1$ such that $a^\Zz$ is a summand of $K^*$. 
Let $\Lambda = \{\pm 1\} \times a^\Zz$. 
Then, there exists a free abelian subgroup of infinite rank $\Gamma \subseteq K^*$ such that $K^* = \Lambda \times \Gamma$. 
We can assume that $\sign(\Gamma)=1$. To see this, let $\{a_i\}_{i=1}^\infty$ be a basis of $\Gamma$. For each $i$, let $b_i = a_i$ if $\sign(a_i)=1$, and let $b_i = aa_i$ if $\sign(a_i) = -1$. Then, $\{b_i\}_{i=1}^\infty$ still generates a free abelian subgroup $\tilde{\Gamma}$ such that $K^*=\Lambda \times \tilde{\Gamma}$, and we have $\sign(\tilde{\Gamma})=1$.
We calculate $\H_p(\Lambda, \Zz_\sign)$. 
By the K\"unneth formula (Theorem~\ref{thm:Kunneth}), we have the following short exact sequence: 
\[ 0 \to \bigoplus_{s+t=p} \H_s(\mu) \otimes \H_t(a^\Zz, \Zz_\sign) \to \H_p(\Lambda,\Zz_\sign) \to 
\bigoplus_{s+t=p-1} \Tor (\H_s(\mu), \H_t(a^\Zz, \Zz_\sign)) \to 0.\]
Since
\begin{align}
\H_t(a^\Zz, \Zz_\sign) \cong \begin{cases} \Zz/2\Zz & \mbox{ if }t=0, \\
0 & \mbox{ if } t>0, \end{cases} 
\label{eqn:Z.sign.homology}
\end{align}
we have 
\[ 
\H_p(\Lambda,\Zz_\sign) \cong 
\begin{cases}
\H_0(\mu) \otimes \Zz/2\Zz & \mbox{if $p=0$}, \\
\H_p(\mu) \otimes \Zz/2\Zz & \mbox{if $p$ is odd}, \\
\Tor(\H_{p-1}(\mu), \Zz/2\Zz) & \mbox{if $p$ is even and $p>0$}.
\end{cases}
\]
Thus, $\H_p(\Lambda,\Zz_\sign) \cong \Zz/2\Zz$ for all $p \geq 0$. Since 
\[ \H_n(K^*, \Zz_\sign) = \bigoplus_{p+q=n} \H_p(\Lambda, \Zz_\sign) \otimes \H_q(K^*/\Lambda) \]
by K\"unneth's formula (Theorem~\ref{thm:Kunneth}), the claim (ii) is true in this case by Corollary~\ref{cor:numberfields}. 
\end{proof}

\subsection{Application to topological full groups}

We now combine our results on homology with \cite{Li:TFG}. First, we consider integral group homology.

\begin{theorem} 
\label{thm:TFGhomology}
Let $K$ be a number field of degree $d=[K :\Qz]$ with ring of integers $R$.
    We have $\H_n(\fg{\cG_{R^\times\acts R}}) = 0$ for $0<n<d$,
    and 
    \[
    \H_d(\fg{\cG_{R^\times\acts R}}) = \begin{cases}
        \Zz & \mbox{if $K$ is totally imaginary}, \\
        \Zz/2\Zz & \mbox{if $K$ has a real embedding}.
    \end{cases}
    \]
    Moreover, $\fg{\cG_{R^\times\acts R}}$ is simple if and only if $K\neq \Qz$, and if $K=\Qz$, then $\fg{\cG_{R^\times\acts R}}^{\rm ab} \cong \Zz/2\Zz$. 
\end{theorem}
\begin{proof}
    By Corollary~\ref{cor:numberfields}, we have $\H_n(\cG_{R^\times\acts R}) = 0$ for $0\leq n<d$ and 
    \begin{align*}
        &\H_d(\cG_{R^\times\acts R}) = \H_0(K^*, \Zz)
        = \Zz/ \langle 1- \sign (a) \colon a \in K^*\rangle  \\
        &= \begin{cases}
        \Zz & \mbox{if $K$ is totally imaginary,} \\
        \Zz/2\Zz & \mbox{if $K$ has a real embedding}.
    \end{cases}
    \end{align*}
    By \cite[Theorem 4.22]{BruceLi2}, the groupoid $\cG_{R^\times\acts R}$ is purely infinite in the sense of \cite[Definition~4.9]{Mat15}, so in particular, $\cG_{R^\times\acts R}$ has comparison.
     Now the claim follows from Theorem~\ref{thm:gpdhomology} and \cite[Corollary~D]{Li:TFG}. 

    By \cite[Theorem~4.10]{BruceLi2}, the groupoid $\cG_{R^\times\acts R}$ is minimal. Since $\cG_{R^\times\acts R}$ is purely infinite and minimal, the commutator subgroup of $\fg{\cG_{R^\times\acts R}}$ is simple by \cite[Theorem~4.16]{Mat15}, which yields the second claim.
\end{proof}

Now we turn to rational group homology.

\begin{theorem}
\label{thm:acyclic}
Let $K$ be a number field of degree $d=[K :\Qz]$ with ring of integers $R$.
\begin{enumerate}[\upshape(i)]
    \item If $K$ is totally imaginary, then 
    \begin{equation}
    \H_n(\fg{\cG_{R^\times\acts R}},\Qz)\cong
    \begin{cases}
        0 &  \text{ if } 1\leq n<d,\\
        \Qz & \text{ if } n=d,\\
        \Qz^{\oplus\infty} & \text{ if } n>d.
    \end{cases}
\end{equation}
    \item If $K$ has a real embedding, then $ \H_n(\fg{\cG_{R^\times\acts R}},\Qz)=0$ for all $n\geq 1$.
\end{enumerate}

In particular, the group $\fg{\cG_{R^\times\acts R}}$ is rationally acyclic if and only if $K$ has a real embedding.  
\end{theorem}
\begin{proof}
As explained in the proof of Theorem~\ref{thm:TFGhomology}, the groupoid $\cG_{R^\times\acts R}$ is minimal and has comparison.
 The statement now follows from Theorem~\ref{thm:gpdhomology} and \cite[Corollary~C]{Li:TFG}.
\end{proof}

\subsection{Comparison of K-theory and groupoid homology}
In this section, we study some spectral sequences associated to the dynamical system $\mathscr{A} \rtimes \mathscr{S} \acts \overline{\mathscr{A}}$. 
As a consequence, we get an isomorphism between groupoid homology and K-theory of the reduced crossed product.
This section recovers the computation of Cuntz--Li in \cite[Section 5]{CL} and Li--L\"{u}ck in \cite[Theorem 1.2]{LiLuck} without using the Cuntz--Li duality of finite and infinite adeles from \cite{CL}.

A key ingredient is the use of Raven's Chern character. In Raven’s thesis \cite{Raven}, for a countable discrete group $G$, two $G$-equivariant bivariant homology theories, namely 
\[
    \mathrm{KK}_*^G((X,A),Y), \quad \text{and} \quad  \mathrm{HH}^*_G((X,A), Y), 
\]
are considered. Here, we may take $(X,A)$ to be a pair of $G$-CW-complexes (a $G$-CW-complex $X$ together with a $G$-subcomplex $A$), and $Y$ to be a locally compact $G$-space.
Moreover, Raven's Chern character
\[
    \mathop{\widehat{\mathrm{ch}}_R} \colon \mathrm{KK}_*^G((X,A),Y) \to \widehat{\mathrm{HH}}{}^*_G((X,A), Y), \quad 
    \widehat{\mathrm{HH}}{}_G^{*} ((X,A) , Y) \coloneqq \bigoplus_{[\gamma ]   \in  \mathcal{C}(G_{\rm tor}) }  \mathrm{HH}_{Z(\gamma )}^{*} ((X^\gamma , A^\gamma)  ,  Y^\gamma  ). 
\]
is constructed, and is proved to be a natural isomorphism after tensoring with $\mathbb{C}$. Here, $\mathcal{C}(G_{\rm tor})$ denotes the set of conjugacy classes of torsion elements in $G$, $Z(\gamma) \coloneqq \{ g \in G \mid g\gamma = \gamma g\}$ denote the centralizer group of $\gamma$, and $X^\gamma \coloneqq \{ x \in X \mid \gamma \cdot x=x\}$ denote the $\gamma$-fixed point subspace. 

This Chern character have various good functorial and natural properties. 
Moreover, Deeley--Willett \cite{DW} proved that, by taking the universal proper $G$-space $\underline{E}G$ (cf.\ Remark \ref{rmk:EG} below) as the first entry, there are isomorphisms
\begin{align}
    \begin{split}
    \mathrm{KK}^G_*(\underline{E}G,Y) \cong {}&{} \mathrm{K}_*(C_0(Y) \rtimes G), \\
    \widehat{\mathrm{HH}}{}_G^*(\underline{E}G,Y) \cong {}&{} \mathrm{H}_{-*}(\widehat{Y} \rtimes G\,; \mathbb{C}) \cong  \bigoplus_{[\gamma] \in \mathcal{C}(G_{\mathrm{tor}})}\H_{-*}(Y^\gamma \rtimes Z(\gamma); \mathbb{C}),
    \end{split}
\label{eqn:Raven.BC}
\end{align}
under a certain assumption as follows. 
First, the above isomorphism in K-theory is given by the Baum--Connes isomorphism under the assumption that $G$ satisfies the property $\gamma=1$ in the sense of \cite[Section 5]{Kasparov} and \cite[Definition 8.6]{MN}, e.g.\ if $G$ satisfies the Haagerup property \cite{HK}. 
Indeed,  $\mathrm{KK}^G_*(\underline{E}G,Y) \cong \varinjlim_{Z \subset \underline{E}G}\mathrm{KK}^G_*(C_0(Z), C_0(Y))$ is the domain of the Baum--Connes assembly map. 
Meanwhile, the second isomorphism in homology is proved in \cite[Proposition 2.16]{DW} under a certain assumption on $Y$ that is satisfied if $Y$ is totally disconnected. 
In conclusion, the isomorphism of $\mathop{\widehat{\mathrm{ch}}{}_R}$ after tensoring with $\mathbb{C}$ is nothing else than the rational HK isomorphism. 

For the reader’s convenience, we collect in Appendix \ref{section:Raven} the results needed in this paper that are either written in \cite{Raven,DW} or basic but whose proofs are not written down explicitly in the literature.

\subsubsection{Comparison of Lyndon--Hochschild--Serre type spectral sequences}
We will compare the spectral sequences of bivariant K-theory and homology groups, whose $E_2$-pages agree with that of the Lyndon--Hochschild--Serre spectral sequence, via Raven's Chern character. 
In this subsubsection, let $G = N \rtimes \Gamma$, where $N$ and $\Gamma$ are countable groups such that $\Gamma$ is torsion-free. 

\begin{remark}\label{rmk:EG}
For s countable discrete group $G$, $\underline{E}G$ denotes the universal space for proper $G$-actions in the sense of \cite[Definition 2.3]{Luck} (in which it denotes $\underline{J}G$). 
Such a $G$-space is characterized by a universal property uniquely up to $G$-homotopy equivalence. We recall some properties of $\underline{E}G$. 
\begin{itemize}
    \item[(i)] There is a proper $G$-CW-complex that models $\underline{E}G$ (\cite[Lemma 3.3]{Luck}, see also \cite[Definition 1.8]{Luck}). 
    \item[(ii)] A proper $G$-CW-complex models $\underline{E}G$ if and only if its $K$-fixed point subspace is contractible for any finite subgroup $K$ (\cite[Theorem 1.9]{Luck}). 
    \item[(iii)] By (ii), the universal free proper $G$-space $EG$ has the same model as $\underline{E}G $ if $G$ is torsion-free. 
    \item[(iv)] By (ii), for a subgroup $H \leq G$, a proper $G$-CW-complex that models $\underline{E}G$ also models $\underline{E}H$ by restricting the $G$-action to $H$.
\end{itemize}
\end{remark}

Let $X = \underline{E}G$, and let $Y$ be a locally compact $G$-space. 
In the next subsubsection, we will apply the theory to the case that $\Gamma = \mathscr{S}$ and $N = \mathscr{A}$ are the groups as in the notations of Section \ref{sec:algactions} and $Y = \overline{\mathscr{A}}$ as in the notation of Section \ref{sec:algactions}.

Let us consider a simplicial set $\mathbf{E}\Gamma  $, whose $k$-simplices are $(\mathbf{E}\Gamma)_k = \Gamma^{k+1}$ and the face maps are given by 
\[
    {d}_i \colon \Gamma^{k+1} \to \Gamma^k, \quad {d}_i(g_0,\cdots, g_k) = (g_0,\cdots, g_ig_{i+1}, \cdots, g_k) \text{\ \ for $0 \leq i \leq k-1 $}
\]
and ${d}_{k}(g_0,\cdots,g_k)=(g_0,\cdots,g_{k-1})$ (we omit the definition of the degeneracy map since it is not used later).
Let $E\Gamma$ denote the fat geometric realization $E\Gamma = \| \mathbf{E}\Gamma \|$ of $\mathbf{E}\Gamma $, which is $\Gamma$-equivariantly homotopy equivalent to the standard geometric realization $|\mathbf{E}\Gamma |$ (see e.g.~\cite[Appendix A]{Segal}), and let $E_n\Gamma$ denote the $n$-skeleton of $E\Gamma$. More explicitly, 
\[ 
    E_n\Gamma = \bigg(\bigsqcup_{0 \leq k \leq n} \Gamma^{k+1} \times \Delta_k\bigg) /( \text{$(d_i (s_0,\cdots,s_k), (t_0,\cdots,t_{k-1})) \sim ((s_0,\cdots,s_k), \delta_i (t_0,\cdots,t_{k-1}))$}),
\]
where $\Delta_k \coloneqq \{ (t_0,\cdots,t_k) \in \mathbb{R}^{k+1}_{\geq 0} : \sum t_i =1\}$ and
\begin{align*}
    \delta_i \colon \Delta_{k-1} \to \Delta_{k}, \quad \delta_i (t_0,\cdots, t_{k-1}) = (t_0,\cdots, t_{i-1},0,t_i, \cdots , t_{k-1}) \text{\ \ for $0 \leq i \leq k $}.
\end{align*} 
We remark that the canonical inclusion
\[
    \bigsqcup_{0 \leq k \leq n-1} \Gamma^{k+1} \times \Delta_k \hookrightarrow \bigsqcup_{0 \leq k \leq n} \Gamma^{k+1} \times \Delta_k
\]
induces a bijective map from $E_{n-1}\Gamma$ to the $(n-1)$-skeleton of $E_n\Gamma$ and 
\begin{align}
    E_n \Gamma \setminus E_{n-1}\Gamma = \Gamma^{n+1} \times \mathop{\mathrm{int}} \Delta_n.\label{eqn:EG.cell}
\end{align}
The group $\Gamma$ acts on $\Gamma^{k+1}$ by $g\cdot (g_0,\cdots,g_k) = (gg_0,g_1,\cdots,g_k)$. Since ${d}_i$ is $\Gamma$-equivariant, each $E_n\Gamma$ is a free proper $\Gamma$-CW-complex. Indeed, $E\Gamma$ is contractible, and hence models the universal free proper $\Gamma$-space.

Therefore, if a $G$-CW-complex $\underline{E}G$ models the universal example of free proper $G$-spaces, so does the direct product $\underline{E}G \times E\Gamma$ with another $G$-space $E\Gamma$ whose $\gamma$-fixed point subspace $(E\Gamma)^\gamma =E\Gamma$ is contractible for any $\gamma \in G_{\mathrm{tor}}$. 
Set 
\begin{align}
    \begin{split}
    \mathrm{DK}_{pq}^1 \coloneqq & \mathrm{KK}_{p+q}^{G } (\underline{E}G \times E_p\Gamma  , Y),\\
    \mathrm{EK}_{pq}^1 \coloneqq & \mathrm{KK}_{p+q}^{G } (\underline{E}G \times (E_p\Gamma, E_{p-1}\Gamma) , Y),\\
    \widehat{\mathrm{DH}}{}_{pq}^1 \coloneqq & \widehat{\mathrm{HH}}{}^{-p-q}_{G} (\underline{E}G \times E_p\Gamma  , Y),\\
    \widehat{\mathrm{EH}}{}_{pq}^1 \coloneqq & \widehat{\mathrm{HH}}{}^{-p-q}_{G } (\underline{E}G \times (E_p\Gamma, E_{p-1}\Gamma) , Y).
    \end{split}\label{eqn:exact.couple.LHS}
\end{align}
Then the long exact sequences for pairs $(\underline{E}G \times E_p\Gamma, \underline{E}G \times  E_{p-1}\Gamma)$ (\cite[Propositions 4.6.8, 6.5.6]{Raven}) make both $(\mathrm{DK}_{pq}^1 ,\mathrm{EK}_{pq}^1)$ and $(\widehat{\mathrm{DH}}{}_{pq}^1 ,\widehat{\mathrm{EH}}{}_{pq}^1)$ exact couples, which associate the homological spectral sequences. Moreover, by \cite[Corollary 7.3.4]{Raven}, Raven's Chern character gives a morphism of exact couples
\begin{align}
    \mathop{\mathrm{ch}_R^e} \colon (\mathrm{DK}_{pq}^1 ,\mathrm{EK}_{pq}^1) \to \bigoplus_{k \in \mathbb{Z}_{\geq 0}} (\widehat{\mathrm{DH}}{}_{p+2k,q}^1 , \widehat{\mathrm{EH}}{}_{p+2k,q}^1). \label{eqn:comparison.LHS}
\end{align} 

\begin{lemma}
    Assume that $E\Gamma$ is modeled by a finite dimensional $\Gamma$-CW-complex. Then, the spectral sequences associated to the exact couples \eqref{eqn:exact.couple.LHS} converges to $\mathrm{KK}_{*}^{G } (E\Gamma \times \underline{E}G  , Y) \cong \mathrm{KK}_{*}^{G } (\underline{E}G  , Y)$ and $\widehat{\mathrm{HH}}{}^{-*}_{G} (E\Gamma \times \underline{E}G  , Y) \cong \widehat{\mathrm{HH}}{}^{-*}_{G} ( \underline{E}G  , Y)$ respectively. 
\end{lemma}
\begin{proof}
By the finite dimensionality of a model of $E\Gamma$, we have $E^2_{pq} \cong 0$ for $p > \dim E\Gamma$. 
Thus, the morphisms $i_2^{pq} \colon \mathrm{DK}^2_{pq} \to \mathrm{DK}^2_{p+1,q-1}$ and $i_2^{pq} \colon \widehat{\mathrm{DH}}{}^2_{pq} \to \widehat{\mathrm{DH}}{}^2_{p+1,q-1}$ of the derived exact couples associated to \eqref{eqn:exact.couple.LHS} are isomorphisms if $p \geq \dim X$. 
Hence, by \cite[Corollary 3.25]{McCleary}, they converge to $\mathrm{KK}_{*}^{G } (X \times \underline{E}G  , Y)$ and $\widehat{\mathrm{HH}}{}^{*}_{G} (X \times \underline{E}G  , Y)$ respectively. 
\end{proof}

The following lemma follows the line of \cite[Proposition 3.2]{PY}.
Here, we regard $X \times \Gamma$ as a $G$-space by $hn \cdot (x,g) = (hn \cdot x, hg)$. Note that $[gn,x]  \mapsto  (gn x, g)$ gives an isomorphism of $G$-spaces $G \times _N X \cong \Gamma \times X \cong  X \times \Gamma $. This $G$-space is equipped with another right action commuting with the left $G$-action, say $X \times \Gamma \curvearrowleft \Gamma : \lambda$, given by $\lambda_h(x,g) \coloneqq (x,gh)$. The induced action on $G \times_N X$ is $\lambda_h([gn,x])  = [gnh,h^{-1}x]$. 
Via the isomorphisms given by Lemma \ref{lemma:induction.tKK.HH}, the groups $\mathrm{KK}^{N}_n(X,Y) \cong \mathrm{KK}_n^G(G \times_N X,Y)$ and $\widehat{\mathrm{HH}}{}_{N}^n(X,Y) \cong \widehat{\mathrm{HH}}{}_G^n(G \times_N X,Y)$ are equipped with the structure of right $\mathbb{Z}[\Gamma]$-modules, which is explicitly given by $g \mapsto \mathrm{KK}_n^G(\lambda_g,Y)$ and $g \mapsto \widehat{\mathrm{HH}}{}^n_G(\lambda_g,Y)$. 
By definition, these actions are compatible via Raven's Chern character. 
As in the definition of a model of $E\Gamma$ below \eqref{eqn:EG.cell}, $\Gamma$ acts on $\Gamma^{k+1}$ by $g \cdot (g_0,\cdots,g_k) = (gg_0,g_1,\cdots,g_k)$ for $g ,g_0,\cdots,g_k\in \Gamma$. This induces the $G$-action via the quotient homomorphism $G \to \Gamma$. This induces the left $\mathbb{Z}[\Gamma]$-module structure on $R[\Gamma^{k+1}]$, for a commutative ring $R$, explicitly given by $g \cdot \delta_{g_0,\cdots,g_k} = \delta_{gg_0,g_1,\cdots,g_{k}}$.    

\begin{lemma}\label{lem:KK.HH.LHS}
    For any proper  $G$-CW-complex $X$ and any $G$-space $Y$, we have isomorphisms
    \begin{align*}
        \mathrm{KK}^G_n(X \times \Gamma^{k+1}, Y) \cong {}&   \mathrm{KK}^{N}_n(X,Y) \otimes_{\mathbb{Z}[\Gamma]}  \mathbb{Z}[\Gamma^{k+1}] , \\
        \widehat{\mathrm{HH}}{}_G^n(X \times \Gamma^{k+1}, Y) \cong {}&   \widehat{\mathrm{HH}}{}_{N}^n(X,Y) \otimes_{\mathbb{C}[\Gamma]} \mathbb{C}[\Gamma^{k+1}]
    \end{align*} 
    that is compatible with Raven's Chern character.
    Moreover, for any $\Gamma$-equivariant map $f \colon \Gamma^{k+1} \to \Gamma^{l+1}$, the induced maps $\mathrm{KK}^G(f,Y)$ and $\widehat{\mathrm{HH}}{}^n_G(f,Y)$ are identical with $\id \otimes_{\mathbb{Z}[\Gamma]} f_* $ and $\id \otimes_{\mathbb{C}[\Gamma]} f_* $ respectively.
\end{lemma}
\begin{proof}
    In the proof, let the pair $(\mathrm{F}^G, R)$ be either $(\mathrm{KK}_n^G, \mathbb{Z})$ or $(\widehat{\mathrm{HH}}{}_G^n ,\Cz)$. Let $\mu_{g_0,\cdots,g_k} \colon \Gamma \to \Gamma^{k+1}$ denote the $\Gamma$-equivariant map $g \mapsto (gg_0,g_1,\cdots,g_n)$. 
    Then the map
    \begin{align*}
    \mathrm{F}^{N}(X,Y) \otimes R[\Gamma^{n+1}] \to {}& \mathrm{F}^{G} (X \times \Gamma^{k+1}  , Y), \\
    \delta_{g_0,\cdots,g_k} \otimes \xi \mapsto {}& \mathrm{F}^G(\id_X \times \mu_{g_0,\cdots,g_k}, Y) (i_{N}^{G} \xi),
    \end{align*}
    factors through the quotient $  \mathrm{F}^{N}(X,Y) \otimes_{R[\Gamma]} R[\Gamma^{k+1}]$ since 
    \[ 
        \mathrm{F}^{G}(\mu_{gg_0,g_1,\cdots,g_k} \times \id_X ,Y) (i_{N}^G\xi) = \mathrm{F}^{G}(\mu_{g_0,g_1,\cdots,g_k} \times \id_X,Y) \circ  \mathrm{F}^{G}(\lambda_g,Y) (i_{N}^G \xi). 
    \]
    The induced map decomposes to the direct product of 
    \[
    \mathrm{F}^G(X \times \Gamma  , Y) \otimes_{R[\Gamma]} R[\Im \mu_{g_0,\cdots,g_k}]
    \to 
    \mathrm{F}^G(X \times \Im \mu_{g_0,\cdots,g_k},Y),
    \]
    each of which is an isomorphism by definition. Moreover, if $f \colon \Gamma^{n+1} \to \Gamma^{m+1}$ is an $\Gamma$-equivariant map, then   
\[
    \mathrm{F}^G(\id_X \times f ,Y ) \circ \mathrm{F}^{G}(\id_X \times \mu_{(g_0,g_1,\cdots,g_n)},Y) (i_{N}^G\xi) = \mathrm{F}^{G}(\id_X \times \mu_{f(g_0,g_1,\cdots,g_n)},Y)(i_{N}^G\xi),
\]
    which means that $\mathrm{F}^G(f \times \id_X,Y) $ is identified with $f_* \otimes_{R[\Gamma]} \id$. Finally, by Lemma \ref{lemma:induction.tKK.HH}, the induced maps are compatible with Raven's Chern character. 
\end{proof}

\begin{remark}\label{rmk:action}
    We give an explicit description of the above $\Gamma$-action on $\mathrm{KK}^N_n(X,Y)$ and $\widehat{\mathrm{HH}}{}_N^n(X,Y)$ defined above, using the explicit descriptions of the isomorphisms $i_N^G$ and $r_G^N$ in the proof of Lemma \ref{lemma:induction.tKK.HH}. Let $\tau \colon \Gamma \curvearrowright X$ and $\nu \colon \Gamma \curvearrowright Y$ denote the $\Gamma$-actions.  
    
    First, the right action on $\mathrm{KK}^N_n(X,Y)$ is given by
    \begin{align*}
        g \cdot [M,f,\xi] ={}&{} r_G^N \circ \mathrm{KK}^G_*(\lambda_g,Y) \circ i_N^G [M,f,\xi] \\
        ={}&{} r_G^N [G \times_N M, \lambda_g \circ (\id_G \times f), \operatorname{Ind}_N^G \xi ] \\
        ={}&{} [g^{-1}N \times_N M,  \lambda_g \circ (\id_G \times f)|_{g^{-1}N \times_N M} ,   \operatorname{Ind}_N^G \xi |_{g^{-1}N \times_N (M \times Y)}]\\
        ={}&{} [{}_{g^{-1}}M,\tau_{g^{-1}} \circ f, (\id_M \times \nu_g)^*\xi], 
    \end{align*}
    where we write ${}_{g^{-1}}M$ for the manifold $M$ which is equipped with the $N$-action given by $(n,x) \mapsto gng^{-1}x$. The last isomorphism is induced from the $N$-equivariant diffeomorphism ${}_{g^{-1}}M \cong g^{-1}N \times _N M$ given by $x \mapsto [g^{-1},x]$. 
    Here we used $\lambda_g \circ (\id_G \times_Nf)[hn,x] = [hng,g^{-1}f(x)]$, and hence $(\lambda_g \circ (\id_G \times_Nf))^{-1}(N \times_N X) = g^{-1}N \times_N X$. 
    
    Second, we determine the right $\Gamma$-action on $\widehat{\mathrm{HH}}{}_N^n(X,Y)$. 
    Let $P_{X}$ and $P_Y$ be projective resolutions of $R\Gamma_c\mathbb{C}_{X}$ and $R\Gamma_c \mathbb{C}_Y$ as complexes of $\mathbb{C}[N]$-modules.
    By \cite[Example 6.5.3 (a)]{Raven}, the $G$-actions give rise to $[\nu_g] \in \End_{\mathbf{K}^+(\mathbb{C}[N])} (P_Y)$ and $[\tau_g] \in \End_{\mathbf{K}^+(\mathbb{C}[N])} (P_{X})$ respectively. 
    Since the right $\Gamma$-action $\lambda$ on $G \times_N X$ is given by $\lambda_g[hn,x]=[hng,g^{-1}x]$, by regarding $\mathbb{C}[G] \otimes_{\mathbb{C}[N]} P_{X}$ as a projective resolution of $R\Gamma_c\mathbb{C}_{G \times_N X}$, we have $\widehat{\mathrm{HH}}{}_G^*(\lambda_g,Y) = ([R_g] \otimes [\tau_{g^{-1}}] ) \circ \cdot $, where $[R_g](\delta_h) = \delta_{hg}$. 
    For $\varphi \in \Hom_{\mathbf{K}^+(\mathbb{C}[N])}(P_{X} ,P_Y)$, we have
    \begin{align*}
    g \cdot \varphi \coloneqq {}&{} r_G^N \circ \widehat{\mathrm{HH}}{}_G^*(\lambda_g,Y) \circ i_N^G (\varphi) = \eta_{X \setminus A } \circ ([R_g]\otimes [\tau_{g^{-1}}]) \circ (\id_{\mathbb{C}[G]} \otimes \varphi)  \circ \mu_{Y}
    \end{align*}
    with respect to Raven's notation (Remark \ref{rmk:Raven.opposite}). For $s \in P_X^m$, we have
    \begin{align*}
    {}&{}(\eta_{X \setminus A } \circ ([R_g] \otimes [\tau_{g^{-1}}]) \circ (\id_{\mathbb{C}[G]} \otimes \varphi ) \circ \mu_{Y})(s) \\
    = {}&{} (([R_g] \otimes [\tau_{g^{-1}}]) \circ (\id \otimes \varphi) \circ \mu_Y)(\delta_e \otimes s) =  ((\id \otimes \varphi) \circ \mu_Y)(\delta_{g} \otimes [\tau_{g^{-1}}]s)\\
    ={}&{}  \mu_{Y}(\delta_{g} \otimes (\varphi \circ [\tau_{g^{-1}}](s))) = ([\tau_{g^{-1}}]\circ \varphi \circ [\nu_g]) (s),
\end{align*}
    This shows that $g \cdot \varphi = [\tau_{g^{-1}}]\circ \varphi \circ [\nu_g]$. 

\end{remark}

\begin{lemma}\label{lem:LHS.homology}
    The $E^2$-page $\widehat{\mathrm{EH}}{}^2_{pq}$of the spectral sequence  associated to the exact couple $(\widehat{\mathrm{DH}}{}_{pq}^1 , \widehat{\mathrm{EH}}{}_{pq}^1)$ is isomorphic to $ \H_p(\Gamma, \widehat{\mathrm{HH}}{}^{-q}_N(\underline{E}G,Y))$.
\end{lemma}
\begin{proof}
    Let $K_p\coloneqq \{ (t_0,\cdots, t_p) \in \Delta_p : t_i \geq 1/(2p+2) \text{ for any $i$} \} \cong \Delta_p$ and let $U_p \coloneqq E_{p}\Gamma \setminus (\Gamma^{p+1} \times K_p)$.  
    By \eqref{eqn:EG.cell}, we have that $\overline{U}_p$ is homotopy equivalent to $E_{p-1}\Gamma$.
    By the excision axiom for the pair $(E_p\Gamma \times \underline{E}G ,  \overline{U}_p \times \underline{E}G)$ and an open subset $ U_p \times \underline{E}G \subset \overline{U}_p \times \underline{E}G$, we have that
    \begin{align}
    \begin{split}
    \widehat{\mathrm{EH}}{}^1_{pq} \cong {}&  \widehat{\mathrm{HH}}{}^{-p-q}_{G} (\underline{E}G \times \Gamma^{p+1} \times (K_p, \partial K_p), Y )\\
    \cong {}& \widehat{\mathrm{HH}}{}^{-q}_{N \rtimes \Gamma} (\underline{E}G \times \Gamma^{p+1}, Y )\\
    \cong {}& \widehat{\mathrm{HH}}{}^{-q}_{N} (\underline{E}G, Y )  \otimes_{\mathbb{C}[\Gamma]}  \mathbb{C}[\Gamma^{p+1}] .
    \end{split}\label{eqn:E1page.tensor}
    \end{align}
    The derivation $d^1_{pq} \colon \widehat{\mathrm{EH}}{}^1_{p,q} \to \widehat{\mathrm{EH}}{}^1_{p-1,q}$ is the boundary map $\delta$ of the long exact sequence for the triple $(E_{p}\Gamma,E_{p-1}\Gamma, E_{p-2}\Gamma)$ (see e.g. \cite[Proposition II.3.2]{Rudyak}). 
    By definition of $E_n\Gamma$, the following diagram commutes;
    \[
    \xymatrix@C=8ex{
    \widehat{\mathrm{HH}}{}^{-p-q}_{G } (\underline{E}G \times  (E_p\Gamma , E_{p-1}\Gamma) , Y) \ar[r]^\delta  \ar[d]^\cong  & \widehat{\mathrm{HH}}{}^{-p-q+1}_{G } (\underline{E}G \times  (E_{p-1}\Gamma ,  E_{p-2}\Gamma) , Y) \ar[d]^\cong \\
    \widehat{\mathrm{HH}}{}^{-q}_{G } (\underline{E}G \times \Gamma^{p+1}, Y) \ar[r]^{\sum_i (-1)^i (({d}_i \times \id_{\underline{E}G})_*) } \ar[d]^\cong  & \widehat{\mathrm{HH}}{}^{-q}_{G } (\underline{E}G \times \Gamma^{p}, Y) \ar[d]^\cong \\
    \widehat{\mathrm{HH}}{}^{-q}_{N} (\underline{E}G, Y) \otimes_{\mathbb{C}[\Gamma]} \mathbb{C}[\Gamma^{p+1}]  \ar[r]^{\sum_i (-1)^i ({d}_i)_* \otimes \id) } & \widehat{\mathrm{HH}}{}^{-q}_{N} (\underline{E}G, Y)  \otimes_{\mathbb{C}[\Gamma]} \mathbb{C}[\Gamma^{p+1}] .
    }
    \]
    Indeed, by a standard argument in cellular homology theory (see for example \cite[Subsection 2.2]{Hatcher}),  the restriction of the boundary map to the each component, namely 
    \begin{align*}
    \delta \colon \widehat{\mathrm{HH}}{}_G^{-p-q}(\underline{E}G \times{}&{} (\Gamma \cdot (g_0,\cdots,g_n)) \times (\Delta_{p},\partial \Delta_p), Y) \\
    {}&{}\to \widehat{\mathrm{HH}}{}_G^{-p-q+1}(\underline{E}G \times (\Gamma \cdot (g_0,\cdots,\hat{g_i},\cdots,g_n)) \times (\Delta_{p-1}, \partial \Delta_{p-1}), Y),
    \end{align*}
    is given by $(-1)^i$ under a canonical identification of the domain and the range with $\widehat{\mathrm{HH}}{}_G^{-q}(\underline{E}G,Y)$, which verifies the commutativity of the top square. 
    The horizontal map below is nothing else than the differential of the group chain complex with coefficient in $\widehat{\mathrm{HH}}{}^{-q}_{N} (\underline{E}G, Y)$. 
\end{proof}

\begin{remark}\label{rmk:DW.free}
    Apply Lemma \ref{lem:LHS.homology} to the case where $N$ is trivial and $Y$ is totally disconnected, in which we have $\widehat{\mathrm{HH}}{}_\Gamma^* = \mathrm{HH}_\Gamma^*$. In this case, since $EG=E\Gamma$ is non-equivariantly contractible, the $E^1$-page concentrates on the $q=0$ row and $\widehat{\mathrm{EH}}{}^1_{p0} \cong \mathbb{C}Y \otimes_{\mathbb{C}[\Gamma]} \mathbb{C}[\Gamma^{p+1}]$. Thus, the spectral sequence collapses at the $E^2$-page, which yields $\mathrm{HH}^{-p}_\Gamma(E\Gamma , Y) \cong \mathrm{H}_p(\Gamma, \mathbb{C}Y)$. This proves the isomorphism \eqref{eqn:Raven.BC} in this case. This is in fact a simplification of the discussion in \cite[Appendix A]{DW} (in particular Lemma A.9) in a restricted setting. Viewing the isomorphism in this way makes the following two facts transparent. First, the isomorphism in \eqref{eqn:Raven.BC} is natural with respect to the second variable $Y$. That is, for a $\Gamma$-map $f \colon Y_1 \to Y_2$, we have
    \begin{align*}
    \xymatrix{
        \mathrm{HH}^{-*}_\Gamma (E\Gamma ,Y_2) \ar[r]^{\text{\eqref{eqn:Raven.BC}}} \ar[d]^{\mathrm{HH}^{-*}_\Gamma (E\Gamma, f)} & \mathrm{H}_*(\Gamma, \mathbb{C}Y_2) \ar[d]^{\mathrm{H}_*(\id_\Gamma, f^*)} \\
        \mathrm{HH}^{-*}_\Gamma (E\Gamma , Y_1) \ar[r]^{\text{\eqref{eqn:Raven.BC}}} & \mathrm{H}_*(\Gamma, \mathbb{C}Y_1). 
        }
    \end{align*}

Second, the above construction of the isomorphism \eqref{eqn:Raven.BC} enables us to compare the $\Gamma$-action considered in Remark \ref{rmk:action} with a natural one on $\mathrm{H}_*(N,\mathbb{C}Y)$.
    Here we assume that $N$ is torsion-free, $G\coloneqq N \rtimes \Gamma$, and $Y$ is a totally disconnected $N \rtimes \Gamma$-space. 
    Consider the modified version of the above spectral sequence, given by $\widehat{\mathrm{DH}}{}_{pq}^1 \coloneqq \widehat{\mathrm{HH}}{}_N^{-p-q}(\underline{E}G \times E_pN , Y)$ and $\widehat{\mathrm{EH}}{}^1_{pq} \coloneqq \widehat{\mathrm{HH}}{}_N^{-p-q}(\underline{E}G \times (E_pN , E_{p-1}N) , Y)$, which converges to $\widehat{\mathrm{HH}}{}_N^{-*}(\underline{E}G \times EN, Y) \cong \widehat{\mathrm{HH}}{}_N^{-*}(EN,Y)$. 
    By letting $\Gamma$ act on $EN$ by a canonical action that preserves each skeleton, and on $\underline{E}G \times EN$ diagonally, we obtain a $\Gamma$-action on this exact couple. Since $\underline{E}G$ is non-equivariantly contractible, the $E^1$-page concentrates on the $q=0$ row and, by Lemma \ref{lem:KK.HH.LHS}, 
    \[
        \widehat{\mathrm{EH}}{}^1_{p0} \cong \widehat{\mathrm{HH}}{}^0_N(N^{p+1} \times \underline{E}G,Y) \cong \widehat{\mathrm{HH}}{}^{0}(\underline{E}G,Y) \otimes_{\mathbb{C}[N]}\mathbb{C}[N^{p+1}] \cong \mathbb{C}Y \otimes_{\mathbb{C}[N]}\mathbb{C}[N^{p+1}].
    \]
    Therefore, the $\Gamma$-action on $\widehat{\mathrm{HH}}{}_\Gamma^{-*}(\underline{E}G\times EN,Y)$ is induced from that on $\mathbb{C}Y \otimes_{\mathbb{C}[N]}\mathbb{C}[N^{*+1}]$. We now make this story more explicit.
    
    Let the $\Gamma$-action on $\mathbf{E}_pN = N^{p+1}$ be given by $\sigma_g(n_0,\cdots,n_p) \coloneqq (\sigma_g(n_0),\cdots,\sigma_g(n_p))$. This action is compatible with the structure maps of the simplicial set $\mathbf{E}N$. Consequently, it gives rise to a continuous action $\sigma \colon \Gamma \curvearrowright EN$ that preserves each skeleton $E_pN$. 
    Let $\tau$ denote the $G$-action on $\underline{E}G$. 
    By Remark \ref{rmk:action}, $\Gamma$ acts on this $E^1$-page by $[(\tau_g \times \sigma_g)^{-1}] \circ \cdot \circ [\alpha_g]$, which is identified with $[\sigma_g^{-1}] \circ \cdot \circ [\alpha_g]$ acting on $\widehat{\mathrm{HH}}{}_N^0(N^{p+1},Y)$ via the $N$-equivariant map $\pr_{N^{p+1}} \colon N^{p+1} \times \underline{E}G \to N^{p+1}$. Through the above isomorphism $\widehat{\mathrm{HH}}{}^0_N(N^{p+1} ,Y) \cong \mathbb{C}Y \otimes_{\mathbb{C}[N]}\mathbb{C}[N^{p+1}]$, this action corresponds to  
    \[
        g \cdot (\xi \otimes \delta_{n_0,\cdots, n_p}) = \alpha_{g}^* (\xi) \otimes \delta_{\sigma_{g^{-1}}(n_0), \cdots, \sigma_{g^{-1}}(n_p)}. 
    \]
    This action agrees with the $\Gamma$-action naturally induced on the standard resolution $C_\bullet(N,\mathbb{C}Y)$, namely the action used to define $\mathrm{H}_*(\sigma_g,\alpha_g)$ considered in Theorem \ref{thm1}. 
\end{remark}

\begin{remark}\label{rmk:action.commute.KK}
    We also prove that the $\Gamma$-action on $\mathrm{KK}^*_N(\underline{E}G,Y)$ is compatible with the canonical action on $\K_*(C^*_r(N \ltimes Y))$, induced by $\nu_g^* \rtimes \Ad(g^{-1}) \colon C_0(Y) \rtimes N \to C_0(Y) \rtimes N$, via the Baum--Connes isomorphism $\mu_N$. Notice that the map $[M,f,\xi] \mapsto [{}_{g^{-1}}M,\tau_{g^{-1}} \circ f , (\id _M\times \nu_g)^*\xi] $ is decomposed into the following three steps; 
    \[
    \mathrm{KK}_*^N(X,Y) \to \mathrm{KK}_*^N({}_{g^{-1}}X,{}_{g^{-1}}Y) \xrightarrow{\mathrm{KK}_*^N({}_{g^{-1}}X,\nu_g)} \mathrm{KK}_*^N({}_{g^{-1}}X,Y) \xrightarrow{\mathrm{KK}_*^N(\tau_g^{-1},Y)} \mathrm{KK}_*^N(X,Y),
    \]
    where the first map is given by $[M,f,\xi] \mapsto [{}_{g^{-1}}M,f,\xi]$. Now, the following diagram commutes; 
    \[
    \xymatrix@C=10ex{
    \mathrm{KK}_*^N(X,Y) \ar[r] \ar[d]^{\mu_N} & \mathrm{KK}_*^N({}_{g^{-1}}X,{}_{g^{-1}}Y) \ar[r]^{\mathrm{KK}_*^N({}_{g^{-1}}X,\nu_g)} \ar[d]^{\mu_N}& \mathrm{KK}_*^N({}_{g^{-1}}X,Y) \ar[r]^{\mathrm{KK}_*^N(\tau_g^{-1},Y)} \ar[d]^{\mu_N}& \mathrm{KK}_*^N(X,Y)\ar[d]^{\mu_N} \\
    \mathrm{K}_*(C^*_r(N \ltimes Y)) \ar[r] & \mathrm{K}_*(C^*_r(N \ltimes {}_{g^{-1}}Y)) \ar[r]^{(\nu_g^* \rtimes \id_N)_* } & \mathrm{K}_*(C^*_r(N \ltimes Y)) \ar@{=}[r] & \mathrm{K}_*(C^*_r(N \ltimes Y)).
    }    
    \]
    The commutativity of the middle and the right squares follows from basic properties of the assembly map $\mu_N$ (\cite[Proposition 5.3]{KasSka} for the middle, and \cite[pp.\ 178--179]{KasSka} for the right). For the left diagram, notice that the domain group $\mathrm{KK}^N_*(X,Y)$, the target group $\mathrm{K}_*(C^*_r(N\ltimes Y))$, and the assembly map $\mu_N$ between them depend only on the isomorphism class of the defining data $(N,X,Y)$. The first and second columns from the left correspond via the isomorphism of defining data
    \[
        (\Ad(g^{-1}),\id,\id)\colon (N,X,Y)\cong (N,{}_{g^{-1}}X,{}_{g^{-1}}Y).
    \]
    Therefore, the commutativity of the left diagram is a tautology. Moreover, since the two left horizontal maps are explicitly given by $[M,f,\xi]\mapsto[{}_{g^{-1}}M,f,\xi]$ and $(\id\rtimes \Ad(g^{-1}))\colon C_0(Y)\rtimes N\to C_0(Y)\rtimes N$, the two horizontal compositions give the action of $g \in \Gamma$ on $\mathrm{KK}_*^N(X,Y)$ and $\mathrm{K}_*(C^*_r(N \ltimes Y))$ respectively.
\end{remark}

\subsubsection{K-theory computation: torsion-free case}
Let us go back to the example coming from number fields. Let $R$ be the ring of algebraic integers in a number field $K$.
We again assume that $A = R$ and $\Zz_{>0} \subseteq S \subseteq R^\times$ is an abelian submonoid. 
Recall that $\mathscr{A} \coloneqq \varinjlim_{s \in S} \{ A,\sigma_s\}$ and $\mathscr{S} \coloneqq S^{-1}S$ as are introduced in Subsection \ref{sec:algactions}. 
We first work in the case that $S$ is, and hence $\mathscr{S}$ is, torsion-free. 
In this case, Raven's localized Chern character $\mathop{\mathrm{ch}_R^e}$ is an isomorphism after tensoring with $\mathbb{C}$.

\begin{theorem}[{\cite[Section 5]{CL}}]\label{theorem:Ktheory.CuntzLi}
    The K-theory of $C^*_r((\mathscr{A} \rtimes \mathscr{S}) \ltimes \overline{\mathscr{A}})$ is isomorphic to
    \[
    \K_*(C^*_r((\mathscr{A} \rtimes \mathscr{S}) \ltimes \overline{\mathscr{A}})) \cong 
        \begin{cases}
            \bigoplus_{k \in \mathbb{Z}}\lwedge_{\mathbb{Z}}^{*+2k} \mathscr{S}, & \text{if $\mathrm{sign}(S) = \{ 1\}$}, \\
            \bigoplus_{k \in \mathbb{Z}}\lwedge_{\mathbb{Z}}^{*+2k} \mathscr{S}^\perp \otimes_{\mathbb{Z}} (\mathbb{Z}/2\mathbb{Z}), & \text{if $\mathrm{sign}(S) = \{ \pm 1\}$}. 
        \end{cases}
    \]
    Here, in the latter case, $\mathscr{S}^\perp$ is a subgroup of $\mathscr{S}$ such that $\mathrm{sign}(\mathscr{S}^\perp)=\{1\}$ and $\mathscr{S}/\mathscr{S}^\perp \cong \mathbb{Z}$. 
\end{theorem}
Indeed, in \cite[Section 5]{CL}, the K-group computation is divided into 4 cases; (1) $|V_{K,\mathbb{R}}| =0$, (2) $|V_{K,\mathbb{R}}|$ odd and $|\{w \in V_{K,\mathbb{R}} : w (b)<0\}|$ even for any $ b \in \mathscr{S}$, (3) $|V_{K,\mathbb{R}}|$ odd and $|\{w \in V_{K,\mathbb{R}} : w(b)<0\}|$ odd for some $ b \in \mathscr{S}$, and (4) $|V_{K,\mathbb{R}}| \geq 2$ even. By Lemma~\ref{lem:sign1}, $\sign(\mathscr{S}) = \{ +1\}$ in cases (1), (2) and $\sign(\mathscr{S}) = \{ \pm 1\}$ in cases (3), (4).

\begin{lemma}\label{thm:EK.E2.collapse}
    The spectral sequence $\mathrm{EK}_{pq}^r \otimes \mathbb{C}$ collapses at the $E^2$-page. 
\end{lemma}
\begin{proof}
    Recall Remark \ref{rmk:DW.free}, in which it is proved that the action $\mathscr{S} \curvearrowright \mathrm{HH}_{\mathscr{A}}^{*}(\underline{E}(\mathscr{A} \times \mathscr{S}) , \overline{\mathscr{A}})$ given in Lemma \ref{lem:KK.HH.LHS} agrees with that on $\mathrm{H}_*(\mathscr{A} \rtimes \overline{\mathscr{A}} ;\mathbb{C})$ considered in Theorem \ref{thm1}, under the identification $\mathrm{H}_*(\mathscr{A} \rtimes \overline{\mathscr{A}} ;\mathbb{C}) \cong \mathrm{H}_*(\mathscr{A} ; \mathbb{C}\overline{\mathscr{A}})$ (Proposition \ref{prop:isomfunctors}). By Proposition \ref{prop:group.homology.LHS} and Lemma \ref{lem:LHS.homology}, the spectral sequence $\widehat{\mathrm{EH}}{}_{pq}^r$ collapses at the $E^2$-page. Since Raven's Chern character \eqref{eqn:comparison.LHS} induces a morphism of spectral sequence, a diagram chasing argument shows that this collapsing inherits to $\mathrm{EK}_{pq}^r \otimes \Cz$.
\end{proof}

\begin{lemma}\label{lem:Raven.Chern.torus}
Raven's localized Chern character
\[ 
    \mathop{\mathrm{ch}_R^e} \colon \K_q(C^*_r( \mathscr{A} \ltimes \overline{\mathscr{A}} )) \cong \KK^{\mathscr{A}}_q(\underline{E}\mathscr{A},\overline{\mathscr{A}})  \to \bigoplus_{k \in \mathbb{Z}}\mathrm{HH}_{\mathscr{A}}^{-q+2k}(\underline{E}\mathscr{A},\overline{\mathscr{A}}) \cong \bigoplus_{k \in \mathbb{Z}}\H_{q+2k}(\mathscr{A} \ltimes \overline{\mathscr{A}} ;\mathbb{C})
    \]
is injective and has the image $\bigoplus_{k\in\Zz}\H_{q + 2k}(\mathscr{A} \ltimes \overline{\mathscr{A}} ;\mathbb{Z}) $. 
\end{lemma}
\begin{proof}
    First, by $\mathscr{A} + \overline{A} = \overline{\mathscr{A}}$ and $\mathscr{A} \cap \overline{A} =A$ (Lemmas \ref{lem:scrAcapbarA=A} and \ref{lem:scrA+barA=barscrA}), we obtain an isomorphism of $\mathscr{A}$-spaces $\mathscr{A} \times_A \overline{A} \to \overline{\mathscr{A}}$ given by the addition. Now we can apply Lemma \ref{lem:induction.2} to get the commutative diagram 
    \[
    \xymatrix{
    \mathrm{KK}_q^{\mathscr{A}}(\underline{E}\mathscr{A},\overline{\mathscr{A}})
    \ar[r]^{\mathrm{ch}_R^e}  \ar[d]^\cong_{\text{Lemma \ref{lem:induction.2}}} & 
    \bigoplus_k \mathrm{HH}_{\mathscr{A}}^{-q+2k}(\underline{E}\mathscr{A}, \overline{\mathscr{A}}) \ar[d]^\cong_{\text{Lemma \ref{lem:induction.2}}}  \\
\mathrm{KK}_q^{A}(\underline{E}\mathscr{A},\overline{A})
    \ar[r]^{\mathrm{ch}_R^e}  & 
    \bigoplus_k \mathrm{HH}_{A}^{-q+2k}(\underline{E}\mathscr{A},\overline{A})  .
    }
    \]
    Since $\underline{E}\mathscr{A}$ models $\underline{E}A$ as $A$-space (Remark \ref{rmk:EG} (iv)), this identifies Raven's Chern character under consideration with $\mathop{\mathrm{ch}_R^e} \colon \KK^{A}_q(\underline{E}A,\overline{A}) \to \bigoplus_{k \in \mathbb{Z}}\mathrm{HH}_{A}^{-q+2k}(\underline{E}A,\overline{A})$. Note that $\mathbb{R}^d$ models $\underline{E}A$ for $A\cong \mathbb{Z}^d$. Next, by the commutativity of the diagrams
    \begin{gather*}
    \xymatrix{
    \varinjlim_s \KK^A_*(\underline{E}A,A/\sigma_s(A)) \ar[r]^\cong \ar[d] & \varinjlim_s \K_*(C(A/\sigma_s(A)) \rtimes A) \ar[d]^\cong \\
    \KK^A_*(\underline{E}A,\overline{A}) \ar[r]^\cong & \K_*(C(\overline{A}) \rtimes A)
    }
    \\ 
    \xymatrix{
    \varinjlim_s \mathrm{KK}_*^A(\underline{E}A,A/\sigma_s(A)) \ar[r]^{\mathrm{ch}_R^e} \ar[d] & \varinjlim_s \bigoplus_{k}\mathrm{HH}^{-*+2k}_A(\underline{E}A,A/\sigma_s(A)) \ar[d] \\
    \mathrm{KK}_*^A(\underline{E}A,\overline{A}) \ar[r]^{\mathrm{ch}_R^e} &\bigoplus_k \mathrm{HH}^{-*+2k}_A(\underline{E}A,\overline{A}),
    }    \\ 
    \xymatrix{
    \varinjlim_s \mathrm{HH}^{-*}_A(\underline{E}A,A/\sigma_s(A)) \ar[r]^\cong \ar[d] & \varinjlim_s \H_*(A,\mathbb{C}[A/\sigma_s(A)]) \ar[d]^\cong \\
    \mathrm{HH}^{-*}_A(\underline{E}A,\overline{A}) \ar[r]^\cong & \H_*(A,\mathbb{C}[\overline{A}]),
    }
    \end{gather*}
    each of which is given in \cite[Proposition 5.3]{KasSka}, Remark \ref{rmk:Raven.natural.cohomological}, and Remark \ref{rmk:DW.free} respectively, it turns out that Raven's Chern character is identified with the inductive limit of  
    \[ 
    \mathop{\mathrm{ch}_R^e} \colon \KK^{A}_q(\underline{E}A,A/\sigma_s(A)) \to  \bigoplus_{k \in \mathbb{Z}} \mathrm{HH}_{A}^{-q+2k}(\underline{E}A,A/\sigma_s(A)).
    \]
    Again by Lemma \ref{lem:induction.2}, there is a commutative diagram
    \[
    \xymatrix{
    \KK^{A}_*(\underline{E}A,A/\sigma_s(A)) \ar[r]^{\mathrm{ch}^e_R} \ar[d]^\cong_{\text{Lemma \ref{lem:induction.2}}} & \bigoplus_k \mathrm{HH}^{*+2k}_A(\underline{E}A, A/\sigma_s(A)) \ar[d]^\cong_{\text{Lemma \ref{lem:induction.2}}} &\\
    \KK^{\sigma_s(A)}_*(\underline{E}A, \mathrm{pt}) \ar[r]^{\mathrm{ch}^e_R} & 
    \bigoplus_k \mathrm{HH}^{*+2k}_{\sigma_s(A)}(\underline{E}A, \mathrm{pt}) .
    }
    \]
    The above discussion, and the fact that $\underline{E}A$ models $\underline{E}\sigma_s(A)$, reduces the problem to proving that 
    \[
    \Im \Big( \mathop{\mathrm{ch}_R^e} \colon \mathrm{KK}^{N}_*(\underline{E}N,\mathrm{pt}) \to \bigoplus_k \mathrm{HH}_{N}^{*+2k}(\underline{E}N , \mathrm{pt})\Big)  =  \bigoplus_k \mathrm{H}_{-*+2k}(N , \mathbb{Z})
    \]
    via the isomorphism \eqref{eqn:Raven.BC} when $N=\mathbb{Z}^d$. 
    To see this, as in Remark \ref{rmk:DW.free}, consider the spectral sequence \eqref{eqn:exact.couple.LHS}, for $\Gamma = N$ and trivial normal subgroup, for both KK- and HH-theories, and compare them via Raven's Chern character. Then, by Lemma \ref{lem:KK.HH.LHS}, the image of $\mathop{\mathrm{ch}_R^e} \colon \mathrm{EK}{}^1_{p0} \to \widehat{\mathrm{EH}}{}_{p0}^1$ is identified with 
    \[
    \mathbb{Z} \otimes_{\mathbb{Z}[N]}\mathbb{Z}[N^{p+1}] \subset \mathbb{C} \otimes_{\mathbb{C}[N]} \mathbb{C}[N^{p+1}],
    \]
    which agrees with the standard chain complex $C_{\bullet}(N, \mathbb{Z})$ computing the group homology of $N$. By Lemma \ref{lem:LHS.homology}, which states that the $d^1$-differential also agrees with the standard bar differential, we also obtain $\Im (\mathop{\mathrm{ch}_R^e} \colon \mathrm{EK}{}^2_{p0} \to \widehat{\mathrm{EH}}{}^2_{p0}) = \mathrm{H}_p(N,\mathbb{Z})$ as subgroups of $\widehat{\mathrm{EH}}{}^2_{p0} \cong \mathrm{H}_p(N,\mathbb{C})$. Since both $\mathrm{EK}{}^1_{p0} $ and $ \widehat{\mathrm{EH}}{}_{p0}^1$ concentrates on the $q=0$ row, these spectral sequences collapse at the $E^2$ page and no extension issues arise. Finally, we obtain that
    \[
    \Im (\mathop{\mathrm{ch}_R^e} \colon \mathrm{EK}{}^2_{p0} \to \widehat{\mathrm{EH}}{}^2_{p0}) = \Im (\mathop{\mathrm{ch}_R^e} \colon \mathrm{EK}{}^\infty_{p0} \to \widehat{\mathrm{EH}}{}^\infty_{p0}) = \Im \Big( \mathop{\mathrm{ch}_R^e} \colon \mathrm{KK}{}^{N}_p(\underline{E}N,\mathrm{pt}) \to \bigoplus_k\widehat{\mathrm{HH}}{}^{-p+2k}_{N}(\underline{E}N,\mathrm{pt})\Big) 
    \]
    is canonically identified with $\bigoplus_k \mathrm{H}_{p+2k}(N,\mathbb{Z})$ as desired.
\end{proof}

\begin{lemma}\label{lem:HK.even}
If $\sign (\mathscr{S}) = \{ +1\}$, then there is an isomorphism
\[ 
    \K_n(C^*_r((\mathscr{A} \rtimes \mathscr{S}) \ltimes \overline{\mathscr{A}})) \cong \bigoplus_{k \in \mathbb{Z}}\lwedge^{n-d+2k}_{\mathbb{Z}} \mathscr{S}.
\]
\end{lemma}
\begin{proof}
By Lemma \ref{lem:Raven.Chern.torus} and Proposition \ref{prop:group.homology.LHS}, the $E^2$-page $\mathrm{EK}_{pq}^2$ is determined as 
\begin{align*}
    \mathrm{EK}^2_{pq} \cong{}& \H_p(\mathscr{S} ; \K_{q}(C_0(\overline{\mathscr{A}}) \rtimes \mathscr{A})) \cong \H_p\Big( \mathscr{S} ; \bigoplus_{k \in \mathbb{Z}}\H_{q+2k}(\mathscr{A} \ltimes \overline{\mathscr{A}} ;\mathbb{Z}) \Big)\\
    \cong {}&
\begin{cases}
    \H_p(\mathscr{S}; \Zz_{\sign}) & \text{ if $q-d$ is even,}\\
    0 & \text{ if $q-d$ is odd.}
\end{cases} 
\end{align*}
By the assumption of $\sign(S) = \{+1\}$, the group $\mathscr{S}$ acts on the coefficient $\Zz_\sign $ trivially, and hence the right hand side becomes the free abelian group $\H_{p}(\mathscr{S};\Zz) \cong \lwedge^p \mathscr{S}$. 
This implies that the map of spectral sequences $\mathrm{EK}^2_{pq} \to \mathrm{EK}^2_{pq} \otimes \mathbb{C}$ is injective. 
Since the spectral sequence $\mathrm{EK}^r_{pq} \otimes \mathbb{C}$ collapses at the $E^2$-page by Theorem \ref{thm:EK.E2.collapse}, an iterated diagram chase argument shows that $d^r_{pq} =0$ for $\mathrm{EK}^r_{pq}$. 
Moreover, since each component $\mathrm{EK}_{pq}^2$ is a free $\mathbb{Z}$-module, there is no extension problem. This shows the desired isomorphism
\[ \K_n(C^*_r(\mathscr{A} \rtimes \mathscr{S}) \ltimes \overline{\mathscr{A}}) \cong \bigoplus_{p+q=n} \mathrm{EK}_{pq}^2 \cong  \bigoplus_{k \in \mathbb{Z}}\lwedge^{n-d+2k} \mathscr{S}. \qedhere  \]
\end{proof}

\begin{proof}[Proof of Theorem \ref{theorem:Ktheory.CuntzLi}]
The remaining part is the same as \cite{CL}.
The case of $\sign (S) =\{ +1 \}$ is already proved in Lemma \ref{lem:HK.even}.
In the case of $\sign (S) =\{ \pm 1 \}$, pick a basis $\{ a_1, a_2, \cdots, \}$ of $\mathscr{S}$ in the way that $\sign (a_1) = -1$ and $\mathscr{S}^\perp\coloneqq \langle a_2, a_3, \cdots \rangle \subset \mathscr{S}$. Set $\mathscr{S}'\coloneqq \langle a_1^2, a_2,\cdots \rangle $. 
Apply Lemma \ref{lem:HK.even} to the dynamical system $(\mathscr{A} \rtimes \mathscr{S}^\perp) \ltimes \overline{\mathscr{A}}$ and$ (\mathscr{A} \rtimes \mathscr{S}') \ltimes \overline{\mathscr{A}}$. 
Then we obtain that 
\begin{align*} 
    \K_n(C^*_r((\mathscr{A} \rtimes \mathscr{S}^\perp) \ltimes \overline{\mathscr{A}})) &{}  \cong \bigoplus_{k\in \Zz}\H_{n-d+2k}(\mathscr{S}^\perp ; \Zz), \\
        \K_n(C^*_r((\mathscr{A} \rtimes \mathscr{S}') \ltimes \overline{\mathscr{A}})) &{}  \cong \bigoplus_{k \in \Zz}\H_{n-d+2k}(\mathscr{S}' ; \Zz).
\end{align*} 
In particular, since the morphism $\bigoplus_{k\in \Zz}\H_{n-d+2k}(\mathscr{S}^\perp ; \Zz) \to \bigoplus_{k\in \Zz}\H_{n-d+2k}(\mathscr{S}' ; \Zz)$ is injective, the morphism $\K_n(C^*_r((\mathscr{A} \rtimes \mathscr{S}^\perp) \ltimes \overline{\mathscr{A}})) \to \K_n(C^*_r((\mathscr{A} \rtimes \mathscr{S}') \ltimes \overline{\mathscr{A}}))$ induced from the inclusion is injective (recall that a filtered map inducing the injective map of the subquotients is itself injective). Now, the Pimsner--Voiculescu exact sequence for the restriction of  $\alpha \colon a_1^\Zz \curvearrowright C^*_r((\mathscr{A} \rtimes \mathscr{S}^\perp) \ltimes \overline{\mathscr{A}}))$ to $a_1^{2 \mathbb{Z}}$ is 
\[
    \xymatrix{
    \K_0(C^*_r((\mathscr{A} \rtimes \mathscr{S}^\perp) \ltimes \overline{\mathscr{A}})) \ar[r]^{\id - \alpha^2_* } & \K_n(C^*_r((\mathscr{A} \rtimes \mathscr{S}^\perp) \ltimes \overline{\mathscr{A}})) \ar@{^{(}->}[r] & \K_n(C^*_r((\mathscr{A} \rtimes \mathscr{S}') \ltimes \overline{\mathscr{A}})) \ar[d] \\
    \K_1(C^*_r((\mathscr{A} \rtimes \mathscr{S}') \ltimes \overline{\mathscr{A}})) \ar[u] & \K_n(C^*_r((\mathscr{A} \rtimes \mathscr{S}^\perp)) \ltimes \overline{\mathscr{A}}) \ar@{_{(}->}[l] & \K_n(C^*_r((\mathscr{A} \rtimes \mathscr{S}^\perp)) \ltimes \overline{\mathscr{A}}) \ar[l]_{\id -\alpha^2_*}.
    }
\]
Hence $\id - \alpha^2_* =0$, that is, $a_1^2$ acts on $\K_*(C^*_r((\mathscr{A} \rtimes \mathscr{S}^\perp) \ltimes \overline{\mathscr{A}}))$ by the identity.  In other words, the action of $a_1^\Zz$ on $\K_*(C^*_r((\mathscr{A} \rtimes \mathscr{S}^\perp) \ltimes \overline{\mathscr{A}}))$ factors through $\mathbb{Z}/2\mathbb{Z}$. 
Moreover, the isomorphism 
\[
\mathrm{Gr}(\K_*(C^*_r((\mathscr{A} \rtimes \mathscr{S}^\perp) \ltimes \overline{\mathscr{A}})) \cong E_{pq}^\infty,
\]
where the left hand side is the graded module associated to the filtration of $\K_*(C^*_r((\mathscr{A} \rtimes \mathscr{S}^\perp) \ltimes \overline{\mathscr{A}}) \cong \KK_*^{\mathscr{A} \rtimes \mathscr{S}^\perp}(\underline{E}(\mathscr{A} \rtimes \mathscr{S}^\perp), \overline{\mathscr{A}})$ coming from the exact couple \eqref{eqn:exact.couple.LHS}, is equivariant under the $a_1^{\mathbb{Z}/2 \mathbb{Z}}$-action, and the right hand side is the direct sum of copies of $\Zz_{\sign}$ as $\mathbb{Z}[\mathbb{Z}/2\mathbb{Z}]$-modules. 
Since any extension of $\Zz_{\sign}$ by $\Zz_{\sign}$ is trivial, i.e., $\mathop{\mathrm{Ext}}^1_{\Zz[\mathbb{Z}/2\mathbb{Z}]}(\Zz_{\sign},\Zz_{\sign}) =0$, we obtain that $\alpha_* =-\id$.
Finally, the Pimsner--Voiculescu exact sequence for  $\alpha \colon a_1^\Zz \curvearrowright C^*_r((\mathscr{A} \rtimes \mathscr{S}^\perp) \ltimes \overline{\mathscr{A}}))$ shows the desired isomorphism. 
\end{proof}

\subsubsection{K-theory computation: torsion case}

Next we consider the case that $\mathscr{S} = \mu \times \Gamma$, where $\Gamma$ is free abelian and $\mu $ is the group of roots of unity. 
Here we revisit the result of Li--L\"uck \cite{LiLuck} from our standpoint.
\begin{theorem}[{\cite[Theorem 1.2]{LiLuck}}]\label{thm:LuckLi}
    If $\mathscr{S}=\mu \times \Gamma$, we have 
    \[ 
    \K_*(C^*_r((\mathscr{A} \rtimes \mathscr{S}) \ltimes \overline{\mathscr{A}})) \cong 
    \begin{cases}
        \bigoplus_{k \in \mathbb{Z}}\K_0(C^*_r(\mu)) \otimes_{\mathbb{Z}} \lwedge^{*+2k} \Gamma & \text{ if $K$ is totally imaginary,}\\
        \bigoplus_{k \in \mathbb{Z}}\lwedge^{*+2k} \Gamma & \text{ if $K$ has a real embedding.}\\
    \end{cases}
    \]
\end{theorem}

In this case, the localized Chern character is no longer rationally isomorphic, but the delocalized part is well controlled.
\begin{lemma}\label{lem:Raven.algaction}
For $G=\mathscr{A} \rtimes \mathscr{S}$ or $\mathscr{A} \rtimes \mu$, we have
\begin{align*}
    \widehat{\mathrm{HH}}{}_{G}^{*} (\underline{E}G, \overline{\mathscr{A}} ) \cong \mathrm{H}_{-*}(G \ltimes \overline{\mathscr{A}} ; \mathbb{C}) \oplus 
    \begin{cases} 
        \mathbb{C}[\mu \setminus \{e\}] & \text{if $*=0$, } \\ 
        0 & \text{otherwise}. 
    \end{cases}
\end{align*}
\end{lemma}
\begin{proof}
By the right equation of \eqref{eqn:Raven.BC}, we have the isomorphism 
\[
\widehat{\mathrm{HH}}{}_{G}^{*} (\underline{E}G, \overline{\mathscr{A}} ) \cong \mathrm{H}_{-*} \Big( G , \mathbb{C}\Big[ \widehat{\overline{\mathscr{A}}} \Big] \Big) \cong \bigoplus_{[\gamma] \in \mathcal{C}(G_{\mathrm{tor}})} \mathrm{H}_{-*}(Z(\gamma), \mathbb{C}[\overline{\mathscr{A}}{}^\gamma]),
\]
where $ \widehat{\overline{\mathscr{A}}} \coloneqq \{ (g,a) \in G_{\mathrm{tor}} \times \overline{\mathscr{A} } :  ga=a \}$. In particular, the `localized' component corresponding to $[e] \in \cC(G_{\mathrm{tor}})$ is isomorphic to the groupoid homology $\mathrm{H}_{-*}(G , \mathbb{C}[\overline{\mathscr{A}}]) \cong \mathrm{H}_{-*}(\overline{\mathscr{A}} \rtimes G ; \mathbb{C})$. 

Any torsion element $g = a \cdot s$ in $\mathscr{A} \rtimes \mu$ or $\mathscr{A} \rtimes \mathscr{S}$ is conjugate to some $\zeta \in \mu$ since the element $\zeta = (-b) \cdot  g \cdot b$, where $b = a(1-s)^{-1} \in \mathscr{A}$, fixes $ 0 \in \mathscr{A}$.  
The centralizer subgroup $Z(\zeta) $ of $\zeta \in \mu \setminus \{e\}$ is $\mu$ since a multipricative group $\mu$ of an integral domain acts on $A \setminus \{0\}$ freely. By the same reason, the fixed point set of the action of $\zeta$ on $\overline{\mathscr{A}}$ is $\overline{\mathscr{A}}{}^\zeta = \{ 0 \}$. 
In conclusion, we have
\begin{align*}
    \bigoplus_{[\gamma] \in \mathcal{C}(G_{\mathrm{tor}}) \setminus \{ [e] \}} \mathrm{H}_{-*}(Z(\gamma), \mathbb{C}[\overline{\mathscr{A}}{}^\gamma]) ={}&{} \bigoplus_{\zeta \in \mu \setminus \{e\}} \mathrm{H}_*(Z(\zeta) , \mathbb{C}[\overline{\mathscr{A}}{}^\zeta] ) 
    \cong  \bigoplus_{\zeta \in \mu \setminus \{ e\} } \mathrm{H}_*(\mu,  \mathbb{C})\\ \cong {}&{} 
    \begin{cases}
        \bigoplus_{\zeta \in \mu \setminus \{e\}}\mathbb{C} & \text{ if $*=0$, } \\ 
        0 & \text{otherwise.} 
    \end{cases}
\end{align*}
This finishes the proof. 
\end{proof}

\begin{lemma}\label{lem:homology.A.mu}
We have
\[
    \mathrm{H}_n((\mathscr{A} \rtimes \mu ) \ltimes \overline{\mathscr{A}} ; \mathbb{C}) \cong   (\lwedge^n_{\mathbb{C}} (A \otimes \mathbb{C}) )_{\mu}.
\]
\end{lemma}
\begin{proof}
    The LHS spectral sequence for the semidirect product $(\mathscr{A}  \ltimes \overline{\mathscr{A}} ) \rtimes \mu$ converges to $\mathrm{H}_n((\mathscr{A} \rtimes \mu ) \ltimes \overline{\mathscr{A}} ; \mathbb{Z})$ and its $E^2$-page is 
    \[ 
    E^2_{pq} \cong \mathrm{H}_p(\mu, \mathrm{H}_q(\mathscr{A}  \ltimes \overline{\mathscr{A}} ;\mathbb{C} )) \cong 
    \begin{cases}
        \mathrm{H}_q(\mathscr{A}  \ltimes \overline{\mathscr{A}} ;\mathbb{C} )_\mu & \text{ if $p=0$}, \\
        0 & \text{ otherwise.}
    \end{cases}
    \]
    This, together with Theorem \ref{thm:rings}, shows the lemma.
\end{proof}

\begin{remark}\label{rmk:homology.A.mu.action}
Lemma \ref{lem:homology.A.mu} also determines the $\Gamma$-action on $\mathrm{H}_n((\mathscr{A} \rtimes \mu) \ltimes \overline{\mathscr{A}} , \mathbb{C}) \cong (\lwedge^n_{\mathbb{C}} (A \otimes \mathbb{C}))_\mu \cong (\lwedge^n_{\mathbb{C}} (A \otimes \mathbb{C}))^\mu$ to be the restriction of the $\mathscr{S}$-action on $\mathrm{H}_n(\mathscr{A}, \mathbb{C}\overline{\mathscr{A}}) \cong \lwedge^n_{\mathbb{C}} (A \otimes \mathbb{C})$ considered in Theorem \ref{thm1}. Via the identification in Remark \ref{rmk:DW.free} (and the proof of Lemma \ref{thm:EK.E2.collapse}), this in turn determines the $\Gamma$-action on $\mathrm{HH}_{\mathscr{A} \rtimes \mu}^*(\underline{E}(\mathscr{A} \rtimes \mu), \overline{\mathscr{A}})$ in the sense of Lemma \ref{lem:KK.HH.LHS}.

Here we again make use of the notions used around Theorem \ref{thm1}, in particular those taken from \cite{Brown}. Consider the morphism
\[
    \mathrm{H}_*(\id_{\mathscr{A} \rtimes \mu} ,1 \otimes \id_{\mathbb{C}\overline{\mathscr{A}}}) \colon \mathrm{H}_*(\mathscr{A} \rtimes \mu  , \mathbb{C}\overline{\mathrm{\mathscr{A}}}) \to \mathrm{H}_* (\mathscr{A} \rtimes \mu, \mathbb{C}\mu \otimes \mathbb{C}\overline{\mathscr{A}})
\]
in group homology, where $\mathscr{A}$ and $\mu$ act on $\mu$ by the trivial and the left regular action respectively, and $1 = \sum_{g \in \mu} \delta_{g} \in \mathbb{C}\mu$. 
By the functoriality of the LHS spectral sequences (Theorem \ref{thm:LHSfunctorial}), this induces the morphism of $E^2$-pages as
\[
    \mathrm{H}_*(\id_\mu, \mathrm{H}_*(\id_{\mathscr{A}} ,1 \otimes \id_{\mathbb{C}\overline{\mathscr{A}}})) \colon \mathrm{H}_p(\mu,\mathrm{H}_q(\mathscr{A}, \mathbb{C} \overline{\mathscr{A}})) \to \mathrm{H}_p(\mu,\mathrm{H}_q(\mathscr{A}, \mathbb{C}\mu \otimes \mathbb{C}\overline{\mathscr{A}} )).
\]
 Since $\mathscr{A}$ acts on $\mathbb{C}\mu$ trivially, we have $\mathrm{H}_q(\mathscr{A}, \mathbb{C}\mu \otimes \mathbb{C}\overline{\mathscr{A}} ) \cong  \mathbb{C}\mu \otimes\mathrm{H}_q(\mathscr{A}, \mathbb{C}\overline{\mathscr{A}} )$ as $\mu$-modules. By Shapiro's lemma, the $E^2$-page of the LHS spectral sequence on the target side is $E^2_{0q} \cong \mathrm{H}_q(\mathscr{A} , \mathbb{C}\overline{\mathscr{A}}) $ and $E^2_{pq} \cong \{0\}$ for $p \neq 0$, and in particular the spectral sequence collapse at the $E^2$-page. Thus, the above morphism at $E^2_{0q}$ is identical with the inclusion 
\[
    \mathrm{H}_q(\mathscr{A}, \mathbb{C}\overline{\mathscr{A}}  )_\mu \hookrightarrow  (\mathbb{C}\mu \otimes \mathrm{H}_q(\mathscr{A}, \mathbb{C}\overline{\mathscr{A}}))_\mu \cong \mathrm{H}_q(\mathscr{A}, \mathbb{C}\overline{\mathscr{A}}) ,
\]
whose image is nothing else than $(\lwedge_\mathbb{C}^q (A \otimes \mathbb{C} ))^\mu$. 
In conclusion, $ \mathrm{H}_*(\mathscr{A} \rtimes \mu  , \mathbb{C}\overline{\mathrm{\mathscr{A}}})$ is identified with the subgroup $\mathrm{H}_*(\mathscr{A},\mathbb{C}\overline{\mathscr{A}})^\mu$ of $\mathrm{H}_* (\mathscr{A} \rtimes \mu, \mathbb{C}\mu \otimes \mathbb{C}\overline{\mathscr{A}}) \cong \lwedge_{\mathbb{C}}^n(A \otimes \mathbb{C})$ via $\mathrm{H}_*(\id_{\mathscr{A} \rtimes \mu} ,1 \otimes \id_{\mathbb{C}\overline{\mathscr{A}}}) $.

The $\mathscr{S}$-action on $\mathrm{H}_*(\mathscr{A}, \mathbb{C}\overline{\mathrm{\mathscr{A}}})$ in Theorem \ref{thm1}, given by $\mathrm{H}_*(\tilde{\sigma}_s,\alpha_s)$, extends to both $ \mathrm{H}_*(\mathscr{A} \rtimes \mu  , \mathbb{C}\overline{\mathrm{\mathscr{A}}})$ and $\mathrm{H}_* (\mathscr{A} \rtimes \mu, \mathbb{C}\mu \otimes \mathbb{C}\overline{\mathscr{A}})$ by $\mathrm{H}_*(\tilde{\sigma}_s\rtimes \id_\mu,\alpha_s)$ and $\mathrm{H}_*(\tilde{\sigma}_s \rtimes \id_\mu,\id _{\mathbb{C}\mu} \otimes \alpha_s)$ respectively. By the functoriality of the morphisms in group homology of the form $\mathrm{H}_*(\sigma,\alpha)$, we obtain that $\mathrm{H}_*(\id_{\mathscr{A} \rtimes \mu} ,1 \otimes \id_{\mathbb{C}\overline{\mathscr{A}}}) $ intertwines the $\mathscr{S}$-actions. In conclusion, the $\mathscr{S}$-action on $ \mathrm{H}_*(\mathscr{A} \rtimes \mu  , \mathbb{C}\overline{\mathrm{\mathscr{A}}})$ is the restriction of that on $\mathrm{H}_*(\mathscr{A}, \mathbb{C}\overline{\mathscr{A}})$ considered in Theorem \ref{thm1}. 
\end{remark}

\begin{lemma}\label{lem:E2.homology.mu}
The spectral sequence $\widehat{\mathrm{EH}}{}_{pq}^r$, for $N=\mathscr{A} \rtimes \mu$ and $\Gamma$ as above, collapses at the $E^2$-page to $\widehat{\mathrm{EH}}{}^2_{pq}$ for $q \neq 0,d$ and 
\[
    \widehat{\mathrm{EH}}{}^2_{p0} \cong \mathbb{C}[\mu \setminus \{e\}] \otimes  \lwedge^p_{\mathbb{C}} (\Gamma \otimes \mathbb{C}), \quad 
    \widehat{\mathrm{EH}}{}^2_{pd} \cong 
    \begin{cases}
    \lwedge^p_{\mathbb{C}} (\Gamma \otimes \mathbb{C}) &\text{ if $K$ is totally imaginary, } \\
    0 & \text{ if $K$ has a real embedding. }
    \end{cases}
    \]
\end{lemma}
\begin{proof}
By Lemmas \ref{lem:LHS.homology}, \ref{lem:Raven.algaction}, \ref{lem:homology.A.mu}, and Remark \ref{rmk:homology.A.mu.action}, the $E^2$-page of this spectral sequence is 
\begin{align*}
    \widehat{\mathrm{EH}}{}^2_{pq} \cong \mathrm{H}_p(\Gamma, \widehat{\mathrm{HH}}{}^{-q}_{\mathscr{A} \rtimes \mu}(\underline{E}G , \overline{\mathscr{A}})) \cong {}&{}  \mathrm{H}_p (\Gamma, \mathrm{H}_{q}(\overline{\mathscr{A}} \rtimes (\mu \ltimes \mathscr{A}), \mathbb{C})) \oplus 
    \begin{cases}
        \mathrm{H}_{p}(\Gamma , \mathbb{C}[\mu \setminus \{e\}]) & \text{ if $q=0$, } \\
        0 & \text{ otherwise}.
    \end{cases} \\
   \cong  {}&{}\mathrm{H}_p (\Gamma, (\lwedge_{\mathbb{C}}^q (A \otimes \mathbb{C}))_\mu ) \oplus 
    \begin{cases}
        \mathbb{C}[\mu \setminus \{e\}] \otimes \lwedge_{\mathbb{C}}^p (\Gamma \otimes \mathbb{C}) & \text{ if $q=0$, } \\
        0 & \text{ otherwise}.
    \end{cases}
\end{align*}
The first direct summand is the complexification of the calculation in Proposition \ref{prop:group.homology.LHS}, and hence is isomorphic to $\mathrm{H}_p(\Gamma , \mathbb{C}_{\mathrm{sign}})$ if $q=d$ and to $0$ otherwise. 
If $\mathrm{sign} \colon \Gamma \to \{ \pm 1\}$ is non-trivial, then $\mathrm{H}_p(\Gamma , \mathbb{C}_{\mathrm{sign}}) \cong 0$. 
Indeed, by decomposing $\Gamma $ to $s^{\mathbb{Z}} \oplus \Gamma^\perp$ in the way that $\mathrm{sign}|_{\Gamma^\perp}$, we obtain the K\"unneth isomorphism
\[
    \mathrm{H}_p(\Gamma , \mathbb{C}_{\mathrm{sign}}) \cong \bigoplus_{k}\mathrm{H}_k(s^{\mathbb{Z}} , \mathbb{C}_{\mathrm{sign}}) \otimes \mathrm{H}_{p-k}(\Gamma^\perp, \mathbb{C}) \cong 0,
\]
where the last isomorphism follows from $\mathrm{H}_*(s^{\mathbb{Z}} , \mathbb{C}_{\mathrm{sign}}) \cong 0$ that is verified in \eqref{eqn:Z.sign.homology}.
In summary, the first direct summand can be non-trivial only if $q=d$ and $\mathrm{sign} \colon \Gamma \to \{ \pm 1\}$ is trivial, in which case it is isomorphic to $\mathrm{H}_p(\Gamma , \mathbb{C})$. 
Finally, since the exact couple $(\widehat{\mathrm{DH}}, \widehat{\mathrm{EH}})$ respects the direct sum decomposition over $\cC(G_{\mathrm{tor}})$, no non-trivial higher differentials are possible. 
\end{proof}

Based on the K-theory computation of the group C*-algebra of the groups of the form $\mathbb{Z}^n \rtimes \mathbb{Z}/m$ by Langer--L\"{u}ck \cite{LangerLuck}, Li--L\"{u}ck shows in \cite[Corollary~4.10]{LiLuck} that $\K_1(C^*_r((A \rtimes \mu) \ltimes \overline{A}) ) \cong 0$ and 
\begin{align}
\begin{split}
    \K_0 (C^*_r((A  \rtimes \mu) \ltimes & \overline{A} )) \cong  
    K_{\mathrm{fin}}^\mu 
    \oplus K_{\mathrm{inf}},\\
    K_{\mathrm{inf}} \cong 
    \bigoplus_{0 \leq 2k < d} \big( \lwedge^{2k }_{\mathbb{Q}} \mathscr{A} \big)_\mu
    \oplus {}&
    \begin{cases} 
    \lwedge^d_{\mathbb{Z}} A \cong \mathbb{Z} & \text{ if $d$ is even, } \\ 0 & \text{ if $d$ is odd.}
    \end{cases}
\end{split}
    \label{eqn:Ktheory.A.mu}
\end{align}
Here, $K_{\mathrm{inf}}, K_{\mathrm{fin}}^\mu \subset \K_0(C_r^*(\mu \ltimes A ))$ are the subgroups defined in \cite[Section 2]{LiLuck}. 
The group $K_{\mathrm{inf}}$ consisting of $\K_0$-classes $\xi$ such that there is $n \in \mathbb{Z}_{\geq 1}$ such that $n \xi$ is contained in the image of the map of K-groups induced by the inclusion $ C^*_r(A \ltimes \overline{A}) \to C^*_r((A \rtimes \mu) \ltimes \overline{A})$. The group $K_{\mathrm{fin}}^\mu$ is a free abelian group whose rank is equal to that of $\widetilde{R}(\mu)$, the complement of the trivial representation in the representation ring $R(\mu)$. By definition, the action of $\Gamma$ on this K-group preserves the subgroup $K_{\mathrm{inf}}$. Hereafter, we regard $K_{\mathrm{fin}}^\mu $ as a $\Gamma$-space through the identification 
\[ 
    K_{\mathrm{fin}}^\mu \cong \K_0(C^*_r((A  \rtimes \mu) \ltimes  \overline{A} )) / K_{\mathrm{inf}}.
\]
For the latter use, we introduce a different decomposition of the K-group. Let $K_{\mathrm{div}} \subset K_{\mathrm{inf}}$ denote the subgroup of divisible elements, which is the direct summand $\bigoplus_{0\leq 2k <d} (\lwedge^{2k}_{\mathbb{Q}}\mathscr{A})_\mu$ and is obviously preserved by the $\Gamma$-action, and let 
\[ 
K_{\mathrm{free}} \coloneqq \K_0(C^*_r((A  \rtimes \mu) \ltimes  \overline{A} )) / K_{\mathrm{div}} \cong K_{\mathrm{fin}}^\mu \oplus K_{\mathrm{inf}}/K_{\mathrm{div}} \cong K_{\mathrm{fin}}^\mu \oplus \mathbb{Z}.
\]
The exact sequence $0 \to K_{\mathrm{div}} \to \K_0((C^*_r((A  \rtimes \mu) \ltimes  \overline{A} )) \to K_{\mathrm{free}} \to 0$ non-equivariantly splits.

\begin{lemma}\label{lem:Kdiv}
    For any $p \in \mathbb{Z}_{\geq 0}$, we have 
    \[
    \H^p(\Gamma,K_{\mathrm{div}}) \cong 0, \quad \H^p(\Gamma, K_{\mathrm{free}}) \cong \H^p(\Gamma , \K_0(C^*_r(A \rtimes \mu) \ltimes \overline{A})).
    \]
\end{lemma}
\begin{proof}
As is considered in \cite[Theorem 0.1 (ii)]{LangerLuck}, the induced morphism in K-theory by the inclusion $k \colon  C^*_r(A \ltimes \overline{A} ) \to C^*_r((A \rtimes \mu) \ltimes \overline{A})$ factors through the coinvariant as
\[
    \bar{k}_* \colon \K_0 (C^*_r(A \ltimes \overline{A}))_\mu  \to \K_0 (C^*_r((A \rtimes \mu) \ltimes \overline{A})).
\]
Indeed, for any $[p] \in \K_0(C^*_r(A \ltimes \overline{A}))$ and $s \in \mu$, we have $k_*[p]=k_*(s_*[p])$ since the two projections $s \cdot p $ and $p$ are conjugate by the unitary $u_s$ in $C^*_r((A \rtimes \mu) \ltimes \overline{A})$. 
By Morita invariance of C*-algebra K-theory, \eqref{eqn:Raven.BC}, Lemma \ref{lem:Raven.Chern.torus}, and Theorem \ref{thm:rings}, we have 
\[ 
    \K_0(C^*_r(A \ltimes \overline{A})) \cong\K_0(C^*_r(\mathscr{A} \ltimes \overline{\mathscr{A}})) \cong  \bigoplus_{k \in \Zz} \mathrm{H}_{2k}(\mathscr{A}, \mathbb{Z}\overline{\mathscr{A}}) \cong \bigoplus_{0 \leq 2k <d}\lwedge^{2k}_{\mathbb{Q}}(\mathscr{A} \otimes_{\mathbb{Z}} \mathbb{Q} ) \oplus 
    \begin{cases}
        \lwedge^d_{\mathbb{Z}} A & \text{ if $d$ is even,}\\
        0 & \text{ if $d$ is odd.}
    \end{cases}
\]
By Remark \ref{rmk:DW.free}, \ref{rmk:action.commute.KK}, and compatibility of the $\mathscr{S}$-actions via $\widehat{\operatorname{ch}}_R$ (a remark above Lemma \ref{lem:KK.HH.LHS}), the above isomorphism is $\mathscr{S}$-equivariant, and the $\mathscr{S}$-action on the right group is given in Theorem \ref{thm:rings}. 

An element of the first direct summand $\bigoplus_{0 \leq 2k <d}\lwedge^{2k}_{\mathbb{Q}}(A \otimes_{\mathbb{Z}} \mathbb{Q} ) $, which is divisible, is sent to $K_{\mathrm{div}}$ by $\bar{k}_*$. By definition of $K_{\mathrm{div}}$, it is surjective. It is proved in \cite[Proposition 4.9]{LiLuck} that it is also injective after taking coinvariant. 
In conclusion, we have $K_{\mathrm{div}} \cong \bigoplus_{0 \leq 2k <d}\lwedge^{2k}_{\mathbb{Q}}(A \otimes_{\mathbb{Z}} \mathbb{Q} )_\mu$, not only as group but also as $\Gamma$-modules. Now, by (the proof of) Proposition \ref{prop:group.homology.LHS}, we obtain $\H^p(\Gamma,\K_{\mathrm{div}}) \cong 0$. The long exact sequence of group cohomology (\cite[Proposition III.6.1]{Brown}) yields the second claim.  
\end{proof}

To end the proof of Theorem \ref{thm:LuckLi}, we determine the action of $\Gamma$ on $K_{\mathrm{fin}}^\mu$. 
When $|V_{K,\mathbb{R}}|$ is odd, we have $\mu = \{\pm 1\}$ and hence $\widetilde{R}(\mu) \cong \Zz$. In this case, $\Gamma$ acts on $K_{\mathrm{fin}}^\mu \cong \Zz$ identically, since $K_{\mathrm{fin}}^\mu$ contains the image of $\widetilde{R}(\mu) \subset \K_0(C^*_r(\mu))$ via the inclusion $C^*_r(\mu) \to C^*_r((A  \rtimes \mu) \ltimes  \overline{A} )$ as a finite index subgroup (\cite[Theorem 0.1 (ii)]{LangerLuck}).
We need an additional discussion only in the case that $|V_{K,\mathbb{R}}|$ is even.

\begin{lemma}\label{lem:realeven.K}
    Assume that $|V_{K,\mathbb{R}}|$ is even. Let $U \subset \hat{A} \cong \mathbb{T}^d$ be an open neighborhood of the origin that is $\mu$-equivariantly homeomorphic to the open disk. Then the map $(\iota \rtimes \mu)_* \colon \K_0(C_0(U) \rtimes \mu) \to \K_0(C(\overline{A}) \rtimes (A \rtimes \mu))$ induced from the inclusion $\iota \coloneqq C_0(U) \hookrightarrow C(\hat{A}) =C^*_rA \hookrightarrow
    C(\overline{A}) \rtimes A$ is split injective. 
\end{lemma}

\begin{proof} 
Note that, since $|V_{K,\mathbb{R}}|$ is assumed to be even, we have that $\mathrm{sign}(\mu) = \{ +1\}$ (discussed in the proof of Theorem \ref{thm:gpdhomology}), and $d=[K:\mathbb{Q}] = |V_{K,\mathbb{R}}| +2|V_{K,\mathbb{C}}|$ is even.  
Let us fix a basis $A = \langle a_1,\cdots, a_d\rangle$ and identify $A$ with $\mathbb{Z}^d$. Take a $\mu$-invariant inner product $\langle \cdot, \cdot \rangle \colon \mathbb{Z}^d \times \mathbb{Z}^d \to \mathbb{R}$ and let $(g^{ij})_{1 \leq i,j \leq d}$ denote the inverse matrix of $(\langle a_i,a_j \rangle)_{i,j}$. Consider unbounded self-adjoint operators 
    \[
    X_j \colon \ell^2\mathbb{Z}^d \to \ell^2 \mathbb{Z}^d, \quad (X_j\phi )(x_1,\cdots,x_d) = x_j \phi(x_1,\cdots,x_d). 
    \]
    Passing through the Fourier transform $\ell^2A \to L^2\hat{A}$, each $X_j$ is identified with 
    \[
    i \frac{\partial}{ \partial \xi_j} \colon L^2(\mathbb{T}^d) \to L^2(\mathbb{T}^d),
    \]    
    where $\hat A\cong\mathbb{T}^d\coloneqq\mathbb{R}^d/(2\pi\mathbb{Z})^d$ via the basis ${a_1,\dots,a_d}$, and $\xi_1,\dots,\xi_d$ denote the standard coordinates on $\mathbb{T}^d$.
    Note that the inner product $\langle \cdot,\cdot \rangle$ gives $\mu$-equivariant identifications
    \[
     \hat{A} \cong \Hom(A,\mathbb{R})/\Hom(A,2 \pi \mathbb{Z}) \cong A_{\mathbb{R}} / \check{A}, \quad \text{where} \     A_{\mathbb{R}} \coloneqq A \otimes_{\mathbb{Z}} \mathbb{R}, \ \check{A} \coloneqq \{ b \in A_{\mathbb{R}} \mid \langle b,A\rangle \subset 2\pi \mathbb{Z} \}.
     \]
     The Riemannian metric on $A_{\mathbb{R}}$ given by $\langle \cdot, \cdot \rangle$ induces that on $\hat{A}$, which is preserved by the $\mu$-action. Since $\mathrm{sign}(\mu)=1$, this $\mu$-action does not change the orientation of $\hat{A}\cong \mathbb{T}^d$.

    Let $\mathbb{C}\ell_A$ be the Clifford algebra, i.e., the C*-algebra generated by elements $\mathfrak{c}(a)$ for $a \in A$ with the relations $\mathfrak{c}(a) \mathfrak{c}(b) + \mathfrak{c}(b) \mathfrak{c}(a) =-2 \langle a,b \rangle $ and $\mathfrak{c}(a)^* = - \mathfrak{c}(a)$ for any $a,b \in A$. Since $d$ is even, there is a $2^{d/2}$-dimensional Hilbert space $S_A$ such that $\mathbb{C}\ell_A \cong \End(S_A)$. 
    This $S$ is equipped with the $\mathbb{Z}/2$-grading so that $\mathfrak{c}(a)$ acts by an odd operator. 
    The $\mu$-action on $\mathbb{C}\ell_A$ given by $\mu \cdot \mathfrak{c}(a) \coloneqq \mathfrak{c}(\mu \cdot a)$ is a priori implemented by a projective representation $\lambda \colon \mu \to \mathcal{U}(S_A)/\mathbb{T}$. Since $\mathrm{sign}(\mu)=\{+1\}$, the $\mu$-action on $A_{\mathbb{R}}$ preserves the orientation, and hence $\lambda(\mu) $ preserves $\mathbb{Z}/2$-grading. 
    Moreover, since $\mu$ is a finite cyclic group, and hence satisfies $\mathrm{H}^2(\mu , \mathbb{T}) \cong 0$, this $\lambda$ lifts to a genuine unitary representation preserving the $\mathbb{Z}/2$-grading.

    The Dirac operator on $\hat{A} \cong \mathbb{T}^d$ with respect to the above Riemannian metric is an odd $\mu$-invariant unbounded self-adjoint operator defined by 
    \[
    D \coloneqq  \sum_{j,k =1}^d g^{jk} \cdot \mathfrak{c}(a_j)  \cdot \frac{\partial}{\partial \xi_k} \colon L^2(\mathbb{T}^d, S_A) \to L^2(\mathbb{T}^d, S_A). 
    \]
    This is an elliptic operator on a compact manifold, and hence is a Fredholm operator. 
    Going back to the Fourier dual $\ell^2A \otimes \mathbb{C}\ell_A$, this $D$ corresponds to the odd self-adjoint operator 
    \[
    C \coloneqq  \sum_{j,k=1}^dg^{jk} \cdot \mathfrak{c}(a_j)  \cdot iX_k \colon \ell^2(\mathbb{Z}^d ,S_A)  \to \ell^2(\mathbb{Z}^d , S_A). 
    \]
    We consider the $\ast$-representation $\pi \colon C(\overline{A}) \rtimes A \to B(\ell^2A)$ arising from the covariant pair consisting of the multiplication representation of $C(\overline{A}) \subset C_b(A)$ on $\ell^2A$ and the left regular representation of $A$ on $\ell^2A$. Then, the bounded transform $C(1+C^2)^{-1/2}$ commutes with $C^*_r(A)$ modulo compact operators, and commutes with $C(\overline{A})$. 
    That is, we obtain an element of Kasparov's KK-group
\[
    \alpha \coloneqq [\ell^2(A, S_A), \pi, C(1+C^2)^{-1/2}] \in \KK^\mu (C(\overline{A}) \rtimes A, \mathbb{C}).
\]
    The composition with $\iota$ in the statement of the lemma is 
    \[
    [\iota] \otimes_{C(\overline{A}) \rtimes A } \alpha  = [L^2(U,S_A), m \circ \iota, PD(1+D^2)^{-1/2}P] \in \KK^\mu(C_0(U),\mathbb{C}),
    \]
    where $m \colon C(\hat{A}) \to B(L^2(\hat{A}))$ denotes the multiplication representation and $P$ denotes the projection onto $L^2(U,S_A)$. 
    By \cite[Proposition 10.8.8]{HigsonRoe}, this KK-class is identical with the one aligned with the Dirac operator on $U$, which is nothing else than the equivariant Bott class on $U \cong \mathbb{R}^d$. Hence, by Kasparov's equivariant Bott periodicity theorem (\cite[{Theorem 7}]{Kasparov1}), it is a $\mu$-equivariant KK-equivalence. 
    Applying Kasparov's descent functor (\cite[Theorem 3.4]{Kasparov}), which we write $\cdot \rtimes \mu$ in this paper, we obtain non-equivariant KK-classes
    \begin{gather*}
    \alpha \rtimes \mu \in  \KK ((C(\overline{A}) \rtimes A) \rtimes \mu, \mathbb{C} \rtimes \mu ) = \KK (C(\overline{A}) \rtimes (A \rtimes \mu) , C^*_r(\mu )),\\
    [\iota] \rtimes \mu = [\iota \rtimes \mu] \in  \KK (C_0(U) \rtimes \mu, (C(\overline{A}) \rtimes A) \rtimes \mu ) .
    \end{gather*}
    By the compatibility of the descent with the Kasparov product, we have that 
    \[
    ([\iota] \rtimes \mu ) \otimes_{(C(\overline{A}) \rtimes A) \rtimes \mu}(\alpha \rtimes \mu) = ([\iota] \otimes_{C(\overline{A}) \rtimes A} \alpha) \rtimes \mu 
    \]
    is a KK-equivalence, with the KK-inverse $\beta \rtimes \mu$, where $\beta \in \mathrm{KK}^\mu(\mathbb{C},C_0(U))$ is the $\mathrm{KK}^\mu$-inverse of $[\iota] \otimes_{C(\overline{A}) \rtimes A} \alpha$. Now, the Kasparov product with $(\alpha \otimes_{\mathbb{C}} \beta) \rtimes \mu$ induces a homomorphism $\mathrm{K}_0(C(\overline{A}) \rtimes (A \rtimes \mu)) \to \K_0( C_0(U) \rtimes \mu)$, which is a right inverse of $(\iota \rtimes \mu)_*$.     
\end{proof}
Since $\K_0(C_0(U) \rtimes \mu) \cong R(\mu) \cong \widetilde{R}(\mu) \oplus \mathbb{Z}$ by the equivariant Bott periodicity as above, Lemma \ref{lem:realeven.K}, \eqref{eqn:Ktheory.A.mu}, and a rank comparison shows that 
\[ 
     \K_0(C^*_r((A \rtimes \mu) \ltimes \overline{A})) \cong \Im (\iota \rtimes \mu)_* \oplus K_{\mathrm{div}}, 
\]
that is, $K_{\mathrm{free}} \cong \Im (\iota \rtimes \mu)_*$. 
This identification determines the $\Gamma$-action on $K_{\mathrm{free}}$ as in the following lemma. For $s \in \Gamma$, we write $\alpha_s \coloneqq \sigma_s^* \colon C_0(\overline{\A}) \to C_0(\overline{\A})$. The $\Gamma$-action on $\K_0(C^*_r((A \rtimes \mu) \ltimes \overline{A})) $ is given by $\theta_s \coloneqq (\alpha_s \rtimes (\sigma_s \rtimes \id_\mu))_*$.

\begin{remark}\label{rmk:Im.indep}
Hereafter, we take $U$ in Lemma \ref{lem:realeven.K} to be the image of a small $\mu$-invariant star-shaped neighborhood of the origin in $A_{\mathbb{R}}$.
One reason is that, for such $U$, the subgroup $\Im(\iota \rtimes \mu)_*$ is independent of the choice of $U$.
Indeed, for two such neighborhoods $U$ and $V$, choose another $\mu$-invariant star-shaped neighborhood $W$ such that $\overline{W} \subset U \cap V$.
The $C^*$-algebras $C_0(U \setminus W)$ and $C_0(V \setminus W)$ are $\mu$-equivariantly contractible via the radial reparametrization $f(x) \mapsto f(tx)$.
Thus, the upper horizontal morphisms in the commutative diagram
\[
    \xymatrix{
    \K_0(C_0(U) \rtimes \mu) \ar[rd]_{(\iota \rtimes \mu)_*} & \K_0(C_0(W) \rtimes \mu) \ar[l] \ar[r] \ar[d]^{(\iota \rtimes \mu)_*} & \K_0(C_0(V) \rtimes  \mu) \ar[ld]^{(\iota \rtimes \mu)_*} \\
    &\K_0(C(\overline{A}) \rtimes (A \rtimes \mu) ),&
    }
\]
that are induced by inclusions, are isomorphisms.
This shows the desired independence of $\Im(\iota \rtimes \mu)_*$.
\end{remark}

\begin{lemma}\label{lem:realeven.AHSS}
The following hold.
\begin{enumerate}[\upshape(i)]
    \item If $K$ is totally imaginary, then $\Gamma$ acts trivially on $K_{\mathrm{free}}$. 
    \item If $|V_{K,\mathbb{R}}|$ is nonzero and even, then, for $s \in \Gamma$, we have
    \begin{itemize}
        \item $\theta_s|_{K_{\mathrm{free}}} = \id$ if $\mathrm{sign} (s) =1$, and
        \item $\ker ( \theta_s|_{K_{\mathrm{free}}} -\id) \cong \mathbb{Z}$ and $\coker ( \theta_s|_{K_{\mathrm{free}}} -\id) \cong \mathbb{Z}$ if $\mathrm{sign}(s) =-1$. 
    \end{itemize}
\end{enumerate}
\end{lemma}

\begin{proof}
    Let $s \in \Gamma$. We prove that $\theta_s$ preserves the subgroup $ \Im (\iota \rtimes \mu)_*$, and determine the action. 
    Following Lemma \ref{lem:realeven.K}, we fix a $\mu$-invariant inner product on $A$, identify the Pontrjagin dual $\hat{A}$ with $A_{\mathbb{R}}/\check{A}$, and choose an open ball $U =B_r(0) \subset A_{\mathbb{R}}$ with respect to the metric $\langle \cdot ,\cdot \rangle$ centered at the origin, which is automatically $\mu$-invariant. The radius $r>0$ is chosen sufficiently small so that the restriction of the quotient map, $U \to \hat{A}$, is injective. For $s \in \Gamma$, we write $\sigma_s^\prime $ for the contragredient of $\sigma_s$ with respect to the $\mu$-invariant inner product $\langle \cdot, \cdot \rangle$ on $A$. Since $\langle A, \sigma_s^\prime (\check{A})\rangle = \langle \sigma_s(A), \check{A}\rangle \subset 2\pi \mathbb{Z}$, $\sigma_s^\prime$ restricts to an endomorphism on $\check{A}$, and hence induces a covering map $\sigma_s^\prime \colon \hat{A} \to \hat{A}$. By definition, the pull-back $\sigma_s^{\prime *} \colon C(\hat{A}) \to C(\hat{A})$ is identified with $\sigma_s \colon C^*_r(A) \to C^*_r(A)$ via the Pontrjagin duality.
    By taking $r >0$ sufficiently small, we may assume that the restriction of the quotient $\sigma_s^\prime (U) \to \hat{A}$ is also injective. Note that $\sigma_s^\prime (U)$ is also $\mu$-invariant since $\sigma_s^\prime$ commutes with the $\mu$-action. 
    Now, the following diagram commutes $\mu$-equivariantly;
\[
    \xymatrix{
    C_0(A_{\mathbb{R}}) \ar[d]^{\sigma_s^{\prime *} } & \ar[l] C_0(U) \ar[r] \ar[d]^{\sigma_s^{\prime*}} & C(\hat{A}) \ar@{=}[r] \ar[d]^{\sigma_s^{\prime*}} & C^*_rA \ar[r] \ar[d]^{\sigma_s} & C(\overline{A}) \rtimes A \ar[d]^{\alpha_s \rtimes \sigma_s} \\
    C_0(A_{\mathbb{R}}) & \ar[l] C_0(\sigma_s^\prime (U)) \ar[r]  & C(\hat{A}) \ar@{=}[r]  & C^*_r(A) \ar[r]  & C(\overline{A}) \rtimes (A) , 
    }
\]
which induces the following morphisms of K-theory:
\begin{align}
\begin{split}
    \xymatrix{
    \K_0(C_0(A_{\mathbb{R}}) \rtimes \mu) \ar[d]^{(\sigma_s^{\prime *} \rtimes \mu)_* } & \ar[l]_\cong \K_0(C_0(U) \rtimes \mu) \ar[r]^{(\iota \rtimes \mu)_*} \ar[d]^{(\sigma_s^{\prime*} \rtimes \mu)_*} &  \K_0(C(\overline{A}) \rtimes (A \rtimes \mu)) \ar[d]^{\theta_s = (\alpha_s \rtimes (\sigma_s \rtimes \id_\mu))_*} \\
    \K_0(C_0(A_{\mathbb{R}}) \rtimes \mu) & \ar[l]_\cong \K_0(C_0(\sigma_s^\prime (U)) \rtimes \mu) \ar[r]^{(\iota \rtimes \mu)_*}  & \K_0(C(\overline{A}) \rtimes (A \rtimes \mu)) .     
    }
    \end{split}\label{eqn:diagram.Kfree}
\end{align}
By Remark \ref{rmk:Im.indep}, the commutativity of the right square shows that $\theta_s $ preserves the subgroup $\Im (\iota \rtimes \mu)_* $. Moreover, by the same reason as Remark \ref{rmk:Im.indep}, the left horizontal maps are isomorphisms. Thus, the left vertical morphism $(\sigma_s^{\prime *} \rtimes \mu)_* $ is identified with $\theta_s|_{\Im (\iota \rtimes \mu)_*}$. 
At the left vertical map, $\sigma_s^{\prime} \colon A_{\mathbb{R}} \to A_{\mathbb{R}}$ is a linear automorphism that commutes with the representation of $\mu$. Therefore, it is homotopic to the identity through $\mu$-equivariant linear isomorphisms if and only if its restriction to each eigenspace of $\mu$-representation has positive determinant. 

    Based on the preparations so far, we prove (i). If $K$ is totally imaginary, then $A_{\mathbb{R}}$ is decomposed to the direct sum of its complex embeddings $\bigoplus_{w \in V_{K,\mathbb{C}}} \mathbb{C}$ (\cite[Proposition I.5.1]{Neu}). 
    On each complex embedding $\mathbb{C}$, the multiplicative group $\Gamma \times \mu$ of $K$ acts by a multiplication by complex numbers. 
    Thus, the determinant as an $\mathbb{R}$-linear map is positive, and hence is homotopic to the identity.

    We then prove (ii).  
    By Lemma \ref{lem:realeven.K} and \eqref{eqn:diagram.Kfree}, the kernel and the cokernel of $\theta_s|_{K_\mathrm{free}} -\id$ in turn coincide with those of $(\sigma_s^{\prime *} \rtimes \id_\mu)_*-\id$.
    To determine these groups, we apply the computation in \cite{CL2}.
    Since $\mu = \{ \pm 1\}$ in this case, which acts on $A_\mathbb{R}$ by multiplication, there is no need to care about the $\mu$-eigenspace decomposition. In this case, $\det (\sigma_s^{\prime}) = \det (\sigma_s) = \operatorname{sign}(s) = -1$. Therefore, $\sigma_s^{\prime}$ is linearly (and hence $\mu$-equivariantly) homotopic to $\mathrm{diag}(1,\cdots ,1 , -1)$. This enables us to apply \cite[Lemma 2.1]{CL2}, to get $\ker ( (\sigma_s^{\prime *} \rtimes \id_\mu)_* -\id) \cong \mathbb{Z}$ and $\coker ( (\sigma_s^{\prime*} \rtimes \id_\mu )_* -\id) \cong \mathbb{Z}$. 
\end{proof}

    
\begin{lemma}\label{lem:E2.Ktheory.even}
    Assume that $|V_{K,\mathbb{R}}|$ is nonzero and even. Then we have
    \[     \mathrm{H}_m (\Gamma; \K_0(C^*_r((A \rtimes \mu) \ltimes \overline{A} ))) \cong  \lwedge_{\mathbb{Z}}^m \Gamma.
\]
\end{lemma}
\begin{proof}
     By Lemma \ref{lem:Kdiv}, it suffices to compute the group $\H_*(\Gamma , K_{\mathrm{free}})$. 
     Pick $s \in \Gamma$ so that $s^{\mathbb{Z}}$ forms a direct summand and $\mathrm{sign}(s) =-1$ (Lemma \ref{lem:sign2}), 
     and a complement $\Gamma^\perp $ of $s^{\mathbb{Z}}$ in $\Gamma$ such that $\mathrm{sign}(\Gamma^{\perp}) =\{+1\}$. 
     Then, by Lemma \ref{lem:realeven.AHSS} (ii), any $t \in \Gamma^\perp$ acts on $K_{\mathrm{free}}$ by the identity, and we have 
     \[
     \H_n(s^{\mathbb{Z}} , K_{\mathrm{free}}) \cong 
     \begin{cases}
         \ker (\theta_s - \id) \cong \mathbb{Z}, & \text{ if $n=0$, }\\
         \operatorname{coker} (\theta_s - \id) \cong \mathbb{Z}, & \text{ if $n=1$, }\\        0 & \text{ otherwise. }
     \end{cases}
     \]
     In particular, this group is torsion-free. 
     Since $\H_p(\Gamma^\perp,\mathbb{Z}) \cong \lwedge_{\mathbb{Z}}^p \Gamma$ is also torsion-free, and $\Gamma^\perp$ acts on $K_{\mathrm{free}}$ trivially, the K\"{u}nneth theorem (see e.g. \cite[Section V.2]{Brown}) shows that 
    \[ 
     \mathrm{H}_n (\Gamma^\perp \times s^{\mathbb{Z}} , K_{\mathrm{free}} ) \cong \bigoplus_{p+q=n}\mathrm{H}_p(\Gamma^\perp, \mathbb{Z}) \otimes \mathrm{H}_q(s^{\mathbb{Z}}, K_{\mathrm{free}} )
    \cong \lwedge^{p}_{\mathbb{Z}} \Gamma^\perp \oplus \lwedge^{p-1}_{\mathbb{Z}} \Gamma^\perp \cong \lwedge^{p}_{\mathbb{Z}}\Gamma. \qedhere 
    \]
\end{proof}

\begin{proof}[Proof of Theorem \ref{thm:LuckLi}]
Here we consider the spectral sequence \eqref{eqn:exact.couple.LHS} associated to the exact couple
\begin{align*}
\begin{split}
    \mathrm{DK}_{pq}^1 \coloneqq & \KK_{p+q}^{G} (\underline{E}G \times E_p\Gamma , \overline{\mathscr{A}}),\\
    \mathrm{EK}_{pq}^1 \coloneqq & \KK_{p+q}^G( \underline{E}G \times ( E_p\Gamma, E_{p-1}\Gamma) , \overline{\mathscr{A}}) \cong \KK^{\mathscr{A} \rtimes \mu}_{q}(\underline{E}G, \overline{\mathscr{A}}) \otimes_{\mathbb{Z}[\Gamma]}\mathbb{Z}[\Gamma^{p+1}] .    
\end{split}
\end{align*}
Here, the isomorphism in the second line is verified in the same way as \eqref{eqn:E1page.tensor}, by Lemma \ref{lem:KK.HH.LHS}. By Lemma \ref{lem:induction.2} for $\overline{\mathscr{A}}=\mathscr{A} \times_A \overline{A}$, Remark \ref{rmk:EG} (iv) claiming $\underline{E}G = \underline{E}A$, and the Baum-Connes isomorphism \eqref{eqn:Raven.BC} for $A \rtimes \mu$ and $\overline{A}$, we have the isomorphism 
\[
    \KK^{\mathscr{A} \rtimes \mu}_{q}(\underline{E}G, \overline{\mathscr{A}}) \cong \KK^{A \rtimes \mu}(\underline{E}A,\overline{A}) \cong  \K_0(C^*_r(A \rtimes \mu) \ltimes \overline{A})  \cong K_{\mathrm{fin}}^\mu \oplus K_{\mathrm{inf}}
\]
is torsion-free. Thus, Raven's Chern character gives an injective map 
\[
    \mathop{\widehat{\mathrm{ch}}_R} \colon \KK^{\mathscr{A} \rtimes \mu}_{q}(\underline{E}G, \overline{\mathscr{A}}) \otimes_{\mathbb{Z}[\Gamma]} \mathbb{Z}[\Gamma^{p+1}] \to \widehat{\mathrm{HH}}{}_{\mathscr{A} \rtimes \mu}^{-q}( \underline{E}G, \overline{\mathscr{A}}) \otimes_{\mathbb{C}[\Gamma]}\mathbb{C}[\Gamma^{p+1}]. 
\]
Since the differentials $d^1_{pq}$ of the spectral sequences $\mathrm{EK}$ and $\widehat{\mathrm{EH}}$ are given by the boundary map of the equivariant homology theories $\mathrm{KK}_{p+q}^G(\underline{E}G \times \cdot, \overline{\mathscr{A}})$ and $\mathrm{KK}_{p+q}^G(\underline{E}G \times \cdot, \overline{\mathscr{A}})$, they are compatible with respect to Raven's Chern character: $d^1_{pq} \circ \mathop{\widehat{\mathrm{ch}}_R} = \mathop{\widehat{\mathrm{ch}}_R} \circ d_{pq}^1$. Therefore, by comparing the two spectral sequence, it turns out that $d^1_{pq}$ for $\mathrm{EK}$, which is the restriction of that for $\widehat{\mathrm{EH}}$, is identical with that of group cohomology with coefficient in $\KK^{\mathscr{A} \rtimes \mu}_{q}(\underline{E}G, \overline{\mathscr{A}})$. Finally, by Remark \ref{rmk:action.commute.KK}, we obtain that 
\[
    \mathrm{EK}^2_{pq} \cong \mathrm{H}^p(\Gamma, \mathrm{KK}_q^N(\underline{E}G, \overline{\mathscr{A}})) \cong \mathrm{H}^p(\Gamma, \mathrm{K}_q(C^*_r((A \rtimes \mu ) \ltimes \overline{A})).
\]

By \eqref{eqn:Ktheory.A.mu} and $\H^p(\Gamma ; K_{\mathrm{div}}) \cong 0$, we have $\mathrm{EK}_{p0}^2 \cong \H^p(\Gamma , K_{\mathrm{free}})$ and $\mathrm{EK}_{p1}^2 \cong 0$.
If $K$ is totally imaginary, by Lemma \ref{lem:realeven.AHSS} (1) we have
    \[ 
    \mathrm{EK}_{pq}^2 \cong \H^p(\Gamma,K_{\mathrm{free}}) \cong
    \begin{cases}
        \lwedge_{\mathbb{Z}}^p \Gamma \oplus \widetilde{R}(\mu) \otimes_{\mathbb{Z}} \lwedge_{\mathbb{Z}}^p \Gamma = R(\mu) \otimes_{\mathbb{Z}} \lwedge_{\mathbb{Z}}^p \Gamma & \text{ if $q \equiv 0 \equiv d$ mod $2$, }\\
        0 & \text{ otherwise.}
    \end{cases}
    \]
    If $|V_{K,\mathbb{R}}|$ is odd, $\Gamma$ acts on $K_{\mathrm{fin}}^\mu$ trivially as is remarked above, and hence  
    \[ 
    \mathrm{EK}_{pq}^2 \cong 
    \begin{cases}
        \widetilde{R}(\mu) \otimes_{\mathbb{Z}} \lwedge_{\mathbb{Z}}^p \Gamma = \lwedge_{\mathbb{Z}}^p \Gamma & \text{ if $q\equiv d$ mod $2$, }\\
        0 & \text{ otherwise.}
    \end{cases}
    \]
    If $|V_{L,\mathbb{R}}|$ is nonzero even, then by Lemma \ref{lem:E2.Ktheory.even} we have
    \[
    \mathrm{EK}^2_{pq} \cong 
        \begin{cases}
        \lwedge_{\mathbb{Z}}^p \Gamma  & \text{ if $q\equiv d$ mod $2$, }\\
        0 & \text{ otherwise.}
    \end{cases}
    \]

    In any cases, the groups $\mathrm{EK}_{pq}^2$ are torsion-free, and hence is sent to $\bigoplus_{k \in \mathbb{Z}}\widehat{\mathrm{EH}}{}_{p,q+2k }^2$ injectively by Raven's Chern character $\widehat{\mathrm{ch}}_R$. 
    Now Lemma \ref{lem:E2.homology.mu} concludes that $\mathrm{EK}_{pq}^r$ collapses at the $E^2$-page. 
    Since there is no extension problem, we obtain the desired isomorphism. 
\end{proof}

\section{Integral dynamics and the Barlak--Omland--Stammeier conjecture}
\label{sec:BOS} 
In this section, we compute the groupoid homology of the groupoids underlying the C*-algebras from \cite{BOS} (Subsection~\ref{sec:BOShom}), and we prove \cite[Conjecture~6.5]{BOS} (Subsection~\ref{sec:BOSKtheory}). Combining these results, we see that the groupoids satisfy Matui's HK property. 

\subsection{Integral dynamics}
\label{sec:BOSprelims}
Let us recall the definition of integral dynamics from \cite{BOS} and fix some notation that will be used throughout this section. Let $\Sigma\subseteq\Zz_{>1}$ be a set of pairwise relatively coprime numbers, and let $S\subseteq\Zz_{>0}$ be the monoid generated by $\Sigma$. 
Then, $S$ acts on $A=\Zz$ by multiplication, yielding an algebraic action $S\acts A$. Let $g_\Sigma \coloneqq \gcd \{ s-1 \colon s \in \Sigma \}$, and put $\cP\coloneqq\{p\in\cP_\Qz : p\mid s \text{ for some }s\in\Sigma\}$. Applying Section~\ref{sec:gpdhom}, we obtain the groupoid $\cG_{S\acts\Zz} = (\A \rtimes \S) \ltimes \overline{\mathscr{A}}$,
where
\begin{align*}
    \A &= \Zz[1/s : s\in S] = \Zz[1/p \colon p \in \cP] \cong \bigotimes_{p \in \cP} \Zz[1/p], \\
    \S &= S^{-1}S = \gp{s : s\in \Sigma} \cong \Zz^{\oplus\Sigma}, \mbox{ and } \overline{A}\cong \Delta \coloneqq\prod_{p\in \cP} \Zz_p.
\end{align*}
The C*-algebra $C^*_r(\cG_{S\acts\Zz})$ is isomorphic to the C*-algebra from \cite[Definition~2.1]{BOS}.  

Our general results Theorem~\ref{thm1} and Proposition~\ref{prop:Z[1/n]} imply that 
\begin{equation}
 \label{eqn:odhom}
\H_0(\Zz, \Zz \Delta) = \A, \quad  \H_1(\Zz, \Zz\Delta) = \Zz,
\end{equation}
and $\H_q(\Zz, \Zz\Delta)=0$ for $q >1$. Note that these special cases also follow from, e.g., \cite{Scar}.

\subsection{Calculation of homology}
\label{sec:BOShom}
First, we assume that $|\Sigma| < \infty$ and let $N=|\Sigma|$. The case $|\Sigma|=\infty$ will be treated in Section~\ref{sec:infinitecase}.

\subsubsection{The torsion subgroupoid} \label{sssec:tor}

It is proven in \cite[Theorem~6.1]{BOS} that the K-theory of $C^*_r(\cG_{S\acts\Zz})$ decomposes as a direct sum of a free abelian part and a torsion part, where the torsion part is isomorphic to the K-theory of a C*-subalgebra of $C^*_r(\cG_{S\acts\Zz})$, called the \emph{torsion subalgebra} (see \cite[Definition~5.3]{BOS}). Moreover, it is observed in \cite[Remark~5.7]{BOS} and \cite[Corollary~5.8]{BOS} that this torsion subalgebra can be realized as the C*-algebra of a $k$-graph as defined in \cite{KP00}. It was proven in \cite{KP00} that such C*-algebras have canonical groupoid models, and the homology of such groupoids has been computed for several specific classes of $k$-graphs in \cite{FKPS}. 

Here, we recall the construction from \cite[Remark~5.7]{BOS} and observe that the groupoid underlying the torsion subalgebra embeds into $\cG_{S\acts\Zz}$. This allows us to compute the homology of $\cG_{S\acts\Zz}$ (Theorem~\ref{thm:homintdyn}) using the functoriality for LHS type spectral sequences for groupoids as established in Theorem~\ref{thm:LHSfunctorial}.

For $a = (a_s)_{s \in \Sigma} \in \Nz^\Sigma\cong S$, let $s^a \coloneqq \prod_{s \in \Sigma} s^{a_s}$ in order to simplify notation.  Let 
\[ 
\Lambda \coloneqq \bigsqcup_{a \in \Nz^\Sigma} \Lambda^a, \mbox{ where } \Lambda^a \coloneqq \{0,\dots, s^a-1\}.
\]
We consider $\Lambda$ as a category with a single object (i.e., a monoid) and composition rule given by
\[ 
\Lambda^a \times \Lambda^b \to \Lambda^{a+b},\ (m,n) \mapsto m+s^an.
\]
Note that this map is bijective. This composition rule is associative and defines the same category as that constructed in \cite[Remark~5.7]{BOS}. 

We consider the category $\Lambda$ as a $k$-graph whose length function is the canonical map $\Lambda \to \Nz^\Sigma$. 
Let $\Pi = \{(m,n) \in \Nz^\Sigma \times \Nz^\Sigma \colon m \leq n\}$. Then, the infinite path space $\Lambda^\infty$ is defined to be the set of functions $x \colon \Pi \to \Lambda$ such that 
\begin{align*}
\begin{array}{ll}
x(m,m)=x(0,0) \in \Lambda^0 & \mbox{for } m \in \Nz^\Sigma, \\
x(m,n) \in \Lambda^{n-m} & \mbox{for } m \leq n, \\
x(m,n)x(n,k)=x(m,k) & \mbox{for } m \leq n \leq k. 
\end{array}
\end{align*}
For $a \in \Nz^\Sigma$, define 
\[ P_a \colon \Lambda^\infty \to \Lambda^a,\ x \mapsto x(0,a). \]
Then, the topology of $\Lambda^\infty$ is generated by clopen sets $P_a^{-1}(m)$ for $m \in \Lambda^a$ and $a \in \Nz^\Sigma$. The space $\Lambda^\infty$ can be written as the inverse limit 
\[ \Lambda^\infty = \varprojlim (\Lambda^a, P_{a,b}), \]
where $P_{a,b} \colon \Lambda^a \to \Lambda^b$ for $b \leq a$ are the appropriate connecting maps which induce $P_a \colon \Lambda^\infty \to \Lambda^a$ as the limit map. We first compute the concrete form of $P_{a,b}$. 

\begin{lemma} \label{lem:Zp1}
    Let $a,b \in \Nz^\Sigma$ with $b \leq a$. 
    Under the identification $\Lambda^a \cong \Zz/s^a\Zz$, 
    the map $P_{a,b}$ coincides with the canonical projection $\Zz/s^a\Zz$ to $\Zz/s^b\Zz$. 
\end{lemma}

\begin{proof}
    Let $x \in \Lambda^\infty$. Then, we have $x(0,b)x(b,a)=x(0,a)$, so that $P_b(x)x(b,a)=P_a(x)$. Hence, we have $P_b(x) + s^bx(b,a) = P_a(x)$ in $\Zz$, which means that $P_b(x)$ coincides with the image of $P_a(x)$ under the canonical projection $\Zz/s^a\Zz \to \Zz/s^b\Zz$. 
\end{proof}
By Lemma~\ref{lem:Zp1}, $\Lambda^\infty$ is identified with $\Delta = \prod_{p \in \cP} \Zz_p$, and $P_a$ coincides with the projection $\Delta \to \Zz/s^a\Zz$ for $a \in \Nz^\Sigma$. Next, we compute the shift map 
\[ \sigma^a \colon \Lambda^\infty \to \Lambda^\infty,\ (\sigma^ax)(m,n) = x(m+a,n+a)\]
for $a \in \Nz^\Sigma$. 
\begin{lemma} \label{lem:shift1}
    For $x \in \Delta$ and $a \in \Nz^\Sigma$, we have 
    \[ \sigma^a(x) = \frac{x-P_a(x)}{s^a}. \]
    In particular, if $x \in \Zz$, then $\sigma^a(x)$ is the quotient of $x$ after division by $s^a$ with remainder. 
\end{lemma}

Note that $P_a(x) \in \Lambda^a = \{0,\dots,s^a-1\}$, so that the right hand side makes sense. 

\begin{proof}
    It suffices to show the desired equation only for $x \in \Nz$ because $\Nz$ is dense in $\Delta$. Take $l \in \Nz^\Sigma$ sufficiently large so that $x \leq s^l$. Then, we have $P_m(x) =x$ for $m \geq l$.  
    From the equation $x(0,a)x(a,l+a) =x(0,l+a)$, we have 
    \[ P_a(x)+s^a(P_l(\sigma^ax)) = P_{l+a}(x) = x, \]
    so that we have 
    \[ P_l(\sigma^ax) = \frac{x-P_a(x)}{s^a}. \]
    The right hand side does not depend on $l$. Hence, we have $\sigma^ax \in \Nz$ and 
    $\sigma^ax = P_l(\sigma^ax)$, which completes the proof. 
\end{proof}
The Deaconu--Renault groupoid associated with the $k$-graph $\Lambda$ is 
\begin{equation}
\label{eqn:torsiongpoid}
   \cG_\tor \coloneqq \{ (x, a-b, y) \in \Delta \times \Zz^\Sigma \times \Delta \colon \sigma^a(x) = \sigma^b(y),\, a,b \in \Nz^\Sigma\}. 
\end{equation}
The C*-algebra $C_r^*(\cG_\tor)$ is the torsion subalgebra from \cite[Definition~5.3.]{BOS}, see \cite[Remark~5.7]{BOS}.

For $(x,a-b,y) \in \cG_\tor$, we have
\[ \frac{x-P_a(x)}{s^a} = \frac{y-P_b(y)}{s^b}, \]
so that $x = s^{a-b}(y-P_b(y)) +P_a(x)$. Therefore, we can embed $\alpha \colon \cG_\tor \to \cG_{S\acts\Zz}$ by 
\[ (x,a-b,y) \mapsto ((-s^{a-b}P_b(y)+P_a(x), s^{a-b}), y). \]
Note that $\alpha$ is an \'etale homomorphism. 

\begin{remark}
    The category $\Lambda$ is a monoid isomorphic to the right LCM monoid $U$ from \cite[Section~5]{BOS}, and the embedding $\alpha$ is essentially the groupoid version of \cite[Corollary~5.8]{BOS}.
\end{remark}

By \cite[Theorem~7.11]{FKPS}, we have 
\begin{equation} \label{eqn:torH}
    \H_p(\cG_{\tor}) \cong \begin{cases}
   (\Zz/g_\Sigma \Zz)^{\oplus \binom{N-1}{p}} & \text{ if } 0\leq p\leq N-1,\\
   0 & \text{ otherwise}.\end{cases}
\end{equation}

\subsubsection{Homology of the integral dynamics groupoid}
In this subsection, we perform the homology analogue of the K-theory calculation for the C*-algebras from \cite{BOS}. In particular, the key step is showing that the exact sequence in \eqref{eqn:ses5} splits by using the LHS spectral sequence and its functoriality, which is the homology analogue of \cite[Corollary~4.7]{BOS}.

Let $\rho \colon \cG_{S\acts\Zz} \to \S$ be the cocycle defined by $\rho(((a,s),x))=s$ for all $((a,s),x)\in\cG_{S\acts\Zz}$. Define a cocycle $\varphi \colon \cG_\tor \to \S$ by $\varphi = \rho \circ \alpha$. 
We have the following commutative diagram
\begin{equation}
    \label{eqn:square}
 \begin{tikzcd}
\ker \varphi \arrow[r] \arrow[d,"\beta"]& \cG_\tor \times_\varphi \S \arrow[d]\\
\ker \rho \arrow[r,"\iota"] & \cG_{S\acts\Zz} \times_\rho \S,
 \end{tikzcd}
\end{equation}
where $\beta$ is the restriction of $\alpha$ and $\iota$ is the canonical inclusion. The top horizontal line in diagram \eqref{eqn:square} is a groupoid equivalence by \cite[Lemma~6.1]{FKPS}.

\begin{remark}
\label{rmk:AF}
    The groupoid $\ker \varphi$ is an AF-groupoid in the sense of \cite[Defintion~2.2]{Mat12} (see, e.g., the proof of \cite[Lemma~6.2]{FKPS}). In particular, $\H_p(\cG_\tor \times_\varphi \S)=\H_p(\ker \varphi) = 0$ for all $p>0$ by \cite[Theorem~4.11]{Mat12}. 
\end{remark}

\begin{lemma} \label{lem:betaisom}
 The map $\H_0(\beta)\colon\H_0(\ker\varphi)\to\H_0(\ker\rho)$ is an isomorphism. In addition, the bottom horizontal map in diagram \eqref{eqn:square} induces an isomorphism in homology. 
\end{lemma}
\begin{proof}
Since $\ker\rho=\Zz\ltimes\Delta$, we have $\H_0(\ker\rho)=\H_0(\Zz,\Zz\Delta)=\Zz[1/n]$ (see Equation~\eqref{eqn:odhom}), where $n\coloneqq s_1\cdots s_N$. For $k\in\{0,...,n-1\}$, we have $(k,0,0)\in\ker\varphi$ because $k-0=P_n(k)-P_n(0)$, so that $[1_{k+n\Delta}]_0=[1_{n\Delta}]_0$ in $\H_0(\ker\varphi)$. Hence, $[1_\Delta]_0=\sum_{k=0}^{n-1}[1_{k+n\Delta}]_0=n[1_{n\Delta}]_0$, so that there is a homomorphism $\Zz[1/n]\to \H_0(\ker\varphi)$ such that $1/n\mapsto [1_\Delta]_0$. It is easy to see that this is the inverse of $\H_0(\beta)$.

    In order to show the second claim, it suffices to show that $(\ker \rho)^{(0)}=\Delta \times \{e\}$ is full in $\cG_{S\acts\Zz} \times_\rho \S$ in the sense that $\Delta \times \{e\}$ intersects every $(\cG_{S\acts\Zz} \times_\rho \S)$-orbit. This can be seen directly.
\end{proof}

\begin{theorem}
\label{thm:homintdyn}
    If $|\Sigma|=N < \infty$, then we have
    \[
    \H_n(\cG_{S\acts \Zz})\cong \begin{cases}
        \Zz/g_\Sigma \Zz & \text{ if } n=0,\\
        \Zz^{\oplus \binom{N}{n-1}} \oplus (\Zz/g_\Sigma \Zz)^{\oplus \binom{N-1}{n}} & \text{ if } 1\leq n\leq N-1,\\
        \Zz^N & \text{ if } n=N,\\
        \Zz & \text{ if } n=N+1,\\
        0 & \text{ if } n\geq N+2.
    \end{cases}
    \]
\end{theorem}
\begin{proof}
    Remark~\ref{rmk:AF} implies that the LHS type spectral sequence
    \[E_{pq}^2(\cG_\tor,\varphi)=\H_p(\S,\H_q(\cG_\tor\times_\varphi\S))\Longrightarrow \H_{p+q}(\cG_\tor), 
    \]
    collapses at the $E^2$-page, so that we have 
    \[ 
    E^2_{p,0}(\cG_\tor,\varphi) = E^\infty_{p,0} = \H_p(\S, \H_0(\cG_\tor \times_\varphi \S)) = \H_p(\S, \H_0(\ker \varphi))
    \]
    On the other hand, $E^2$-page of the LHS type spectral sequence 
    \[   E_{pq}^2(\cG_{S\acts\Zz},\rho)=\H_p(\S,\H_q(\cG_{S\acts\Zz}\times_\rho\S))\Longrightarrow \H_{p+q}(\cG_{S\acts\Zz})
    \]
    is 
    \[ E^2_{pq}(\cG_{S\acts\Zz},\rho) = \H_p(\S, \H_q(\cG_{S\acts\Zz} \times_\rho \S)) = \H_p (\S, \H_q(\Zz \ltimes \Delta))\]
    for $q=0,1$, and is zero if $q > 1$ by Lemma~\ref{lem:betaisom}. 
    By Theorem~\ref{thm:LHSfunctorial} and Example~\ref{ex:etalehom}, $\alpha$ induces  
    a morphism $E^*(\cG_\tor,\varphi) \to E^*(\cG_{S\acts\Zz}, \rho)$ of spectral sequences, which gives rise to a morphism between long exact sequences as follows: 
    \[ 
    \begin{tikzcd}
      E^2_{p+1,0} \ar[r, "0"] \ar[d] & E^2_{p-1,1}(\cG_\tor, \varphi) \ar[r] \ar[d] & \H_p(\cG_\tor) \ar[r] \ar[d]& E^2_{p,0}(\cG_\tor, \varphi) \ar[r, "0"] \ar[d] & E^2_{p-2,1} \ar[d] \\
      E^2_{p+1,0} \ar[r, "d"] & E^2_{p-1,1}(\cG_{S\acts\Zz}, \rho) \ar[r] & \H_p(\cG_{S\acts\Zz}) \ar[r]& E^2_{p,0}(\cG_{S\acts\Zz}, \rho) \ar[r, "d"] & E^2_{p-2,1}\nospacepunct{.}
    \end{tikzcd}
    \]
    By Lemma~\ref{lem:betaisom}, the above diagram becomes
    \[ 
    \begin{tikzcd}
     0 \ar[r] & 0 \ar[r] \ar[d] & \H_p(\cG_\tor) \ar[r, "\cong"] \ar[d]& \H_p (\S, \H_0(\ker \varphi)) \ar[r] \ar[d, "\cong"]& 0 \\
      \dots \ar[r, "d"] & \H_{p-1} (\S, \H_1(\Zz \ltimes \Delta)) \ar[r] & \H_p(\cG_{S\acts\Zz}) \ar[r]& \H_p (\S, \H_0(\Zz \ltimes \Delta)) \ar[r, "d"]& \dots \nospacepunct{.}
    \end{tikzcd}
    \]
    In particular, the differential $d \colon E^2_{p,0}(\cG_{S\acts\Zz}, \rho) \to E^2_{p-2,1}(\cG_{S\acts\Zz}, \rho)$ is zero, so that we obtain the split exact sequence 
    \begin{equation}
    \label{eqn:ses5}
        0 \to \H_{p-1}(\S) \to \H_p(\cG_{S\acts\Zz}) \to \H_p(\S, \A) \to 0.
    \end{equation}
Here, we have used that $\H_1(\Zz \ltimes \Delta)=\H_1(\A,\Zz\ol{\A})$ by Proposition~\ref{prop:equivAdditive}, the action of $\S$ on $\H_1(\Zz \ltimes \Delta)$ is trivial by Theorem~\ref{thm1} because $\lwedge^1(s)\otimes \frac{1}{[\Zz : s\Zz]}=1$ for all $s\in S$, and that $\Zz\cong \H_1(\Zz \ltimes \Delta)$ by Proposition~\ref{prop:Z[1/n]}(iii).
    Now the claim follows from Equation~\eqref{eqn:torH}. 
\end{proof}

\subsubsection{Computation of $\H_p(\S,\A)$} \label{sssec:kunneth}
In the proof of Theorem~\ref{thm:homintdyn}, we saw that $\H_*(\cG_\tor)\cong\H_*(\S,\A)$. We now give an alternate direct proof of the calculation of $\H_*(\S,\A)$ without using Equation~\eqref{eqn:torH}. We need this direct calculation to deal with the case $|\Sigma|=\infty$. We continue to assume $|\Sigma|=N$ and let $\Sigma = \{s_1, \dots, s_N\}$. 

We let $\Zz[T,T^{-1}]$ denote the Laurent polynomial ring with one variable $T$. Note that $\Zz[T,T^{-1}]$ is isomorphic to the group ring $\Zz[\Zz]$. Let $F_\bullet$ be the standard projective resolution 
\[ 0 \to \Zz[T,T^{-1}] \xrightarrow{\times (T-1)} \Zz[T,T^{-1}] \xrightarrow{T \mapsto 1} \Zz \to 0\]
of $\Zz$ over $\Zz[T,T^{-1}]$ (note that $F_0=F_1=\Zz[T,T^{-1}]$). Let $s \in \Sigma$. Then, the tensor product $F_\bullet \otimes_{\Zz[T,T^{-1}]} \Zz[1/s]$ is isomorphic to the complex 
\[ P_\bullet(s)[1/s] \colon 0 \to \Zz[1/s] \xrightarrow{\times (s-1)} \Zz[1/s] \to 0,\]
so that $\H_0(P_\bullet(s))[1/s] = \Zz/(s-1)\Zz$ and $\H_1(P_\bullet(s)[1/s])=0$. It is known that the tensor product $F_\bullet^{\otimes N}$ is a projective resolution of $\Zz$ over $\Zz[\Zz^N]$ (see \cite[Section~V.1]{Brown}). In addition, 
\[ F_\bullet^{\otimes N} \otimes \Zz[1/s_1 \dots s_N] \cong F_\bullet^{\otimes N} \otimes \bigotimes_{i=1}^N \Zz[1/s_i] \cong \bigotimes_{i=1}^N (F_\bullet \otimes \Zz[1/s_i]) = \bigotimes_{i=1}^N P_\bullet(s_i)[1/s_i]. \]
Note that tensor products are always taken as chain complexes, and the coefficient modules $\Zz[1/s_i]$ are regarded as chain complexes 
\[ 0 \to \Zz[1/s_i] \to 0, \]
where $\Zz[1/s_i]$ sits at the $0$-th degree. Therefore, we have 
\[ 
\H_*(\S, \A) \cong \H_*\left(\bigotimes_{i=1}^N P_\bullet(s_i)[1/s]\right). 
\]
\begin{lemma}\label{lem:Kunneth.torus.torsion}
    The tensor product of chain complexes $\bigotimes_{s \in \Sigma} P_{\bullet }(s)$, where 
    \[ 
    P_{\bullet}(s) \colon 0 \to \mathbb{Z} \xrightarrow{s-1} \mathbb{Z} \to 0,
    \]
    has the homology 
\[ \H_p\Big( \bigotimes_{s \in \Sigma} P_{\bullet}(s) \Big) \cong 
\begin{cases}
(\Zz/g_{\Sigma}\Zz)^{\oplus \binom{N-1}{p}} & \text{ if } 0\leq p\leq N-1,\\
0 & \text{ otherwise. }
\end{cases}
\]
\end{lemma}
\begin{proof}
Let $Q_\bullet \coloneqq \bigotimes_{i=1}^{N-1} P_\bullet(s_i)$ and $P_\bullet \coloneqq P_\bullet(s_N)$. 
Then, since $\H_0(P_{\bullet}) \cong \mathbb{Z}/(s_N-1)\mathbb{Z}$ and $\H_1(P_{\bullet})=0$, the K\"unneth formula for chain complexes 
yields the following K\"unneth exact sequence
\[ 0 \to \H_p(Q_\bullet) \otimes_{\mathbb{Z}} \mathbb{Z}/(s_N-1)\mathbb{Z} \to \H_p (Q_\bullet \otimes P_\bullet) \to \Tor^{\mathbb{Z}}_1(\H_{p-1}(Q_\bullet), \mathbb{Z}/(s_N-1)\mathbb{Z}) \to 0. \]
This exact sequence splits because each $P_n$ is a free $\mathbb{Z}$-module. 
By induction, we see that 
\[ \H_p(Q_\bullet) \otimes \Zz/(s_N-1)\Zz \cong (\Zz/g_{\Sigma}\Zz)^{\oplus \binom{N-2}{p}} \mbox{ and }
\Tor(\H_{p-1}(Q_\bullet), \Zz/(s_N-1)\Zz) \cong (\Zz/g_{\Sigma}\Zz)^{\oplus \binom{N-2}{p-1}},
\]
and hence $\H_p(\bigotimes_{s \in \Sigma} P_{\bullet }(s) ) \cong (\mathbb{Z}/g\mathbb{Z})^{\binom{n-1}{p}}$. 
\end{proof}
This Lemma \ref{lem:Kunneth.torus.torsion} and the universal coefficient theorem concludes that 
\begin{equation}
    \label{eqn:H(cS,cA)}
\H_p (\S, \A) \cong 
\begin{cases}
(\Zz/g_{\Sigma}\Zz)^{\oplus \binom{N-1}{p}} & \text{ if } 0\leq p\leq N-1,\\
0 & \text{ otherwise,} 
\end{cases}
\end{equation}
which recovers \eqref{eqn:torH}.

\subsubsection{The infinite case}
\label{sec:infinitecase}
Now we consider the case that $|\Sigma|=\infty$.

\begin{theorem} \label{thm:BOShomology}
    If $|\Sigma|=\infty$, then we have 
    \[
    \H_p(\cG_{S\acts\Zz}) \cong \begin{cases}
    \Zz/g_\Sigma \Zz & \text{ if } p=0,\\
        \Zz^{\oplus\infty} \oplus (\Zz/g_\Sigma \Zz)^{\oplus \infty} & \text{ if } p\geq 1.
    \end{cases}
    \]
\end{theorem}
\begin{proof}
    Fix a finite subset $\Sigma_0 \subseteq \Sigma$ such that $g_{\Sigma_0} = g_{\Sigma}$. Let $\Sigma_k \subseteq \Sigma$ be an increasing sequence of finite subsets of $\Sigma$, starting with $\Sigma_0$, such that $\Sigma_{k+1} = \Sigma_k \sqcup \{s_k\}$ and $\bigcup_k \Sigma_k = \Sigma$. Then, we have $g_{\Sigma_k} = g_\Sigma$ for all $k \geq 0$. 
    For each $k$, let $\A_k$, $S_k$, $\S_k$, $\Delta_k$, $\cG_k\coloneqq\cG_{S_k\acts\Zz}$, and $\cP_k$ be associated with $\Sigma_k$ as in Section~\ref{sec:BOSprelims}.  
    Let $\pi_k\colon\Delta_{k+1}\to \Delta_k$ be the canonical projection map and $j_k\colon \A_k\rtimes \S_k\to \A_{k+1}\rtimes \S_{k+1}$ the canonical inclusion map.
By Theorem~\ref{thm:LHSfunctorial}, Lemma~\ref{lem:technical}, and Proposition~\ref{prop:isomfunctors}, we have the following morphism between the exact sequences from Equation~\eqref{eqn:ses5}: 
   \begin{equation}
        \begin{tikzcd}
            0 \arrow[r] & \H_{p-1}(\S_{k+1},\H_1(\Zz,\Zz\Delta_{k+1})) \arrow[r] & \H_p(\cG_{k+1}) \arrow[r] & \H_p(\S_{k+1}, \H_0(\Zz,\Zz\Delta_{k+1})) \arrow[r] & 0 \\
            0 \arrow[r] & \H_{p-1}(\S_k,\H_1(\Zz,\Zz\Delta_k)) \arrow[r]\arrow[u,"\H_{p-1}(j_k\text{,}\,\H_1(\id\text{,}\,\pi_k^*))"] & \H_p(\cG_k) \arrow[r]\arrow[u,"\H_p(\Omega_k)"] & \H_p(\S_k, \H_0(\Zz,\Zz\Delta_k)) \arrow[r]\arrow[u,"\H_p(j_k\text{,}\,\H_0(\id\text{,}\,\pi_k^*))"] & 0\nospacepunct{,}
        \end{tikzcd}
    \end{equation}
where $\Omega_k\colon \cG_k\to\cG_{k+1}$ is the \'etale correspondence associated with $(j_k,\pi_k)$ as in Lemma~\ref{lem:technical}. To see that the equivariance condition in Definition~\ref{def:equiv} is satisfied, notice that $\S_k$ acts on $\prod_{p\in\cP_{k+1}\setminus\cP_k}\Zz_p$ by homeomorphisms.

Next, we need to check that, under the identification as in Equation~\eqref{eqn:odhom}, $\H_0(\id,\pi_k^*)$ is identified with $\id_\Zz$ and $\H_1(\id,\pi_k^*)$ is identified with the inclusion map $\iota_k \colon \A_k \hookrightarrow \A_{k+1}$.
For $q=1$, we have the following commutative diagram
\begin{equation}
    \begin{tikzcd}
       \H_1(\Zz)\arrow[r,"\rho_{1,\infty}"] \arrow[d,"\id"]& \H_1(\Zz,\Zz\Delta_k)\arrow[d,"\H_1({\id,\pi_k^*})"] \\ 
       \H_1(\Zz) \arrow[r,"\rho_{1,\infty}"]& \H_1(\Zz,\Zz\Delta_{k+1})\nospacepunct{,}
    \end{tikzcd}
\end{equation}
where the horizontal maps are from \eqref{eqn:rhoinfty}. They are isomorphisms by Proposition~\ref{prop:Z[1/n]}. 
For $q=0$, consider the diagram 
\begin{equation}
    \begin{tikzcd}
   \H_0(\Zz)\otimes \A_k\arrow[d,"\id"] \arrow[r,"\tilde{\rho}_{1,\infty}"] & \H_0(\Zz,\Zz\Delta_k) \otimes\A_k=\H_0(\Zz,\Zz\Delta_k)\arrow[d,"\H_0({\id,\pi_k^*})"] \\
   \H_0(\Zz)\otimes \A_k \arrow[d,"\incl"]\arrow[r,"\rho_{1,\infty}\otimes\id_{\A_k}"]& \H_0(\Zz,\Zz\Delta_{k+1}) \otimes\A_k=\H_0(\Zz,\Zz\Delta_{k+1})\arrow[d,"\id"] \\
   \H_0(\Zz)\otimes \A_{k+1} \arrow[r,"\tilde{\rho}_{1,\infty}"]& \H_0(\Zz,\Zz\Delta_{k+1})\otimes\A_{k+1}=\H_0(\Zz,\Zz\Delta_{k+1})\nospacepunct{,}
    \end{tikzcd}
    \end{equation}
    where the top and bottom horizontal maps are the isomorphisms from Proposition~\ref{prop:invertingrho}, and  $\rho_{1,\infty}$ is the map from \eqref{eqn:rhoinfty}. Proposition~\ref{prop:Z[1/n]} gives us the three equalities on the right hand side. It is easy to see that each of the smaller squares commutes, so the larger square also commutes.
    
Consequently, we have the following morphism of short exact sequences:
\begin{equation}
        \begin{tikzcd}
            0 \arrow[r] & \H_{p-1}(\S_{k+1}, \Zz) \arrow[r] & \H_p(\cG_{k+1}) \arrow[r] & \H_p(\S_{k+1}, \A_{k+1}) \arrow[r] & 0 \\
            0 \arrow[r] & \H_{p-1}(\S_k, \Zz) \arrow[r] \arrow[u,"\H_p({j_k,\id})"] & \H_p(\cG_k) \arrow[r] \arrow[u,"\H_p(\Omega_k)"] & \H_p(\S_k, \A_k) \arrow[r] \arrow[u,"\H_p({j_k,\iota_k})"] & 0\nospacepunct{.}
        \end{tikzcd}
    \end{equation}

By continuity of group homology (\cite[Exercise~V.5.3]{Brown} and continuity of Tor) combined with continuity of groupoid homology (Proposition~\ref{prop:continuity}), we obtain the following split exact sequence: 
    \[ 
    0 \to \H_{p-1}(\S, \Zz) \to \H_p(\cG_{S\acts\Zz}) \to \H_p(\S, \A) \to 0. 
    \]

    We have $\H_{p-1} (\S, \Zz) = \lwedge^{p-1} \S \cong \Zz^{\oplus\infty}$. In addition, we also have
    \begin{align*} \H_p(\mathscr{S}_k,\mathscr{A}_{\sigma}) \cong {}& \H_p(\mathscr{S}_k,(\mathscr{A}_k)_{\sigma}) \otimes_{\mathscr{A}_k} \mathscr{A} \oplus \Tor^{\mathscr{A}_k}_1(\H_p(\mathscr{S}_k,(\mathscr{A}_k
    )_{\sigma}), \mathscr{A})\\
    \cong {}& (\mathbb{Z}/g_{\Sigma}\mathbb{Z})^{\binom{|\Sigma_k|-1}{p-1}} \otimes_{\mathscr{A}_k} \mathscr{A} \oplus \Tor^{\mathscr{A}_k}_1((\mathbb{Z}/g_{\Sigma}\mathbb{Z})^{\binom{|\Sigma_k|-1}{p-1}}, \mathscr{A})\\
    \cong {}&(\mathbb{Z}/g_{\Sigma}\mathbb{Z})^{\binom{|\Sigma_k|-1}{p-1}},
    \end{align*}
    by the universal coefficient theorem (here $\mathscr{A}_\sigma$ denotes the $\mathscr{A}$-module $\mathscr{A}$ on which $\mathscr{S}$ acts nontrivially by $\sigma$) and hence
    \[
    \H_p(\mathscr{S},\mathscr{A}_{\sigma}) \cong \varinjlim_{k} \H_p(\mathscr{S}_k,\mathscr{A}_\sigma) \cong \varinjlim_{k} (\mathbb{Z}/g_{\Sigma}\mathbb{Z})^{\binom{n_k-1}{p-1}} \cong (\mathbb{Z}/g_{\Sigma}\mathbb{Z})^{\infty}. \qedhere
    \]
\end{proof}

\begin{remark}
If $g_\Sigma$ is odd, then the map $\zeta\colon \H_0(\cG_{S\acts\Zz}) \otimes \Zz/2\Zz \to \fg{\cG_{S\acts\Zz}}^{\rm ab}$ from \cite[Corollary~E]{Li:TFG} vanishes because $\H_0(\cG_{S\acts\Zz}) \otimes \Zz/2\Zz=0$, so that $\fg{\cG_{S\acts\Zz}}^{\rm ab}\cong \H_1(\cG_{S\acts\Zz})$. If $g_\Sigma$ is even, then the situation is similar to that for SFT groupoids in \cite[Section~5.5]{Mat15}, i.e., $\zeta$ sometimes vanishes, nevertheless $\H_0(\cG_{S\acts\Zz}) \otimes \Zz/2\Zz $ is nontrivial. 
\end{remark}

\subsection{Calculation of K-theory}
\label{sec:BOSKtheory}

Barlak, Omland, and Stammeier conjectured the following result:
\begin{theorem}[{\cite[Conjecture~6.5]{BOS}}]
\label{thm:BOS}
For any set $\Sigma\subseteq\Zz_{>1}$ of pairwise relatively coprime numbers with $|\Sigma|\geq 2$, the C*-algebra $C_r^*(\cG_\tor)$ is isomorphic to $\bigotimes_{s\in\Sigma}\cO_s$, where $\cG_\tor$ is the groupoid from \eqref{eqn:torsiongpoid}.
\end{theorem}
Since it is already shown that $C^*_r(\cG_{\mathrm{tor}})$ is a UCT Kirchberg algebra (\cite[Corollary~5.2]{BOS}), this theorem is equivalent to $\K_0(C^*_r(\cG_{\tor})) \cong \K_0(C^*_r(\cG_{\tor})) \cong  (\Zz / g_{\Sigma} \Zz)^{2^{|\Sigma| -1}}$ and $n[1] =0$ if and only if $g_{\Sigma} \mid  n$.
By \cite[Theorem~6.1]{BOS} and Theorem \ref{thm:homintdyn}, this means that the groupoid $\cG_{S \curvearrowright \mathbb{Z}}$ has the HK property. 
In this subsection, we prove Theorem~\ref{thm:BOS}. Hereafter, we assume that $N=|\Sigma| <\infty$, and the case of $|\Sigma| =\infty$ is dealt with in the last paragraph of the proof of Theorem \ref{thm:BOS} below.  

First of all, we recall a presentation of the torsion subalgebra $C^*_r(\mathcal{G}_{\mathrm{tor}})$ from \cite{BOS} by a semigroup crossed product $\Mz_{\mathfrak{s}^\infty}\rtimes\Nz^N$, where $\Mz_{\mathfrak{s}^\infty}$ is the UHF algebra (see \cite[Corollary~5.4]{BOS}).
We first recall the definition of this $\mathbb{N}^N$-action on $\mathbb{M}_{\mathfrak{s}^\infty}$. 
Let $\Sigma=\{ s_i\}_{i=1}^N$ be a finite set of mutually coprime positive integers.
Set $\mathfrak{s}:=\prod_{i=1}^N s_i$ and $g=\gcd (\{s_i-1 \})$. 
The space of $\mathfrak{s}$-adic integers is denoted by $\Delta  := \varprojlim \Delta_{\mathfrak{s}^l}$, where $\Delta_m:= \mathbb{Z} / m \mathbb{Z}$. 
The inclusion $C(\Delta _{\mathfrak{s}^l}) \hookrightarrow C(\Delta_{\mathfrak{s}^{l+1}})$ extends to $\mathbb{K}(\ell^2\Delta_{\mathfrak{s}^l}) \to \mathbb{K}(\ell^2 \Delta_{\mathfrak{s}^{l+1}})$ by sending the matrix unit $e_{x,y}$ to $\sum_{a \in \mathbb{Z}/\mathfrak{s}\mathbb{Z}} e_{x+\mathfrak{s}^la,y+\mathfrak{s}^la}$ for any $x, y \in \Delta_{\mathfrak{s}^l}$. 
We write $B:=\varinjlim \mathbb{K}(\ell^2\Delta_{\mathfrak{s}^l})$. 
The product by $s_i$ gives a (non-unital) $\ast$-homomorphism $\beta_i \colon \mathbb{K}(\ell^2\Delta_{\mathfrak{s}^l /s_i}) \to \mathbb{K}(\ell^2\Delta_{\mathfrak{s}^{l}})$. More explicitly, $\beta_i(e_{x,y})=e_{s_ix,s_iy}$ for any $x,y \in \Delta_{\mathfrak{s}^l /s_i}$. It is assembled to an endomorphism $\beta_i \colon B \to B$. 
Let $\widetilde{B}\coloneqq \varinjlim_{\Nz^N}\{B; \beta_i\}$ be the direct limit with connecting maps coming from the $\Nz^N$-action defined by $\beta_i$.
Then, $1_B(\widetilde{B} \rtimes \mathbb{Z}^N)1_B$ is isomorphic to $C^*_r(\cG_{\tor})$ (cf.~\cite[Theorem 2.2.1]{Laca}).

Following the line of \cite{Kasparov} and \cite{BOS}, we rephrase the K-theory of the crossed product by $\mathbb{Z}^N$ in terms of bundles (or continuous fields) of C*-algebras over the classifying space $\mathbb{T}^N =B\mathbb{Z}^N$. This enables us to apply the Atiyah--Hirzebruch (AH) spectral sequence (see e.g.~\cite{HJJS}) for computing the K-group of $A \rtimes_\alpha \mathbb{Z}^N$ for a $\mathbb{Z}^N$-C*-algebra $A$. This is the same thing as the spectral sequence of \cite[Theorem 6.7]{BOS} (we also refer to \cite{SavinienBellissard,barlak}).
\begin{lemma}\label{lem:Dirac.dualDirac}
    Let $A$ be a $\mathbb{Z}^N$-C*-algebra. Then, the C*-algebra defined by
    \[ \mathsf{MT}(A,\alpha) = \{ a \in C([0,1]^N ,A) : a(t_1,\cdots, t_{j-1},0,t_{j+1},\cdots, t_N) =\alpha_j (a(t_1,\cdots, t_{j-1},1,t_{j+1},\cdots, t_N)) \} \]
    satisfies $\K_*(A \rtimes \mathbb{Z}^N) \cong \K_{*-N}(\mathsf{MT}(A,\alpha ))$.
\end{lemma}
The proof can be found in e.g.\ \cite[Theorem 1.2.6]{barlak2}.
There, it is first shown that $(A \rtimes \mathbb{Z}^N) \otimes C_0(\mathbb{R}^N)$ and $(A \otimes C_0(\mathbb{R}^{N})) \rtimes \mathbb{Z}^N $ have the same K-theory, and then the latter C*-algebra is stably isomorphic to the mapping torus. 
The first step is proved either via the Connes–Thom isomorphism, as in \cite{barlak2}, or via the fact that $A$ and $A \otimes C_0(\mathbb{R}^{2N})$ are $\mathbb{Z}^{N}$-equivariantly KK-equivalent, as established by Kasparov’s Dirac–dual Dirac method \cite{Kasparov}. 
The second step is given by the $\ast$-isomorphism
    \[ 
    (A \otimes C_0(\mathbb{R}^N)) \rtimes \mathbb{Z}^N \xrightarrow{\cong } \mathsf{MT}(A \otimes \mathbb{K}(\ell^2\mathbb{Z}^N),\beta \otimes \Ad \lambda), \quad (a \otimes f )\lambda_g \mapsto \Big( t \mapsto a\otimes \sum_{h \in \mathbb{Z}^N} f(t+h) \cdot e_{g+h,h} \Big),
    \]
    where $\lambda $ and $e_{gh}$ denotes the left regular representation and the matrix unit on $\ell^2\mathbb{Z}^N$ respectively. Indeed, the right hand side is Morita equivalent to $\mathsf{MT}(A,\alpha )$ via the imprimitivity bimodule $\mathsf{MT}(A \otimes \ell^2 \mathbb{Z}^N,\alpha \otimes \lambda ) $ defined in the same way as the mapping torus of C*-algebras.

The C*-algebra $\mathsf{MT}(A,\alpha)$ is regarded as the continuous section algebra of the mapping torus bundle of C*-algebras
\[ 
    \mathcal{MT}(A,\alpha) \coloneqq [0,1]^N \times A  / \{ (t_1,\cdots,t_{j-1},1,t_{j+1},\cdots , t_N,a) \sim  (t_1,\cdots,t_{j-1},0,t_{j+1},\cdots , t_N,\alpha_j(a))\}. 
\]
The definition of this bundle can be generalized to the case that $\beta$ is an endomorphism action of $\mathbb{N}^N$ with a little effort: we need to replace the fiber $A$ over $t \in [0,1]^N$ with $\Im (\alpha_{k_1} \circ \cdots \circ \alpha_{k_s})$ if $t_{k_l} = 0$ for $l=1,\cdots ,s$ (a more precise definition is given below in Definition \ref{def:mapping.torus}, covering the more general case). 
In this case, the above $\mathcal{MT}(A,\alpha)$ remains well-defined as a continuous field of C*-algebras in the sense of \cite[Definitions 10.1.2, 10.3.1]{Dixmier}. However, it is no longer a bundle of C*-algebras in the sense that it does not have a local trivialization, i.e., there is $p \in \mathbb{T}^N$ such that any open neighborhood $p \in U$ does not admit a fiber-preserving homeomorphism $\mathcal{MT}(A,\alpha)|_{U} \cong U \times A$ that is a $\ast$-isomorphism at each fiber.
Indeed, in general, a C*-algebra over a compact Hausdorff space (i.e., a C*-algebra $A$ equipped with a non-degenerate $\ast$-homomorphism $C(X) \to \mathcal{Z}\mathcal{M}(A)$) 
 is always regarded as the continuous section algebra $C(\mathbb{T}^N,\mathcal{A})$ of an upper-semicontinuous field of C*-algebras $\mathcal{A}$ (cf. \cite[Section 1]{Blanchard}). 

The notion of a mapping torus can be further generalized to the case where not only the range but also the domain of each $\beta_i$ is a proper subalgebra of $B$. The precise definition is given as follows. 

\begin{definition}\label{def:mapping.torus}
    Let $A$ be a C*-algebra. We call a triple $(\alpha ,D_\alpha,R_\alpha )$ (denoted by $\alpha$ in short) a partial $\ast$-automorphism if $D_\alpha$, $R_\alpha$ are C*-subalgebras of $A$ and $\alpha \colon D_\alpha \to R_\alpha$ is a $\ast$-isomorphism. The composition of partial $\ast$-automorphisms is given by the triple $(\alpha_2 \circ \alpha_1, \alpha_1^{-1}(D_{\alpha_2}), \alpha_2(D_{\alpha_2} \cap R_{\alpha_1}))$. 
    For a mutually commutative $N$-tuple of partial $\ast$-automorphisms $\alpha_1, \cdots ,\alpha_N$, its mapping torus field is defined as
    \[ 
    \mathcal{MT}(A,\alpha):=\Big\{ (t,b) \in [0,1]^N \times A : \begin{array}{l}
    \text{$b \in R_{\alpha_{j_1} \circ \cdots \circ \alpha_{j_r}}$ if $t_{j_l} =0$ for $l=1,\cdots,r$} \\ 
     \text{$b \in D_{\alpha_{k_1} \circ \cdots \circ \alpha_{k_s}}$ if $t_{k_l} =1$ for $l=1,\cdots,s$}
     \end{array}
     \Big\}/\sim,
    \]
    where $\sim $ is the equivalence relation given by
    \[
       ((t_1,\cdots, t_{j-1},0,t_{j+1},\cdots, t_N),b) = ((t_1,\cdots, t_{j-1},1,t_{j+1},\cdots, t_N), \alpha_j(b)).
    \]
\end{definition}
When $D_\alpha \subsetneq A$, this definition can be seen as a C*-algebraic and higher-dimensional analogue of the HNN extension of groups.

Next, we recall some generalities of the AH spectral sequence for K-theory of C*-algebras over a compact Hausdorff space $X$.
Compare the following definition with the AH spectral sequence of topological K-theory (which is found in e.g.\ \cite[Definition 21.3.4]{HJJS}). We also note that the same construction appears in \cite[Appendix A]{barlak}.
Let us fix an increasing sequence $\emptyset =X_{-1} \subset X_0 \subset X_1 \subset \cdots \subset X_n =X$ of closed subspaces of $X$. 
The AH spectral sequence for a continuous field $\mathcal{D}$ of C*-algebras over $X$, associated to this filtration of $X$, is the spectral sequence associated to the exact couple 
\begin{gather}
    D_1^{pq} \coloneqq \K_{-p-q}(C_0(X \setminus X_{p-1}, \mathcal{D})), \quad E^{pq}_1 =  \K_{-p-q} (C_0(X_p \setminus X_{p-1}, \mathcal{D}) ),
\end{gather}
whose structure maps 
\[
    D_1^{p+1,q-1} \xrightarrow{i} D_1^{p,q} \xrightarrow{j} E_1^{p,q} \xrightarrow{k} D_1^{p+1,q-2}
\]
are given by the morphisms of six-term exact sequences associated to 
\[
    0 \to C_0(X \setminus X_p , \mathcal{D}) \to C_0(X \setminus X_{p-1}, \mathcal{D}) \to C_0(X_{p} \setminus X_{p-1} ,\mathcal{D}) \to 0. 
\]
Here, for a locally closed subspace $U \subset X$, $C_0(U, \mathcal{D})$ denotes the continuous section algebra on it vanishing at its boundary, i.e., 
\[ C_0(U, \mathcal{D}):=\{ f \in C(\overline{U},\mathcal{D}) : f|_{\overline{U} \setminus U} \equiv 0 \}. \]

By the general theory of the spectral sequence associated to an exact couple (see \cite[Subsection 6.1]{Schochet} for our setting), the first differential $d_1^{pq} \colon E_1^{pq} \to E_1^{p+1,q}$ is given by the composition $j\circ k$, in other words, the boundary map associated to the exact sequences 
\[ 0 \to C_0(X_{p} \setminus X_{p-1}, \mathcal{D}) \to C_0(X_p \setminus X_{p-2} , \mathcal{D}) \to C_0(X_{p-1}\setminus X_{p-2}, \mathcal{D}) \to 0. \]
This spectral sequence converges to $\K_*(C_0(X,\mathcal{D}))$ by the same reason as topological K-theory (for the precise proof, see \cite[Lemma 6.7]{Schochet}). 

\begin{remark}\label{rmk:d1.mapping.torus}
In the case of $C(\mathbb{T}^N,\mathcal{A}) =\mathsf{MT}(A,\alpha)$, where $\mathcal{A}\coloneqq \mathcal{MT}(A,\alpha)$ associated to a mutually commuting $N$-tuple $\{\alpha_i\}$ of partial $\ast$-automorphisms on a C*-algebra $A$, we can employ the standard filtration of $\mathbb{T}^N$, i.e., 
\[
    \mathbb{T}^N_p = \bigsqcup_{I \subset \{ 1,\cdots,N\}, |I| \leq p}\mathbb{T}^I \subset \mathbb{T}^{\{1,\cdots,N\}} =\mathbb{T}^N.
\] 
For each $p \geq 0$, the continuous field $\mathcal{A}$ has a trivialization on $\mathbb{T}^N_p \setminus \mathbb{T}^N_{p-1} \cong \bigsqcup_{|I|=p}(0,1)^I$. Hence we have 
\[
    \K_{-p-q}(C_0(\mathbb{T}^N_p \setminus \mathbb{T}^N_{p-1}, \cA)) \cong \bigoplus _{I \subset \{ 1,\cdots,N\}, |I| = p} \K_{-q} (D_{\alpha_{I^c}}),
\]
where $\alpha_I \coloneqq \alpha_{i_1}\circ \cdots \circ \alpha_{i_p}$ if $I=\{i_1,\cdots,i_p\}$. 
Moreover, for each $I, J \subset \{ 1,\cdots,N\}$ with $|I|=p$ and $|J|=p-1$, the corresponding component of the differential $d_1^{pq} \colon \K_{-q}(D_{\alpha_{J^c}}) \to \K_{-q}(D_{\alpha_{I^c}})$ is the boundary map of the exact sequence
\begin{align}
    0 \to C_0((\mathbb{T}\setminus \{0\})^I ,\mathcal{A}) \to C_0((\mathbb{T}\setminus \{0\})^I \cup (\mathbb{T}\setminus \{0\})^J , \mathcal{A}) \to C_0((\mathbb{T}\setminus \{0\})^J,\mathcal{A}) \to 0.
    \label{eqn:boundary.d1.AHSS}
\end{align}
For example, when $N=1$, then the above exact sequence for $I=\{1\}$ and $J=\emptyset$ is 
\[
    0 \to C_0(\mathbb{T}\setminus \{0\} ,A) \to \{ b \in C([0,1],A) \mid b(1) \in D_\alpha, \alpha(b(1))=b(0) \}  \xrightarrow{\ev_1} D_{\alpha}\to 0.
\]
In this case, by comparing the six-term exact sequence of 
\[
\xymatrix{
    0 \ar[r] & C_0(\mathbb{T}\setminus \{0\},A) \ar[r] \ar@{=}[d] & C(\mathbb{T}^N,\mathcal{MT}(A,\alpha)) \ar[r] \ar[d] & D_\alpha \ar[r] \ar[d]^{\iota \oplus \alpha}  & 0 \\
    0 \ar[r] & C_0(\mathbb{T}\setminus \{0\},A) \ar[r]  & C([0,1],A) \ar[r]  & A \oplus A \ar[r]   & 0 ,
}
\]
where $\iota \colon D_\alpha \to A$ denotes the inclusion, the boundary map turns out to be $ \alpha_* - \iota_* \colon \K_*(D_{\alpha}) \to \K_{*}(A)$.  
By the same reason, for general $I,J$, the boundary map of K-groups associated to \eqref{eqn:boundary.d1.AHSS} is 
\[
    d_{1}^{pq}= \sum_{|I|=p, \ I=J \cup \{ i \}} \Big( (\alpha_i)_* - \iota_* \colon \K_0(D_{\alpha_{J^c}}) \to \K_0(D_{\alpha_{I^c}}) \Big) . 
\]

When the inclusions $D_{\alpha_I} \to A$ induce an isomorphism on $\K$-groups, as in the case we are interested in, the groups $E_1^{pq}$ are isomorphic to direct sums of copies of $\K_{-q}(A)$. In this case, the resulting $E_1$-page is nothing else than the cellular cohomology group of $\mathbb{T}^N$ with respect to the cellular decomposition $\mathbb{T}^N=\bigsqcup_{p=0}^{N} (\mathbb{T}^N_p \setminus \mathbb{T}_{p-1}^N)$, with the coefficient $\K_{-q}(A)$ twisted by the monodromy $\alpha_*$ (for the definition of the cohomology twisted by a local system, we refer the readers to \cite[3.H]{Hatcher}). 
\end{remark}

\begin{remark}\label{rem:sheaf}
Although it is not used in the rest of the paper, we remark that the above discussion is a special case of a general story: In general, the $E_2$-page computed from the above $E_1$-page is identical with the sheaf cohomology $\H^p(X; \mathcal{K}_{-q}(\mathcal{D}))$, where the sheaf $\mathcal{K}_{-q}(\mathcal{D})$ is the sheafification of the presheaf $U \mapsto \K_{-q}(C_b(U,\mathcal{D}))$ (\cite[Proposition 5.2]{Segal2}). 
In particular, this $E_2$-page is independent of the choice of a cellular decomposition of $X$. 

When the continuous field $\mathcal{D}$ comes from a locally trivial bundle of C*-algebras with the fiber $D$ defined by the transition functions $g_{ij} \colon U_i \cap U_j \to \mathop{\mathrm{Aut}}(D)$, the sheaf $\mathcal{K}_{-q}(\mathcal{D})$ comes from the local system (i.e., a flat $\K_{-q}(D)$-bundle) whose transition functions are given by $(g_{ij})_*$. Especially, when $\mathcal{D} = \mathcal{MT}(A,\alpha)$ as before, the sheaf $\mathcal{K}_{-q}(\mathcal{MT}(A,\alpha))$ is identical with the local system coming from the flat $\K_{-q}(B)$-bundle
\[ \underline{\K_{-q}(A)}:= [0,1]^N \times \K_{-q}(A) /(((t_1, \cdots, t_{i-1},0,t_{i+1}, \cdots, t_N), x) \sim ((t_1, \cdots, t_{i-1},1,t_{i+1}, \cdots, t_N), (\alpha_i)_*(x))). \]
\end{remark}

Set $B_l:=\mathbb{K}(\ell^2\Delta_{\mathfrak{s}^l}) \subset B$. 
Then the restriction of endomorphisms $\beta_i$ to $B_l$ gives a mutually commuting $n$-tuple of partial $\ast$-automorphism $(\beta_i , \beta_i^{-1}(B_l), \beta_i(\beta_i^{-1}(B_l)))$, denoted by $\beta_i$ or $(\beta_i,D_{l,\beta_i}, R_{l,\beta_i})$ in short.  
To be more explicit, $D_{l,\beta_i} $ is $ \mathbb{K}(\ell^2\Delta_{\mathfrak{s}^l/s_i})$ regarded as a subalgebra of $B_l$ by an extension of $C(\Delta_{\mathfrak{s}^l/s_i}) \to C(\Delta_{\mathfrak{s}^l})$ induced from the projection $\Delta_{\mathfrak{s}^l} \to \Delta_{\mathfrak{s}^l/s_i}$, and $R_{l,\beta_i} = \mathbb{K}(\ell^2(s_i\Delta_{\mathfrak{s}^l/s_i}))$ regarded as a full corner subalgebra of $B_l$. We define
\[ \mathcal{B}_l:=\mathcal{MT} (B_l,\{\beta_i\}_{i=1}^N), \quad \mathcal{B}:=\mathcal{MT} (B,\{\beta_i\}_{i=1}^N), \quad \widetilde{\mathcal{B}}:=\mathcal{MT} (\widetilde{B},\{\beta_i\}_{i=1}^N).\]

\begin{lemma}[cf. {\cite[Proposition 6.12]{BOS}}]\label{lem:AHSS.torsion.K}
    The $E_2$-pages of the AH spectral sequence computing
    \[ 
    \K_*(C(\mathbb{T}^N,\widetilde{\mathcal{B}})) \cong \K_*(\mathsf{MT}(\widetilde{B},\beta)) \cong \K_*(\widetilde{B}\rtimes \mathbb{Z}^N),
    \]
    where the isomorphisms come from Lemma \ref{lem:Dirac.dualDirac} and the discussion below, is isomorphic to the cohomology group
    \[ E_2^{pq}  \cong 
    \begin{cases}
       (\mathbb{Z}/g \mathbb{Z})^{\binom{N-1}{p-1}} & \text{if $q$ is even}, \\
        0 & \text{if $q$ is odd}.
    \end{cases} \]
    Moreover, for any $l \geq 1$ and $r \geq 1$, the inclusions $\mathcal{B}_l \to \mathcal{B} \to \widetilde{\mathcal{B}}$ induce isomorphisms of $E_r$-pages of the AH spectral sequences. 
    In particular, they induce isomorphisms of K-groups of the global section algebras. 
\end{lemma}
\begin{proof}
We write $s_I \coloneqq \prod_{i \in I} s_i$ for $I \subset \{1,\dots,N\}$.
Recall that $B$ is an AF algebra with $\K_0(B) \cong \mathbb{Z}[\mathfrak{s}^{-1}]$.
We fix this isomorphism so that a rank $1$ projection of $B_l$ corresponds to $1 \in \mathbb{Z}[\mathfrak{s}^{-1}]$.
Under this identification, for each $I$, the $\K_0$-group of the matrix subalgebra $D_{\beta_I} \cong \mathbb{K}(\ell^2\Delta_{\mathfrak{s}^l / s_I}) \subset B_l \subset B$ is identified with its image in $\K_0(B)$, namely the subgroup $s_I \mathbb{Z} \subset \mathbb{Z}[\mathfrak{s}^{-1}]$. 
Moreover, since each partial $\ast$-homomorphism
$\beta_i \colon D_{\beta_{I}} \to D_{\beta_{I \setminus \{i\}}}$ sends a rank $1$ projection to a rank $s_i$ one, it induces multiplication by $s_i^{-1}$ on $\K_0$. 
This also implies that $\beta_{i,*} = \K_0(B) \to \K_0(B)$ is also given by multiplication by $s_i^{-1}$. 
In particular, $\beta_{i,*}$ are isomorphisms. This shows $\K_0(B) \cong \K_0(\widetilde{B})$, on which $\mathbb{Z}^N$ acts by $s_i^{-1}$. 

Consider the AH spectral sequence $E^{pq}_r$ associated to $\widetilde{\mathcal{B}}$. Then we have $E_1^{p,2q+1}\cong 0$, and by the above discussion and Remark \ref{rmk:d1.mapping.torus}, the cochain complex $(E_1^{p,2q},d_{1}^{p,2q})$ is decomposed into the tensor product of length $2$ complexes as
\begin{align}
\begin{split}
    E_1^{*,2q} \cong {}&{} \bigg( 0 \to \bigoplus_{|I_0|=0}\mathbb{Z}[\mathfrak{s}^{-1}] \cdot x_{I_0}
    \xrightarrow{\partial^0} 
    \bigoplus_{|I_1|=1} \mathbb{Z}[\mathfrak{s}^{-1}] \cdot x_{I_1}
    \xrightarrow{\partial^1} 
     \cdots  \xrightarrow{\partial^{N-1}} \bigoplus_{|I_N|=N} \mathbb{Z}[\mathfrak{s}^{-1}]\cdot x_{I_N} \to 0  \bigg)  \\
    \cong {}&{} \bigotimes_{s \in \Sigma} (0 \to \mathbb{Z}[\mathfrak{s}^{-1}] \xrightarrow{s^{-1} -1} \mathbb{Z}[\mathfrak{s}^{-1}] \to 0) \\
    \cong {}&{}  \bigotimes_{s \in \Sigma} (0 \to \mathbb{Z}[\mathfrak{s}^{-1}] \xrightarrow{s-1} \mathbb{Z}[\mathfrak{s}^{-1}] \to 0),
\end{split}\label{eqn:E1.B1}
\end{align}
where each $I_j $ denotes a subset of $\{1,\cdots,N\}$, $x_I$ denotes the generator of $\mathrm{K}_0(C_0((0,1)^I, \widetilde{\mathcal{B}}) )$ corresponding to $1 \in \mathbb{Z}[\mathfrak{s}^{-1}]$ under the above identification, and
\[
    \partial^j =\bigoplus_{I_{j+1}=I_j \cup \{s_i\} } \Big( \sigma(I_j,s_i) \cdot (s_i^{-1}-1) \colon \mathbb{Z}[\mathfrak{s}^{-1}] \cdot x_{I_j} \to \mathbb{Z}[\mathfrak{s}^{-1}] \cdot x_{I_{j+1}}\Big).
\]
Here, $\sigma(I_j,s_i) \coloneqq (-1)^{m-1}$ if $s_i$ is the $m$-th element when the elements of $I_j\cup\{s_i\}$ are arranged in increasing order.
The second isomorphism is given by identifying the direct summand $\mathbb{Z}[\mathfrak{s}^{-1}] \cdot x_{I_j}$ with the tensor factor in which, for each $s \in \Sigma$, one takes the left copy of $\mathbb{Z}[\mathfrak{s}^{-1}]$ if $s \notin I_j$ and the right copy if $s \in I_j$.
The third isomorphism is given by the commutative diagram below:  
\[
    \xymatrix{
        \mathbb{Z}[\mathfrak{s}^{-1}] \ar[r]^{s^{-1}-1} \ar[d]^{-s^{-1}} & \mathbb{Z}[\mathfrak{s}^{-1}] \ar@{=}[d] \\
        \mathbb{Z}[\mathfrak{s}^{-1}] \ar[r]^{s-1} & \mathbb{Z}[\mathfrak{s}^{-1}].
    }
\]
By inverting the degree, this chain complex $(E_1^{N-p,2q},d_1^{N-p,2q})$ is almost the same thing as the one considered in Lemma \ref{lem:Kunneth.torus.torsion}, the only difference being that $\mathbb{Z}$ is replaced with $\mathbb{Z}[\mathfrak{s}^{-1}]$. Therefore, by a similar computation as Lemma \ref{lem:Kunneth.torus.torsion}, the cohomology group $E_2^{p,2q}$ is isomorphic to $(\mathbb{Z}/g \mathbb{Z})^{\binom{N-1}{N-p}} = (\mathbb{Z}/g \mathbb{Z})^{\binom{N-1}{p-1}}$.
This shows the first half of the claim. 

Again by Remark \ref{rmk:d1.mapping.torus}, the AH spectral sequence $\bar{E}_1^{p,2q}$ associated to $\mathcal{B}_l$ satisfies $\bar{E}_{2}^{p,2q+1} \cong 0$, and the cochain complex $(\bar{E}_1^{p,2q},d_1^{p,2q})$ a similar decomposition as
\begin{align*} 
    \bar{E}_1^{*,2q} \cong \bigotimes_{s \in \Sigma} (0 \to s\mathbb{Z} \xrightarrow{s^{-1} -1} \mathbb{Z} \to 0) \cong \bigotimes_{s \in \Sigma} (0 \to \mathbb{Z} \xrightarrow{s-1} \mathbb{Z} \to 0),
\end{align*}
where the right isomorphism is given by a restriction of the above commutative diagram. This is exactly the same as the complex considered in Lemma \ref{lem:Kunneth.torus.torsion}, and hence $\bar{E}_2^{p0}$ is isomorphic to $(\mathbb{Z}/g \mathbb{Z})^{\binom{N-1}{N-p}} = (\mathbb{Z}/g \mathbb{Z})^{\binom{N-1}{p-1}}$ as well. 

The above  computation also shows that the map $\bar{E}_2^{p,2q} \to E_2^{p,2q}$ induced by the inclusion $\iota \colon \mathcal{B}_l \to \mathcal{B}$ is an isomorphism. 
By comparing the differentials $d_{r}^{pq}$ inductively, it implies that the induced maps $\bar{E}_2^{pq} \to E_2^{pq}$ are isomorphisms on every page $E_r$ for $r\ge 2$.
\end{proof}

In order to determine the higher differentials of the AH spectral sequence for $\mathcal{B}_1$, we first study the continuous field $\mathcal{B}_1 \otimes \mathcal{O}_{g+1} \otimes \mathbb{K}$. 
\begin{remark}\label{lem:mapping.torus.composition}
    The continuous field $\mathcal{B}_1 \otimes \mathcal{O}_{g+1} \otimes \mathbb{K}$ is isomorphic to a locally trivial bundle of C*-algebras.  
    In general, by \cite[Theorem 1.1]{Dadarlat2}, the mapping torus $\mathcal{MT}(A,\alpha_1,\cdots,\alpha_N)$ associated to a mutually commuting $N$-tuple of partial $\ast$-automorphisms on a stable Kirchberg algebra is isomorphic to a locally trivial bundle of C*-algebras if the inclusions $D_{\alpha_{i_1}\circ \cdots \circ \alpha_{i_k}} \subset A$ and $R_{\alpha_{i_1}\circ \cdots \circ \alpha_{i_k}}\subset A$ induces isomorphism in K-theory.
    For $A=B_1$, since $\beta_{s_i}$ induces the multiplication by $s_i$ in $\K$-theory, the induced map $(\beta_{s_i} \otimes \id_{\mathcal{O}_{g+1}})_*$ is the identity on $\K_0(\mathcal{O}_{g+1} \otimes \mathbb{K}) \cong \mathbb{Z}/g\mathbb{Z}$.
\end{remark}

\begin{lemma}\label{lem:unital.B}
    The continuous section algebra $C(\mathbb{T}^N,\mathcal{B}_1 \otimes \mathcal{O}_{g+1} \otimes \mathbb{K})$ has a projection $e$ whose evaluation at any point $t \in \mathbb{T}^N$ represents the generator of $\K_0(\mathcal{O}_{g+1}) \cong \mathbb{Z}/g\mathbb{Z}$.
\end{lemma}
\begin{proof}
    Let us decompose $B_1=\mathbb{K}(\ell^2\Delta_{\mathfrak{s}})$ into the tensor product $\bigotimes_{s \in \Sigma} \mathbb{K}(\ell^2\Delta_{s})$ by the Chinese remainder theorem. 
    Since $\mathbb{Z}/\mathfrak{s}\mathbb{Z} \cong \prod_{s \in \Sigma}(\mathbb{Z}/s\mathbb{Z})$ is a ring isomorphism, the multiplication by $s$ acts on each tensor component individually. 
    For $s, u \in \Sigma$, we write $\delta_{s}^{u}$ for the partial $\ast$-automorphism induced from multiplication by $s$ on $\Delta_{u}$. If $s \neq u$, then $\delta_{s}^u$ is a $\ast$-automorphism. If $s = u$, the domain and range are $D_{\delta_u^u} = \mathbb{C}\cdot 1$ and $R_{\delta_u^u} = \mathbb{C}\cdot e_u$, where $e_u$ denotes the projection onto $\ell^2(u\Delta_{u}) \cong \Cz$. Set
    \[
    \mathcal{D}_{u}:=\mathcal{MT}(\mathbb{K}(\ell^2\Delta_{u}), \{ \delta_{s}^{u}\}_{s \in \Sigma}).
    \]
Then, $\mathcal{D}_{u}|_{\mathbb{T}^{\Sigma \setminus \{u\}}}$ is the matrix bundle associated to the $\mathbb{Z}^{\Sigma\setminus \{u\}}$-action $\{ \delta_{s}^{u}\}_{s \in \Sigma \setminus \{ u \} }$ by automorphisms. 
Notice that the tuple $\{\delta_{s}^{u}\}_{s \in \Sigma \setminus \{u\} }$ is implemented by a mutually commuting tuple of unitaries $v_{s}^u$ permuting the standard basis. These unitaries satisfy $e_uv_{s}^u e_u=e_u$. Therefore, we get a continuous section
\begin{align}
    e_u' \coloneqq e_u \otimes 1_{C([0,1]^{\Sigma \setminus \{ u \} })} \in C(\mathbb{T}^{\Sigma \setminus \{ u \}}, \mathcal{MT}(\mathbb{K}(\ell^2\Delta_u) , \{ \delta_s^u\}_{s \in \Sigma }) )  \subset C([0,1]^{\Sigma} ,\mathbb{K}(\ell^2\Delta_u))
    \label{eqn:section.D.e}
\end{align}
of $\mathcal{D}_u$. 

By the simultaneous diagonalization, we obtain a homotopy $\tilde{\delta}_s^u$ of $\ast$-automorphisms on $\mathbb{K}(\ell^2\Delta_u)$ connecting $\delta_s^u$ and the identity, i.e., $\tilde{\delta}_s^u \colon \mathbb{K}(\ell^2\Delta_u) \to C([0,1]) \otimes \mathbb{K}(\ell^2\Delta_u)$ such that $\ev_0 \circ \tilde{\delta}_s^u = \delta_s^u$ and $\ev_1 \circ \tilde{\delta}_s^u =\id$, such that $\ev_{t_1}\circ \tilde{\delta}_{s_1}^u$ and $\ev_{t_2} \circ \delta_{s_2}^u$ commutes for any $t_1,t_2 \in [0,1]$ and $s_1,s_2 \in \Sigma$. Moreover, we may choose such $\tilde{\delta}_s^u$ so that $\tilde{\delta}_s^u(e_u) = e_u \otimes 1$. Now, the bundle map 
\[
    \Psi \colon \mathcal{D}_u |_{\mathbb{T}^{\Sigma \setminus \{ u \} }}  \to \mathbb{T}^{\Sigma \setminus \{u\}} \times  \mathbb{K}(\ell^2\Delta_u)
\]
given by
\[
    \Psi( \{ t_s \}_{s \in \Sigma \setminus \{u\}} ,d) =\big(  \{ t_s \}_{s \in \Sigma \setminus \{u\}} , (\ev_{t_{s_N}} \circ\tilde{\delta}_{s_N}^u) \circ \cdots \circ (\ev_{t_{s_1}} \circ\tilde{\delta}_{s_1}^u)(d)\big)
\]
is well-defined, gives a trivialization of $\mathcal{D}_u|_{\mathbb{T}^{\Sigma \setminus \{u \} }}$, and sends $e_u' $ in \eqref{eqn:section.D.e} to the constant section $e_u \otimes 1 \in \mathbb{K}(\ell^2\Delta_u) \otimes C(\mathbb{T}^{\Sigma \setminus \{ u\}})$. 

Now, since $e_u'$ is identified with a constant function in the above trivialization of $\mathcal{D}_u$, we have
\[ u[e_u \otimes 1_{C([0,1])^{\Sigma\setminus \{u\} }}] =[1_{C(\mathbb{T}^{\Sigma \setminus \{ u \} },\mathcal{D}_u)}] \in \K_0(C(\mathbb{T}^{\Sigma \setminus \{u\}},\mathcal{D}_{u})) \cong \K_0(C(\mathbb{T}^{\Sigma \setminus \{u\}}) \otimes \mathbb{K}(\ell^2\Delta_{u})).\]
Hence there is a continuous path of projections $\tilde{e}_u(t) \in C(\mathbb{T}^{\Sigma \setminus \{u\} },\mathcal{D}_{u} \otimes \mathcal{O}_{u} \otimes \mathbb{K})$ connecting $\tilde{e}_u(0)=e_u \otimes 1_{\mathcal{O}_u} \otimes p$ and $\tilde{e}_u(1)=1_{\mathcal{O}_u} \otimes p$, where $p$ denotes a fixed rank $1$ projection in $\mathbb{K}$. 
Since $\delta_u^u(1_{\mathbb{K}(\ell^2\Delta_u)})=e_u$, the element $\tilde{e}_u$ is indeed contained in the subalgebra 
\begin{align*}
    C(\mathbb{T}^{\Sigma}, \mathcal{D}_u \otimes \mathcal{O}_u  \otimes \mathbb{K}) \cong {}&{} \{ b \in  C([0,1], C(\mathbb{T}^{\Sigma \setminus \{u\}} , (\mathcal{D}_u \otimes \mathcal{O}_u  \otimes \mathbb{K})|_{\mathbb{T}^{\Sigma \setminus \{u\}}})) : \delta_u^u(b(1)) = b(0) \}. 
\end{align*}

Since $(B_1, \beta) \cong \bigotimes_{s \in \Sigma} (\mathbb{K}(\ell^2\Delta_s) , \delta^u)$, the unital inclusions $\mathbb{K}(\ell^2\Delta_s) \to B_1$ induce $\ast$-homomorphisms $\phi_s \colon C(\mathbb{T}^N,\mathcal{D}_s) \to C(\mathbb{T}^N,\mathcal{B}_1)$ whose images are mutually commutative, and hence
\[
    \Phi \coloneqq \bigotimes_{s \in \Sigma} (\phi_s \otimes \id_{\mathcal{O}_{s}  \otimes \mathbb{K}} )\colon \bigotimes_{s \in \Sigma}C(\mathbb{T}^N,\mathcal{D}_{s}\otimes \mathcal{O}_{s} \otimes \mathbb{K}) \to C(\mathbb{T}^N,\mathcal{B}_1\otimes (\mathcal{O}_{s_1} \otimes \cdots \otimes \mathcal{O}_{s_N}) \otimes  \mathbb{K}). 
\]
This defines a projection
\[
    \tilde{e} \coloneqq \Phi(\tilde{e}_{s_1} \otimes \cdots \otimes \tilde{e}_{s_N}) = \varphi_{s_1}(\tilde{e}_{s_1})\cdots  \varphi_{s_N}(\tilde{e}_{s_N})
    \in 
    C(\mathbb{T}^N, \mathcal{B}_1 \otimes (\mathcal{O}_{s_1} \otimes \cdots \otimes \mathcal{O}_{s_N}) \otimes \mathbb{K}). 
\]
For any $t \in \mathbb{T}^N$, we have 
\[ 
    [\tilde{e}(t)]=[\tilde{e}(1,\cdots,1)]=[1_{\mathbb{K}(\ell^2\Delta_s) \otimes \mathcal{O}_{s_1} \otimes \cdots \otimes \mathcal{O}_{s_N}} \otimes p] \in \K_0(\mathbb{K}(\ell^2\Delta_{\mathfrak{s}}) \otimes \mathcal{O}_{s_1} \otimes \cdots \otimes \mathcal{O}_{s_N}  \otimes \mathbb{K}).
\]
Note that the unit $1$ and the rank one projection in $\mathbb{K}(\ell^2\Delta_{\mathfrak{s}})$ represent the same $g$-torsion element in this K-group.  Now, by choosing a unital embedding $\iota \colon \mathcal{O}_{s_1} \otimes \cdots \otimes \mathcal{O}_{s_N} \to \mathcal{O}_{g+1}$ by means of the Kirchberg--Phillips theorem, we get the desired projection $e:=\iota(\tilde{e})$ of $C(\mathbb{T}^N,\mathcal{B}_1 \otimes \mathcal{O}_{g+1})$. 
\end{proof}

\begin{lemma}\label{lem:bundle.Cuntz.AHSS}
    The Atiyah--Hirzebruch spectral sequence for the $\K_*$-group of $C(\mathbb{T}^N,\mathcal{B} \otimes \mathcal{O}_{g+1} \otimes \mathbb{K})$ collapses at the $E_2$-page to
    \[ 
    E_\infty^{pq} \cong E_2^{pq}  \cong \begin{cases} \lwedge ^p (\mathbb{Z}/g\mathbb{Z})^N \cong (\mathbb{Z}/g\mathbb{Z})^{\binom{N}{p}} & \text{if $q$ is even, }\\  0 & \text{if $q$ is odd.} \end{cases}
    \]
    Moreover, the inclusion $\mathcal{B}  \to \mathcal{B} \otimes \mathcal{O}_{g+1} \otimes \mathbb{K}$, given by $b \mapsto b \otimes 1_{\mathcal{O}_{g+1}} \otimes p$ ($p$ is a rank $1$ projection), induces an injective map $(\mathbb{Z}/g\mathbb{Z})^{\binom{N-1}{p-1}} \to (\mathbb{Z}/g\mathbb{Z})^{\binom{N}{p}}$ of $E_2$-pages.
\end{lemma}
\begin{proof}
By the K\"unneth theorem, the the cochain complex $(E_1^{p0},d_1^{p0})$ for the bundle $\mathcal{B} \otimes \mathcal{O}_{g+1} \otimes \mathbb{K}$ is obtained from \eqref{eqn:E1.B1} by tensoring with $\mathbb{Z}/g\mathbb{Z}$. Since $s_i -1$ acts on $\mathbb{Z}/g\mathbb{Z}$ by the identity, we have $d_1^{pq}=0$. From this, and the K\"{u}nneth theorem of cochain complexes, we see that the $E_2$-page takes the form described in the statement, and moreover that the inclusion $\mathcal{B} \to \mathcal{B} \otimes \mathcal{O}_{g+1} \otimes \mathbb{K}$ induces an injection on the $E_2$-page.

The remaining task is to show the convergence at the $E_2$-page. 
By Lemma \ref{lem:unital.B}, $\iota(f) \coloneqq e\cdot f$ for $f \in C(\mathbb{T}^N)$ defines a  $\ast$-homomorphism $\iota \colon C(\mathbb{T}^N) \to C(\mathbb{T}^N, \mathcal{B} \otimes \mathcal{O}_{g+1} \otimes \mathbb{K})$ that respects the $C(\mathbb{T}^N)$-C*-algebra structure. 
This induces the map of AH spectral sequences. 
Meanwhile, the AH spectral sequence for $\K^0(\mathbb{T}^N)$, say $\tilde{E}_2^{pq}$, has the torsion-free $E_1$-page $\tilde{E}_1^{p,2q} \cong \mathbb{Z}^{\binom{N}{p}}$ and $\tilde{E}_1^{p,2q+1} \cong 0$. 
By comparing their ranks with those of $\K^0(\mathbb{T}^N) \cong  \K^1(\mathbb{T}^N) \cong  \mathbb{Z}^{2^{N-1}}$, it follows that $\tilde{E}_r^{pq}$ collapses at $E_1$-page. 
Moreover, since $[e(t)]$ generates $\K_0(B \otimes \mathcal{O}_{g+1} \otimes \mathbb{K}) \cong \mathbb{Z}/g\mathbb{Z}$ for any $t \in \mathbb{T}^N$, one has that
\[
    \iota_* \colon \widetilde{E}_1^{pq} \cong \K_{-p}(C_0(\mathbb{T}_p^N \setminus \mathbb{T}_{p-1}^N)) \to \K_{-p}(C_0(\mathbb{T}_p^N \setminus \mathbb{T}_{p-1}^N , \mathcal{B} \otimes \mathcal{O}_{g+1} \otimes \mathbb{K})) \cong E_1^{pq}
\]
is surjective. 

They show the vanishing of the higher differentials by induction on $r$. Indeed, by the vanishing of $d_l^{pq}$ for $l \leq r-1$, we obtain the commutative diagram
\[
    \xymatrix{
    \tilde{E}_{r}^{p,q} \ar@{}[r]|\cong &\mathbb{Z}^{\binom{N}{p}} \ar@{->>}[r] \ar[d]^{\tilde{d}_r^{p,q} =0} & \mathbb{Z}/g\mathbb{Z}^{\binom{N}{p}} \ar@{}[r]|\cong \ar[d]^{d_r^{p,q}} & E_r^{p,q} \\
    \tilde{E}_{r}^{p+r,q-r+1} \ar@{}[r]|\cong &\mathbb{Z}^{\binom{N}{p}} \ar@{->>}[r]  & \mathbb{Z}/g\mathbb{Z}^{\binom{N}{p}} \ar@{}[r]|\cong & E_r^{p+r,q-r+1}, 
    }
\]
which shows the vanishing of $d_r^{pq}$. 
\end{proof}
\begin{remark}
    There is another way to understand Lemma \ref{lem:bundle.Cuntz.AHSS}. Indeed, the group $E_r^{00} \subset E_2^{00} \cong  \K_0(B \otimes \mathcal{O}_{g+1} \otimes \mathbb{K})$ is by definition the subgroup consisting of elements that has an extension to the $r$-skeleton of $\mathbb{T}$. The differential $d_r^{00}([p])$ is given by the boundary map $\partial [p]$, which vanishes if and only if $[p]$ has an extension to the $(r+1)$-skeleton. Therefore, a $(B \otimes \mathcal{O}_{g+1} \otimes \mathbb{K})$-bundle has a global section by projection whose restriction to the point is a generator of $\K_0(\mathcal{O}_{g+1})$ if and only if $d_r^{00}=0$ for any $r \geq 2$, i.e., $E_2^{00}=E_\infty^{00}$. Indeed, one can see that $d_r^{00}=0$ implies $d_r^{pq} =0$ for any $p,q$, and hence the AH spectral sequence collapses at the $E_2$-page, by using the structure of $\K^*(\mathbb{T}^N)$-modules on $\K_*(C(\mathbb{T}^N,\mathcal{B} \otimes \mathcal{O}_{g+1} \otimes \mathbb{K}))$.
\end{remark}

\begin{lemma}\label{prp:Kunneth}
    We have $\K_*(C(\mathbb{T}^N,\mathcal{B} \otimes \mathcal{O}_{g+1} \otimes \mathbb{K}))\cong \K_*(\mathsf{MT}(B,\beta) \otimes \mathcal{O}_{g+1}) \cong \bigoplus_{k \in \mathbb{Z}}\lwedge ^{*+2k} (\mathbb{Z}/g\mathbb{Z})^N$. 
\end{lemma}
\begin{proof}
    By Lemma \ref{lem:bundle.Cuntz.AHSS}, the remaining task is to show that $\K_*(\mathsf{MT}(B,\beta) \otimes \mathcal{O}_{g+1})$ is isomorphic to the direct sum of $E_\infty$-terms, that is, there is no nontrivial extension. 
    To see this, notice that an extension $M$ of $(\mathbb{Z}/g\mathbb{Z})^m$ by $(\mathbb{Z}/g\mathbb{Z})^l$ is not trivial if and only if there is $x \in M$ such that $g x \neq 0$. 
    Besides that, the K\"{u}nneth theorem shows that  $\K_*(C(\mathbb{T}^N,\mathcal{B}) \otimes \mathcal{O}_{g+1})$ is isomorphic to a direct sum of abelian groups of the form $M_1 \otimes (\mathbb{Z}/g\mathbb{Z})$ and $\Tor^{\mathbb{Z}}_1 (M_2 , \mathbb{Z}/g\mathbb{Z}) $, where $M_1$ and $M_2$ are finitely generated. Thus, any element of this group is $g$-torsion. This finishes the proof. 
\end{proof}

\begin{lemma}\label{lem:abelian.group.tor}
    Let $\{0\} =G_{-1} \leq G_0 \leq \cdots \leq G_n=G$ be an increasing sequence of abelian groups such that, for any $i=0,\cdots,n$, we have $G_i/G_{i-1} \cong (\mathbb{Z}/g\mathbb{Z})^{m_i}$ for some $m_i \in \mathbb{Z}_{\geq 1}$. Then, $|G \otimes_{\mathbb{Z}} (\mathbb{Z}/g\mathbb{Z}) \oplus \operatorname{Tor}^{\mathbb{Z}}_1(G,\mathbb{Z}/g\mathbb{Z})| =|G|^2$ if and only if $G \cong (\mathbb{Z}/g\mathbb{Z})^{m_0+\cdots+m_n}$.   
\end{lemma}
\begin{proof}
    First, we have $|G \otimes_{\mathbb{Z}} (\mathbb{Z}/g\mathbb{Z})| = |G/\Im (g \colon G \to G) | \leq |G|$ and $|\operatorname{Tor}^{\mathbb{Z}}_1(G,\mathbb{Z}/g\mathbb{Z})| = |\operatorname{Ker} (g \colon G \to G) | \leq |G|$. This means $|G \otimes_{\mathbb{Z}} (\mathbb{Z}/g\mathbb{Z}) \oplus \operatorname{Tor}^{\mathbb{Z}}_1(G,\mathbb{Z}/g\mathbb{Z})| \leq |G|^2$ and the equality holds if and only if every element of $G$ is $g$-torsion. 
    This proves the ``if part''. To see the ``only if part'', assume that every element of $G$ is $g$-torsion. Then, the exact sequence $0 \to G_{i-1} \to G_{i} \to G_{i}/G_{i-1}  \to 0$ splits. Indeed, choosing lifts of a basis of $G_i/G_{i-1}\cong (\mathbb{Z}/g\mathbb{Z})^{m_i}$ defines a group homomorphism $G_i/G_{i-1} \to G_i$ splitting the quotient map. This inductively shows that $G$ is a direct sum of copies of $\mathbb{Z}/g\mathbb{Z}$.
\end{proof}

\begin{proof}[Proof of Theorem~\ref{thm:BOS}]
    We first consider the case that $|\Sigma| <\infty$. We first show that 
    \[ 
    \K_*(C^*_r(\mathcal{G}_{\mathrm{tor}})) = \K_{*}(\widetilde{B} \rtimes \mathbb{Z}^N) \cong \K_{*-N}(C(\mathbb{T}^N , \cB)),
    \]
    where the isomorphism given in Lemma \ref{lem:Dirac.dualDirac}, is isomorphic to 
    \[
    \K_*\Big(\bigotimes_{s \in \Sigma }\mathcal{O}_{s}\Big) \cong \bigoplus_{p \in * + 2\mathbb{Z}}(\mathbb{Z}/g\mathbb{Z})^{\binom{|\Sigma|-1}{p-1} }.
    \]
    Note that, since we have assumed $N=|\Sigma| \geq 2$, the number of copies of $\mathbb{Z}/g\mathbb{Z}$ of the $\K_0$ and $\K_1$ groups are the same, thus there is no need to care for the degree shift by $-N$.
    We have already observed in Lemma \ref{lem:AHSS.torsion.K} that the $E_2$-page of the AH spectral sequence for $\K_*(C(\mathbb{T}^N,\mathcal{B}_1))$ is isomorphic to $(\mathbb{Z}/g\mathbb{Z})^{\binom{|\Sigma|-1}{p-1}}$ for even $q$ and $0$ for odd $q$. Hence it suffices to show that this spectral sequence collapses at the $E_2$-page and there is no extension problem. They follow from Lemmas \ref{lem:bundle.Cuntz.AHSS} and \ref{prp:Kunneth} in the following way.

    As is shown in Lemma \ref{lem:bundle.Cuntz.AHSS}, the inclusion $C(\mathbb{T}^N,\mathcal{B}_1) \to C(\mathbb{T}^N,\mathcal{B}_1 \otimes \mathcal{O}_{g+1} \otimes \mathbb{K})$ induces injections of $E_2$-pages, and the AH spectral sequence for $\mathcal{B}_1 \otimes \mathcal{O}_{g+1} \otimes \mathbb{K}$ collapses at the $E_2$-page. Thus, a diagram chase shows that the higher derivatives $d_r^{pq}$ for the AH spectral sequence for $\mathcal{B}_1$ also vanish. 
    Suppose that $\K_*(C(\mathbb{T}^N,\mathcal{B}_1))$ is a nontrivial extension of $E_\infty$-pages. Then, by Lemma \ref{lem:abelian.group.tor} and the K\"unneth theorem in C*-algebra K-theory, the number of elements of $\K_*(C(\mathbb{T}^N,\mathcal{B}_1) \otimes \mathcal{O}_{g+1})$ is strictly smaller than the square of the number of elements of $\K_*(C(\mathbb{T}^N,\mathcal{B}_1))$. This contradicts Lemma \ref{prp:Kunneth}.

By the Kirchberg--Phillips theorem, the remaining task is to show that $[1] \in \K_0(C^*_r(\cG_{\tor}))$ satisfies that $n[1]=0$ if and only if $g \mid n$. Under the isomorphism of K-theory in Lemma \ref{lem:Dirac.dualDirac}, $[1_B]$ corresponds to the image of the Bott generator $\beta \in \K_N(C_0(\Rz^N) )$ via the inclusion
\[
    C_0(\Rz^N) \cong C_0(\mathbb{R}^N) \otimes \Cz  1_B  \hookrightarrow C_0(\mathbb{R}^N) \otimes \widetilde{B} \hookrightarrow (C_0(\mathbb{R}^N) \otimes \widetilde{B}) \rtimes \mathbb{Z}^N.
\]
Up to homotopy and Morita equivalence, this $\ast$-homomorphism is identified with the inclusion 
\[
    C_0((0,1)^N) \otimes \Cz 1_B \hookrightarrow C(\mathbb{T}^N,\cB).
\]
Hence the Bott generator $\beta \in \K_N(C_0(\Rz^N) \otimes \Cz 1_B)$ is sent to the generator of 
\begin{equation*}
 E_2^{N0} = \mathop{\mathrm{Coker}}(\bigoplus \beta_s -\iota \colon \bigoplus_s \K_0(D_{\beta_s}) \to \K_0(B)) \cong \Zz/g\Zz,   
\end{equation*}
which is a subgroup of $\K_N(C(\mathbb{T}^N,\mathcal{B}))$.

    Finally, we consider the case that $|\Sigma|=\infty$. Even in this case, we have the same K-theory computation $\K_*(C^*_r(\mathcal{G}_{\mathrm{tor}})) \cong  (\mathbb{Z}/g\mathbb{Z})^{\infty}$.
    To see this, let us choose an increasing sequence of finite subsets $\Sigma_1 \subset \Sigma_2 \subset \cdots \subset \Sigma$ such that $\bigcup\Sigma_n =\Sigma$, $N_n\coloneqq |\Sigma_n|$ is even and $\gcd (\{ s-1 : s \in \Sigma_1\}) = g$. Then we have
    \[ 
    \widetilde{B}_{\Sigma} =\varinjlim_{s^a \in \mathbb{N}^{\Sigma}} \bigotimes_{s \in \Sigma} \mathbb{M}_{s^\infty} \cong \varinjlim_{n \to \infty} \bigg( \varinjlim_{s^a \in \mathbb{N}^{\Sigma_n}} \bigotimes_{s \in \Sigma} \mathbb{M}_{s^\infty} \bigg) \cong \varinjlim_{n \to \infty} \bigg( \varinjlim_{s^a \in \mathbb{N}^{\Sigma_n}} \bigotimes_{s \in \Sigma_n} \mathbb{M}_{s^\infty} \bigg) = \varinjlim_{n \to \infty} \widetilde{B}_{\Sigma_n}
    \]
    and hence 
    \[ 
    C_r^*(\mathcal{G}_{\mathrm{tor}}) \sim_{\mathrm{Morita}} \widetilde{B}_{\Sigma} \rtimes \mathbb{Z}^\Sigma 
    = \varinjlim_n \widetilde{B}_{\Sigma_n} \rtimes \mathbb{Z}^{\Sigma_n}  \sim_{\mathrm{Morita}} \varinjlim_n C^*_r(\mathcal{G}_{n,\mathrm{tor}}), 
    \]
    where $\mathcal{G}_{n,\mathrm{tor}}$ denotes the torsion subgroupoid associated to $\Sigma_n$. Through the K-theory isomorphism from Lemma~\ref{lem:Dirac.dualDirac}, the inclusion $\widetilde{B}_{\Sigma_n} \rtimes \Zz^{N_n} \to \widetilde{B}_{\Sigma_{n+1}} \rtimes \Zz^{N_{n+1}}$ corresponds to the inclusion
    \begin{align*}
    \mathsf{MT}(\widetilde{B}_{\Sigma_n},\{ \beta_{s} \}_{s \in \Sigma_n}) \otimes C_0(\mathbb{R}^{N_{n+1}-N_n}) 
     \to
    \big(
    \mathsf{MT}(\widetilde{B}_{\Sigma_{n+1}}, \{\beta_{s}\}_{s \in \Sigma_{n}}) \otimes C_0(\mathbb{R}^{N_{n+1}-N_n}) \big) \rtimes \Zz^{N_{n+1} - N_n},
    \end{align*}
    where the left hand side is KK-equivalent to $\mathsf{MT}(\widetilde{B}_{\Sigma_n},\beta_{n})$ and the right hand side is Morita equivalent to $ \mathsf{MT}(\widetilde{B}_{\Sigma_{n+1}},\beta_{n+1})$.  
    This inclusion is homotopic to 
    \begin{align*}
    {}&{}C(\mathbb{T}^{N_{n}}, \mathcal{MT}(\widetilde{B}_{\Sigma_n},\{ \beta_s\}_{s \in \Sigma_n})) \otimes C_0((0,1)^{N_{n+1}-N_n}) \\
    \to {}&{} 
    C(\mathbb{T}^{N_{n}}, \mathcal{MT}(\widetilde{B}_{\Sigma_{n+1}},\{\beta\}_{s \in \Sigma_n})) \otimes C_0((0,1)^{N_{n+1}-N_n}))
    \cong C_0(\mathbb{T}^{N_{n+1}} \setminus \mathbb{T}^{N_n}, \mathcal{MT}(\widetilde{B}_{\Sigma_{n+1}},\{\beta\}_{s \in \Sigma_{n+1}}))\\
    \subset {}&{}
    C(\mathbb{T}^{N_{n+1}} , \mathcal{MT}(\widetilde{B}_{\Sigma_{n+1}},\beta_{n+1})) .
    \end{align*}
    By Lemmas \ref{lem:Kunneth.torus.torsion} and (the proof of)  \ref{lem:AHSS.torsion.K}, the induced map in $E_2$-pages is 
    \[
    \H_{N_n-*}\Big( \bigotimes_{s \in \Sigma_n} P(s) \Big) \to \H_{N_{n+1}-*}\Big( \bigotimes_{s \in \Sigma_{n+1}} P(s) \Big), 
    \]
    induced from the tensor product of the inclusion of chain complexes $(0 \to 0 \to \mathbb{Z} \to 0) \to (0 \to \mathbb{Z} \xrightarrow{s-1} \mathbb{Z} \to 0)$. By the K\"unneth theorem, 
    this inclusion induces the injective maps to a direct summand in the $E_2$-terms.
    Thus, the above proof of Theorem \ref{thm:BOS} for the case where $\Sigma$ is finite shows that 
    \[ 
    \K_*(C^*_r(\mathcal{G}_{\mathrm{tor}})) \cong \varinjlim_n  \K_*(C^*_r(\mathcal{G}_{n,\mathrm{tor}})) \cong \varinjlim_{n \to \infty} (\mathbb{Z}/g\mathbb{Z})^{\binom{|\Sigma_n|-1}{p-1}} \cong  
    (\mathbb{Z}/g\mathbb{Z})^{\infty}. 
    \] 
    For any $n$,  the unit of $C^*_r(\cG_{n,\mathrm{tor}})$ satisfies $m[1] =0 \in \K_0(C^*_r(\cG_{n,\mathrm{tor}}))$ if and only if $g \mid m$. Thus so does $[1] \in \K_0(C^*_r(\cG_{\tor}))$. This finishes the proof. 
\end{proof}

\appendix

\section{Raven's bivariant Chern character} \label{section:Raven}
In this Appendix, we review Raven’s Chern character \cite{Raven,DW}, a natural transformation from topological bivariant equivariant KK-theory to the Baum--Schneider equivariant bivariant homology theory, and collect a number of basic facts that is used in the paper.

Let $G$ be a countable discrete group. For a pair of proper cocompact $G$-spaces $(Z,W)$, i.e.\ a proper cocompact $G$-space $Z$ and its closed $G$-subspace $W$,  let $\mathrm{K}_G^0(Z,W)$ denote the equivariant K-group of the pair $(Z,W)$, that is, the Gr\"{o}thendieck group of the triples $(E_0,E_1,u)$, where $E_0,E_1$ are $G$-vector bundles on $Z$ and $u \colon E_0|_W \to E_1|_W$ is a $G$-isomorphism.
Let $(X,A)$ be a pair of proper cocompact $G$-CW-complexes and let $Y$ be a compact Hausdorff $G$-space. 
The topological KK-group $\KK^G((X,A),Y)$ is defined by the set of equivalence classes of equivariant KK-cycles (\cite[Section 4]{Raven}). 
Here, an even (resp.\ odd) topological $\mathrm{KK}^G$-cycle is a triple $(M,f,\xi)$, where $M$ is an even (resp.\ odd) dimensional proper cocompact $G$-$\mathrm{Spin}^c$-manifold with possibly non-empty boundary, $f \colon (M,\partial M) \to (X,A)$ is a proper $G$-map, and $\xi \in \mathrm{K}_G^0(M \times Y)$. 
In \cite[Corollary 4.7.9]{Raven}, this group is shown to be isomorphic to Kasparov's equivariant KK-group $\mathrm{KK}^G(C_0(X \setminus A),C(Y))$. 
When $A = \emptyset$, we abbreviate it and write $\mathrm{KK}^G(X,Y)$.

On the other hand, for such $(X,A)$ and $Y$, the Baum--Schneider bivariant homology $\mathrm{HH}^*_G((X,A),Y)$ is defined in terms of sheaf theory as  
\begin{align}
    \mathrm{HH}^n_G((X,A),Y)\coloneqq \Hom_{\mathbf{D}^+(\mathbb{C}[G])} (R\Gamma_c \mathbb{C}_{X \setminus A}, R\Gamma_c\mathbb{C}_Y[n]).\label{eqn:Baum.Schneider}
\end{align}
Here, $\mathbf{D}_G^+(X;\mathbb{C})$ and $\mathbf{D}^+(\mathbb{C}[G])$ denote the derived categories of nonnegative chain complexes of $G$-equivariant sheaves on $X$ and those of $\mathbb{C}[G]$-modules respectively, $\mathbb{C}_X$ denotes the locally constant sheaf on $X$, and $R\Gamma_c \colon \mathbf{D}_G^+(X;\mathbb{C}) \to \mathbf{D}^+(\mathbb{C}[G])$ denotes the derived functor of the compactly supported global section functor.
We remark that in \cite{Raven} the coefficient ring is arbitrary, but in this paper we only use $\mathbb{C}$ and abbreviate it. 
\begin{remark}\label{rmk:Raven.opposite}
    We follow Raven’s notation, in which the symbol $\circ$ for morphisms in the derived categories is used in the opposite direction from the usual convention for composition of maps.
\end{remark}
As is proved in \cite[Theorem 4.6.9, Subsection 6.5]{Raven}, both $\mathrm{KK}^G((X,A),Y)$ and $\mathrm{HH}^n_G((X,A),Y)$ are equivariant homology functors with respect to the first variable.

\begin{remark}\label{rmk:homology.functor.boundary}
When $Y$ is not compact, the equivariant KK-groups and HH-groups are defined by
\begin{align*} 
    \mathrm{KK}^G_*((X,A),Y) \coloneqq {}&{} \mathop{\mathrm{Ker}} (\mathrm{KK}^G_*((X,A),Y^+) \to \mathrm{KK}^G_*((X,A),\mathrm{pt})), \\
    \mathrm{HH}_G^*((X,A),Y) \coloneqq {}&{} \mathop{\mathrm{Ker}} (\mathrm{HH}_G^*((X,A),Y^+) \to \mathrm{HH}_G^*((X,A),\mathrm{pt})),
\end{align*}
where $Y^+$ denotes the one-point compactification and $\mathrm{pt} \in Y^+$ denotes the additional point. In particular, if $Y= \bigsqcup_{n \in \mathbb{N}} Y_n$ and each $Y_n$ is compact, then we have $\mathrm{F}^G((X,A),Y) \cong \bigoplus_n \mathrm{F}^G((X,A),Y_n)$ for the both case $\mathrm{F}^G=\mathrm{KK}^G_*$ and $\mathrm{HH}_G^*$. 
Even in this case, the topological and C*-algebraic KK-groups are isomorphic, which is verified by comparing the two short exact sequences 
\begin{gather}
    0 \to \KK^G((X,A),Y) \to \KK^G((X,A),Y^+) \to \KK^G((X,A),\mathrm{pt}) \to 0, \\
    0 \to \KK^G(C_0(X \setminus A), C_0(Y)) \to \KK^G(C_0(X \setminus A), C(Y^+)) \to \KK^G(C_0(X \setminus A), \mathbb{C}) \to 0.
\end{gather}
By the naturality of the isomorphism $\KK^G((X,A),Y) \cong \KK^G(C_0(X\setminus A),C(Y))$ with respect to the second variable, which is obvious from the definition (\cite[p. 43]{Raven}), we obtain an isomorphism of these exact sequences. 
Note that a topological KK-cycle $[M,f,\xi] \in \KK^G_*((X,A),Y^+)$ determines an element of $\KK^G_*((X,A),Y)$ if $\xi \in \K^0_G(M \times Y^+, M \times \mathrm{pt} \cup \partial M \times Y^+)$.  
As for the Baum--Schneider homology, this new definition of $\mathrm{HH}_G^*((X,A),Y)$ agrees with the original one, \eqref{eqn:Baum.Schneider}, by \cite[Proposition 6.5.6]{Raven}. 
By comparing the short exact sequences, $\widehat{\mathrm{ch}}_R$ is defined for non-compact $Y$, and is an isomorphism after tensoring with $\mathbb{C}$ (cf.\ \cite[Corollary~7.3.12]{Raven}). 
The following Remark \ref{rmk:Raven.noncompact} and Lemma \ref{lemma:induction.tKK.HH} deals with compact $Y$ in the proof, however, these results are generalized to this non-compact setting by considering $Y^+$.   
\end{remark}

    The equivariant KK- and HH- theories are both equivariant bivariant homology theories. In other words, for each fixed first (resp. second) variable, they define an equivariant cohomology (resp. homology) functor in the second (resp. first) variable, and any morphism in the first (resp. second) variable induces a morphism between the corresponding cohomology (resp. homology) functors. This fact is essentially proved in Raven, but since some of the statements are not discussed explicitly, we provide the details in the following Remarks \ref{rmk:bivariant.homologies}, \ref{rmk:bivariant.homologies.2} and \ref{rmk:bivariant.homologies.3}.

\begin{remark}\label{rmk:bivariant.homologies}
    If one has $g \colon (X_1, A_1) \to (X_2, A_2)$ and $h \colon Y_1 \to Y_2$, then the induced maps commutes, i.e., 
    \begin{align}
    \begin{split}
        \mathrm{KK}^G_*(g,Y_1) \circ \mathrm{KK}^G_*((X_1,A_1),h)[M,f,\xi] ={}&{} [M,g \circ f, h^*\xi] \\
        ={}&{} \mathrm{KK}^G((X_2,A_2),h) \circ \mathrm{KK}^G_*(g,Y_2)[M,f,\xi], \\
        \mathrm{HH}_G^*(g,Y_1) \circ \mathrm{HH}_G^*((X_1,A_1),h)(\phi) ={}&{} [g] \circ \phi \circ [h] \\
        ={}&{} \mathrm{HH}_G^*((X_2,A_2),h) \circ \mathrm{HH}_G^*(g,Y_2)(\phi).
        \end{split}\label{eqn:bivariant.natural}
    \end{align}
    Here, $[g]$ and $[h]$ denote the induced map in the derived category (\cite[Example 6.5.3 (a)]{Raven}). 
\end{remark}

\begin{remark}\label{rmk:bivariant.homologies.2}
    Second, the boundary map in the first (resp.\ second) variable commutes with morphisms in the first (resp.\ second) variable, as part of the axioms of a generalized (co)homology theory.
    For the KK-theory, the boundary map in the first variable is given by $\delta[M,f,\xi] \coloneqq [\partial M,f|_{\partial M},\xi|_{\partial M}]$, and hence  
    \begin{align*}
    \mathrm{KK}^G(g|_A,Y)\circ \delta ([M,f,\xi]) ={}&{} \mathrm{KK}^G(g|_A,Y)([\partial M,f|_{\partial M},\xi|_{\partial M}]) = [\partial M, g|_A \circ f|_{\partial M}, \xi|_{\partial M}]\\
    ={}&{}\delta ([M,g \circ f,\xi]) = \delta \circ \mathrm{KK}^G(g,Y)([M,f,\xi]).
    \end{align*}
    For the boundary map in the second variable follows from \cite[Theorem 4.6.9]{Raven} and the corresponding fact in the equivariant Kasparov theory. Indeed, by \cite[Appendix]{KasSka}, the functor $\mathrm{KK}^G_*(C_0(X \setminus A), \cdot ) $ is half-exact.
    Thus a short exact sequence $0 \to C_0(Y \setminus Z) \to C_0(Y) \to C_0(Z) \to 0$ of commutative $G$-C*-algebras associates the long (six-term) exact sequence. Moreover, if one has a morphism of short exact sequences of commutative $G$-C*-algebras, then the induced maps intertwines the boundary map.  
    
    For the HH-theory, this fact is essentially proved in \cite[Proposition 6.5.6]{Raven}. We briefly expand the argument and make the naturality statement explicit.
    For a $G$-sheaf on $X$ and a closed $G$-subspace $ A \subset X$, let $\mathscr{F}_{X \setminus A}$ be the sheaf whose section on $U \subset X$ consists of sections of $\mathscr{F}$ on $U$ that vanishes on $U \cap A$, and  $\mathscr{F}_{A}$ be the quotient sheaf $\mathscr{F}/\mathscr{F}_{X \setminus A}$. 
    Both are special cases of the general definition in \cite[Definition 6.4.3]{Raven}. Under the identifications $R\Gamma_c \mathbb{C}_{X,X \setminus A } \cong R\Gamma_c \mathbb{C}_{X \setminus A}$ and $R\Gamma_c \mathbb{C}_{X,A } \cong R\Gamma_c \mathbb{C}_{A}$ (in the proof of \cite[Proposition 6.5.6]{Raven}), the boundary map 
    \[
    \delta \colon  \Hom_{\mathbf{D}^+(\mathbb{C}[G])}(R\Gamma_c\mathbb{C}_{X \setminus A}, R\Gamma_c\mathbb{C}_Y[n])
    \to \Hom_{\mathbf{D}^+(\mathbb{C}[G])}(R\Gamma_c\mathbb{C}_{A}, R\Gamma_c\mathbb{C}_Y[n+1]) \cong \mathrm{HH}^{n+1}_G(A,Y)
    \]
    is defined by the precomposition, after a degree shift by $1$, with 
    \[
    R\Gamma_c \big( \mathbb{C}_{X,A} \xleftarrow{\simeq} M(\mathbb{C}_{X,X \setminus A} \to \mathbb{C}_X ) \to \mathbb{C}_{X,X \setminus A}[1]\big) \in \Hom_{\mathbf{D}^+(\mathbb{C}[G])}(R\Gamma_c\mathbb{C}_{X,A}, R\Gamma_c\mathbb{C}_{X,X\setminus A}[1]).  
    \]
    Here, the mapping cone sheaf $M(\mathbb{C}_{X,X \setminus A} \to \mathbb{C}_{X})$ is quasi-isomorphic to $\mathbb{C}_{X,A}$ by the exactness of $0 \to \mathbb{C}_{X,X\setminus A} \to \mathbb{C}_X \to \mathbb{C}_{X,A} \to 0$. 
    For a continuous $G$-map $f \colon (X_1,A_1) \to (X_2,A_2)$, one has a $G$-sheaf morphism $f^* \colon \mathbb{C}_{X_2} \to f_*\mathbb{C}_{X_1}$ sending a locally constant function  $s \in \Gamma(U,\mathbb{C}_{{X_2}})$ to $f^*s \in \Gamma (f^{-1}(U),\mathbb{C}_{X_1} ) \cong \Gamma (U, f_*\mathbb{C}_{X_1})$. Since it restricts to a $f^* \colon \mathbb{C}_{X_2,X_2 \setminus A_2} \to f_*\mathbb{C}_{X_1,X_1 \setminus A_1}$, we get
    the commutative diagram 
    \[
    \xymatrix{
     \mathbb{C}_{X_2,A_2} \ar[d] & \ar[l]_{\simeq \quad \quad \quad } M(\mathbb{C}_{X_2,X_2 \setminus A_2} \to \mathbb{C}_{X_2} ) \ar[r]\ar[d]  & \mathbb{C}_{X_2,X_2 \setminus A_2}[1] \ar[d] \\
    f_*\mathbb{C}_{X_1,A_1} & \ar[l]_{\simeq \quad \quad \quad } f_*M( \mathbb{C}_{X_1,X_1 \setminus A_1} \to \mathbb{C}_{X_1} ) \ar[r] & f_*\mathbb{C}_{X_1,X_1 \setminus A_1}[1] \\
    }
    \]
    in the derived category of sheaves $\mathbf{D}^+_G(X_2;\mathbb{C})$. This concludes the desired naturality. The same argument also proves the naturality of $\delta$ with respect to the second variable.
\end{remark}

\begin{remark}\label{rmk:bivariant.homologies.3}
    Recall that a family of natural transforms of $G$-equivariant homology functors $\rho \colon \mathrm{F}_{1,*} \to \mathrm{F}_{2,*}$ is a morphism of $G$-equivariant homology functors (i.e., commutes with the boundary map) if the following diagram commutes (cf.\ \cite[Proposition II.3.6]{Rudyak});
    \[
    \xymatrix{
    \mathrm{F}_{1,*+1}(\Sigma X) \ar[r]^\sigma \ar[d]^\rho  & \mathrm{F}_{1,*}( X) \ar[d]^\rho \\
    \mathrm{F}_{2,*+1}(\Sigma X) \ar[r]^\sigma   & \mathrm{F}_{2,*}( X).
    }
    \]
    Here, $\Sigma \coloneqq ([0,1],\{0,1\})$, $\Sigma X \coloneqq \Sigma \times X$, and the horizontal maps $\sigma$ are suspension isomorphisms given by
    \[
        \mathrm{F}_{i,*+1}(X \times [0,1], X \times \{0,1\}) \xrightarrow{\delta} \mathrm{F}_{i,*}(X \times \{0,1\} ) = 
        \mathrm{F}_{i,*}(X) ^{\oplus 2}\xrightarrow{\pr_1} \mathrm{F}_{i,*}(X). 
    \]
    In KK- and HH- theories, the suspension isomorphisms have explicit descriptions given by 
    \begin{align*}
        \sigma \colon{}&{} \mathrm{KK}^G_*(X,Y) \to \mathrm{KK}^G_{*+1}(\Sigma X, Y) ,{}&&{} \quad [M,f,\xi] \mapsto [[0,1] \times M,\id_{[0,1]} \times f, 1_{[0,1]} \otimes \xi], \\
        \sigma \colon{}&{}\mathrm{HH}_G^{-*}(X,Y) \to \mathrm{HH}_G^{-(*+1)}(\Sigma X, Y),{}&&{} \quad \varphi \mapsto \alpha \otimes \phi.
    \end{align*}
    Here, we write $\alpha$ the generator of $\mathrm{HH}^{-1}(\Sigma ,\mathrm{pt}) \cong  \mathbb{C}$ (\cite[Corollary 6.5.7]{Raven}) and $\otimes$ the external product in Baum--Schneider homology, which is given in the following way. For a $G$-space $Z$, we write $P_Z$ a projective resolution of $R\Gamma_c\mathbb{C}_Z$. The K\"unneth theorem of sheaf cohomology (see e.g.\ \cite[Exercise II.18]{KasSha}) states that there is an isomorphism $R\Gamma_c\mathbb{C}_{X \times (0,1)}  \cong R\Gamma_c\mathbb{C}_X \otimes^L R\Gamma_c\mathbb{C}_{(0,1)}\cong P_X \otimes P_{(0,1)}$. Thus, for $\phi \in \Hom_{\mathbf{K}^+(\mathbb{C}[G])}(P_X,P_Y[-n])$, the morphism $\phi \otimes \alpha \colon  P_{(0,1)} \otimes P_X \to P_Y[-n] \otimes \mathbb{C}[-1]$ defines an element of $\Hom_{\mathbf{D}^+(\mathbb{C}[G])} (R\Gamma_c\mathbb{C}_{X \times (0,1)}, R\Gamma_c\mathbb{C}_Y[-n-1])$.

    In particular, this shows that the boundary map in the first variable commutes with the morphism in the second variable. Indeed, for $g \colon Y_1 \to Y_2$, we have
    \begin{gather*}
        \sigma \circ \mathrm{KK}_*^G(X,g) [M,f,\xi] = [[0,1] \times M, \id_{[0,1]} \times f, 1 \otimes (\id_M \otimes g)^*\xi ] = \mathrm{KK}_*^G(\Sigma X,g)  \circ \sigma , \\
        \sigma \circ \mathrm{HH}^*_G(X,g)(\phi) = \alpha \otimes (\phi \circ [g])=(\alpha \otimes \phi) \circ [g] = \mathrm{HH}_G^*(\Sigma X,g) \circ \sigma,
    \end{gather*}
    for $[M,f,\xi] \in \KK^G_*(X,Y_2)$ and $\phi \in \Hom_{\mathbf{K}^+(\mathbb{C}[G])} (P_X,P_{Y_2})$.
\end{remark}

In \cite[Definition 7.3.1]{Raven}, Raven introduced the bivariant Chern character at $\gamma \in G_{\mathrm{tor}}$, a torsion element $\gamma $ of $G$, as
\[
    \mathop{\mathrm{ch}_R^\gamma} \colon \KK^G_n((X,A),Y) \to \bigoplus_{k \in \mathbb{Z}}\mathrm{HH}_{Z(\gamma)}^{-n + 2k}((X^{\gamma},A^\gamma),Y^{\gamma }),
\]
where $Z(\gamma) \coloneqq \{ g \in G \mid g\gamma = \gamma g\}$ denote the centralizer group of $\gamma$ and $X^\gamma \coloneqq \{ x \in X \mid \gamma \cdot x=x\}$ denote the $\gamma$-fixed point set. The case of $\gamma =e$ is called the localized Chern character, while others are called delocalized. As is proved in \cite[Corollary 7.3.4]{Raven}, $\mathop{\mathrm{ch}_R^\gamma}$ is a natural transform of the $G$-equivariant homology functors with respect to the first variable $(X,A)$. 
Let
\begin{align}
    \widehat{\mathrm{HH}}{}_G^{*} ((X,A) , Y) \coloneqq \bigoplus_{[\gamma ]   \in  \mathcal{C}(G_{\rm tor}) }  \mathrm{HH}_{Z(\gamma )}^{*} ((X^\gamma , A^\gamma)  ,  Y^\gamma  ). 
    \label{eq:range.Raven.torsion}
\end{align}
Here, $\mathcal{C}(G_{\rm tor})$ denotes the set of conjugacy classes of torsion elements in $G$. In \cite[Theorem 7.3.11]{Raven}, 
\[
    \mathop{\widehat{\mathrm{ch}}_R} \coloneqq \bigoplus_{[\gamma] \in \mathcal{C}(G_{\mathrm{tor}})} 
        \mathop{\mathrm{ch}_R^\gamma} \colon \KK^G_n((X,A),Y) \to \bigoplus_{k \in \mathbb{Z}}    \widehat{\mathrm{HH}}{}_G^{-n+2k} ((X,A) , Y)
\]
is shown to be an isomorphism after tensoring with $\mathbb{C}$. Note that, if $G$ is torsion-free, then the right hand side is $\bigoplus_k \mathrm{HH}_G^{-n+2k}((X,A),Y)$. 

By comparing the short exact sequences in Remark \ref{rmk:Raven.noncompact}, $\widehat{\mathrm{ch}}_R$ is defined for non-compact $Y$, and is an isomorphism after tensoring with $\mathbb{C}$ (cf.\ \cite[Corollary~7.3.12]{Raven}). 
The following Remark \ref{rmk:Raven.noncompact} and Lemma \ref{lemma:induction.tKK.HH} deals with compact $Y$ in the proof, however, these results are generalized to this non-compact setting by considering $Y^+$.

\begin{remark}\label{rmk:Raven.precise}
We do not enter into the details of Raven’s theory any further than necessary, but we must establish the minimal preliminaries in order to complete the formal argument. 
We recall the definition of Raven's Chern character.
Raven’s Chern character $\mathrm{ch}_{R}^\gamma([M,f,\xi])$ localized at $\gamma \in G_{\mathrm{tor}}$, defined in \cite[Definition 7.3.1]{Raven}, is explicitly given by
\begin{align}
    \begin{split}
        \mathrm{ch}_{R}^\gamma([M,f,\xi])\coloneqq {}&{}  [f] \circ [\pi_{\mathring{M}}] \circ [[ \mathrm{ch}_G^\gamma (\xi)]] \circ (\mathrm{ch}_G^\gamma(M) \times 1_{Y^\gamma}), \\
        \mathrm{ch}_G^\gamma(M) \coloneqq {}&{}[[\mathrm{Td}_G^\gamma (M)]] \circ <\!\!< \mathrm{ch}_G^\gamma (\beta_{TM^\natural \oplus n}) >\!\!> \circ \int_{(TM \oplus n)^\gamma} . 
    \end{split}
    \label{eqn:Raven.Chern}
\end{align}
We refer the reader to \cite{Raven} for the precise meaning of each notation.
We remind the meaning of each term.
For a manifold $M$, the derived global section $R\Gamma_c \mathbb{C}_M$ is given by the complex of compactly supported differential forms $\Omega_c^\bullet (M) $. 
For a totally disconnected space $Y$, the sheaf $\mathbb{C}_Y$ is $c$-soft, and hence $R\Gamma_c \mathbb{C}_Y = \Gamma_c\mathbb{C}_Y=\mathbb{C}[Y]$. 
Under these identifications, \eqref{eqn:Raven.Chern} is given by the composition 
    \begin{align}
    R\Gamma_c\mathbb{C}_{X^\gamma \setminus A^\gamma} \xrightarrow{f^* } R\Gamma_c\mathbb{C}_{\mathring{M}^\gamma} ={}&{} \Omega_c^\bullet (\mathring{M}^\gamma ) \xrightarrow{\pi_{\mathring{M}^\gamma} ^* } \Omega_c^\bullet (\mathring{M}^\gamma ) \otimes \mathbb{C}[Y^\gamma] \xrightarrow{ \mathrm{ch}_G^\gamma (\xi)} \Omega_c^{\bullet}(\mathring{M}^\gamma ) \otimes \mathbb{C}[Y^\gamma]  \\
    \xrightarrow{\mathrm{Td}_G^\gamma(M^\gamma ) \otimes \id}{}&{} \Omega_c^\bullet(\mathring{M}^\gamma ) \otimes \mathbb{C}[Y^\gamma] \xrightarrow{\mathrm{ch}_G^\gamma (\beta_{TM^\natural \oplus n}) } \Omega_c^{\bullet}((T\mathring{M} \oplus n)^\gamma ) \otimes \mathbb{C}[Y^\gamma] \xrightarrow{\int_{(T\mathring{M} \oplus n)^\gamma } \otimes \id} \mathbb{C}[Y^\gamma] ,
    \end{align}
    where $\mathrm{ch}_G^\gamma (\xi)$ and $\mathrm{Td}_G^\gamma(M^\gamma )$ are given by the exterior product (or the cup product) with the corresponding $Z(\gamma)$-invariant differential forms, whose definitions are given in \cite[Definitions 7.1.4, 7.2.9]{Raven}. The morphism $<\!\!<\! \mathrm{ch}_G^\gamma (\beta_{TM^\natural \oplus n})\!>\!\!>$ is the Thom isomorphism, given by the composition
    \[
    \Omega_c^\bullet (M^\gamma ) \xrightarrow{\pi^*} \Omega^\bullet((TM \oplus n)^\gamma ) \xrightarrow{\mathrm{ch}_G^\gamma (\beta_{TM ^\natural \oplus n})} \Omega^\bullet ((TM \oplus n)^\gamma ),   
    \]
    whose image is actually contained in $\Omega_c^\bullet ((TM \oplus n)^\gamma )$ (\cite[Example 7.1.6]{Raven}). 
\end{remark}

\begin{remark}\label{rmk:Raven.natural.cohomological}
     It is proved in \cite[Corollary 7.3.4]{Raven} that Raven's Chern character is natural with respect to the first variable. That is, if one has a continuous $G$-map $f \colon (X_1,A_1) \to (X_2,A_2)$, then  
    \[
    \mathop{\widehat{\mathrm{ch}}_R} \circ \mathrm{KK}^G_*(f,Y) = \widehat{\mathrm{HH}}{}^{-*}_{G}(f,Y) \circ \mathop{\widehat{\mathrm{ch}}_R}.
    \] 
    Moreover, it is also proved in \cite[Corollary 7.3.4]{Raven} that $\mathop{\widehat{\mathrm{ch}}_R}$ intertwines the boundary maps with respect to the first variable; $\mathop{\widehat{\mathrm{ch}}_R} \circ \delta_{\mathrm{KK}^G} = \delta_{\widehat{\mathrm{HH}}_G} \circ \mathop{\widehat{\mathrm{ch}}_R}$.
    
    As is stated in \cite[Theorem 2.19]{DW}, Raven's Chern character is natural with respect to the second variable, as well as the first variable. Indeed, for a continuous $G$-map $g \colon Y_1 \to Y_2$ of compact $G$-spaces, we have 
    \begin{align*}
        \mathrm{ch}_R^\gamma (g^*[M,f,\xi]) = {}&{} \mathrm{ch}_R^\gamma ([M,f,g^*\xi]) =  [f] \circ [\pi_{\mathring{M} \times Y_1 \to \mathring{M}}] \circ [[\mathrm{ch}_G^\gamma(g^*\xi)]] \circ (\mathrm{ch}_G^\gamma(\mathring{M}) \times 1_{Y_1^\gamma}) \\
    ={}&{}  [f] \circ [\pi_{\mathring{M} \times Y_2 \to \mathring{M}}] \circ [\id_{\mathring{M}} \times g] \circ  [[g^*\mathrm{ch}_G^\gamma(\xi)]] \circ (\mathrm{ch}_G^\gamma(\mathring{M}) \times 1_{Y_1^\gamma}) \\
    ={}&{}  [f] \circ [\pi_{\mathring{M} \times Y_2 \to \mathring{M}}] \circ  [[\mathrm{ch}_G^\gamma(\xi)]] \circ [\id_{\mathring{M}} \times g] \circ  (\mathrm{ch}_G^\gamma(\mathring{M}) \times 1_{Y_1^\gamma}) \\
    ={}&{}  [f] \circ [\pi_{\mathring{M} \times Y_2 \to \mathring{M}}] \circ  [[\mathrm{ch}_G^\gamma(\xi)]] \circ (\mathrm{ch}_G^\gamma(\mathring{M}) \times 1_{Y_2^\gamma}) \circ [g] \\
    ={}&{} \mathrm{ch}_R^\gamma (M,f,\xi) \circ [g]. 
    \end{align*}
    The equality at the third line is \cite[Proposition 6.5.12]{Raven}. Moreover, $\mathop{\widehat{\mathrm{ch}}_R}$ intertwines suspension isomorphisms, and hence $\delta$, with respect to the second variable. This is verified as
    \begin{align*}
        {}&{}\mathrm{ch}_R^\gamma ([[0,1] \times M,\id_{[0,1]} \times f,1_{[0,1]} \otimes \xi])\\
        = {}&{} (\id_{R\Gamma_c\mathbb{C}_{(0,1)}} \otimes [f]) \circ (\id_{R\Gamma_c\mathbb{C}_{(0,1)}} \otimes [\pi]) \circ (\id_{R\Gamma_c\mathbb{C}_{(0,1)}} \otimes [[\mathrm{ch}_G^\gamma (\xi)]]) \circ (\mathrm{ch}(0,1) \otimes \mathrm{ch}_G^\gamma(\mathring{M}) \times 1_{Y^\gamma}) \\ 
        ={}&{} \mathrm{ch}([[0,1],\id,1]) \otimes \mathrm{ch}_R^\gamma ([M,f,\xi]) = \alpha \otimes \mathrm{ch}_R^\gamma ([M,f,\xi]). 
    \end{align*}
    Here, we have used the fact that the Chern character of the fundamental class of $[0,1]$, $\mathrm{ch}([[0,1],\id,1]) \in \mathrm{HH}(\Sigma, \mathrm{pt})$, is the generator $\alpha$. 
\end{remark}

\begin{remark}\label{rmk:Raven.noncompact}
    The above definition is for $X$ cocompact. For non-cocompact $G$-CW-pair $(X,A)$, the groups are defined as
    \begin{align*} 
        \mathrm{F}^G((X,A),Y) \cong{} \varinjlim_{(Z,B) \subset (X,A)}\mathrm{F}^G((Z,B),Y) 
    \end{align*}
    for both $\mathrm{F}^G=\mathrm{KK}^G_*$ and $\mathrm{HH}_G^{*}$, where $(Z,B)$ runs over all pairs of cocompact $G$-CW-subcomplexes of $(X,A)$ (cf.~\cite[Definition~2.15]{DW}). Note that the KK-group of this definition is called the representable KK-group and is denoted by $\mathrm{RKK}_*^G((X,A),Y)$ in the literature. For example, for an infinite family of proper cocompact $G$-CW-complexes $\{X_j\}_{j \in J}$, 
    \begin{align}
    \bigoplus_{j \in J} \mathrm{F}^G(\iota_j,Y) \colon \bigoplus_{j \in J} \mathrm{F}^G(X_j,Y) \xrightarrow{\cong} \mathrm{F}^G\bigg( \bigsqcup_{j \in J} X_j,Y\bigg) \label{eqn:disjoint.union}
    \end{align}
    are isomorphisms for both $\mathrm{F}^G=\mathrm{KK}^G_*$ and $\widehat{\mathrm{HH}}{}_G^*$, where $\iota_j \colon X_j \to \bigsqcup X_j$ denotes the inclusion. 

    A continuous map $f \colon (X_1,A_1) \to (X_2,A_2)$, even if it is not proper, induces 
    \begin{align}
        \mathrm{F}^G(f,Y) \colon \mathrm{F}^G((X_1,A_1),Y) \to \mathrm{F}^G((X_2,A_2),Y). \label{eqn:nonproper.functoriality.tKK.HH}
    \end{align}
    Indeed, for any cocompact $G$-CW-subcomplex $(Z_1,W_1) \subset (X_1,A_1)$, the restriction $f |_{(Z_1,W_1)}$ becomes a proper map to some cocompact $G$-CW-subcomplex $(Z_2,W_2) \subset (X_2,A_2)$.  
    In the case of HH-theory, we add a brief remark on the explicit form of the map $\mathrm{HH}^*_G(f,Y)$. Notice that we have a homomorphism
    \begin{align*}
    \mathrm{HH}^*_G((X,A),Y) \cong {}&{} \varinjlim_{(Z,B) \subset (X,A)} \Hom_{\mathbf{D}^+(\mathbb{C}[G])}  (R\Gamma_c \mathbb{C}_{Z \setminus B} , R\Gamma_c\mathbb{C}_Y) \\
    {}&{}\to \Hom_{\mathbf{D}^+(\mathbb{C}[G])} \bigg( \varprojlim_{(Z,B) \subset (X,A)} R\Gamma_c\mathbb{C}_{Z,B}, R\Gamma_c \mathbb{C}_{Y} \bigg),
    \end{align*}
    which is not necessarily an isomorphism in general. 
    If there are increasing sequences $(Z_{a,i},B_{a,i}) \subset (X_a,A_a)$ for $a =1,2$ such that $f(Z_{1,i}) \subset Z_{2,i}$, the pull-back  $(f|_{Z_{1,i} \setminus B_{1,i}})^* \colon R\Gamma_c\mathbb{C}_{Z_{2,i} \setminus B_{2,i}} \to R\Gamma_c\mathbb{C}_{Z_{1,i} \setminus B_{1,i}}$ defines 
    \[
    f^* \coloneqq \varprojlim_i(f|_{Z_{1,i},B_{1,i}})^* \colon \varprojlim_{i} R\Gamma_c\mathbb{C}_{Z_{2,i} \setminus B_{2,i}} \to \varprojlim_{i} R\Gamma_c\mathbb{C}_{Z_{1,i} \setminus B_{1,i}}, 
    \]
    which can be composed with an element of $\mathrm{HH}^*_G((X_2,A_2),Y)$. For example, if $(X,A)$ is a disjoint union of proper cocompact $G$-CW-complexes $(X_i,A_i)$, we have $\varprojlim_{(Z,B) \subset (X,A)} R\Gamma_c\mathbb{C}_{Z \setminus B} \cong \prod_i R\Gamma_c \mathbb{C}_{X_i \setminus A_i}$. 

    Moreover, by its naturality with respect to the first variable (Remark \ref{rmk:Raven.natural.cohomological}), Raven's Chern character is defined for non-cocompact $(X,A)$ (\cite[Corollary 2.21]{DW}), and satisfies the same naturality $\mathop{\widehat{\mathrm{ch}}_R}\circ \mathrm{KK}^G_*(f,Y) = \widehat{\mathrm{HH}}{}_G^{-*} (f,Y) \circ \mathop{\widehat{\mathrm{ch}}_R}$. 
\end{remark}

    \begin{remark}\label{rmk:BS.Zcoeff}
        In this paper, we work throughout with Baum–Schneider homology with $\mathbb{C}$-coefficients. However, at one point in Lemma \ref{lem:Raven.Chern.torus}, we need to use the version with $\mathbb{Z}$-coefficients (as is treated in \cite{DW}). 
        In the statements of the following Lemmas \ref{lemma:induction.tKK.HH} and \ref{lem:induction.2}, the isomorphisms induced on Baum–Schneider homologies remains valid with $\mathbb{Z}$-coefficients by the same proof.
    \end{remark}
\begin{lemma}\label{lemma:induction.tKK.HH}
    Let $H$ be a subgroup of $G$. Then, for any pair of proper $H$-spaces $(X,A)$ and any $G$-space $Y$, there are isomorphisms
    \begin{align*}
        \mathrm{KK}_*^G(G \times_H (X,A) , Y) \cong \mathrm{KK}_*^H((X,A),Y), \quad \widehat{\mathrm{HH}}{}_G^*(G \times_H (X,A),Y) \cong \widehat{\mathrm{HH}}{}_H^*((X,A),Y),
    \end{align*}
    that form morphisms of $H$-equivariant homology functors with respect to the first variable, and are compatible with localized Chern character.
\end{lemma}
    Here, for a pair $(X,A)$ of $H$-spaces, we write $G \times_H (X,A) \coloneqq  (G \times_H X,G \times_H A)$. 
    When $H$ is a finite subgroup of $G$, it is proved in \cite[Propositions 4.7.5, 6.5.11]{Raven} that they are isomorphic, and in \cite[Proposition 7.3.9]{Raven} that they are compatible with Raven's Chern character (under an additional assumption that $X=\mathrm{pt}$). 
\begin{proof}    
    Following \cite[Definitions 4.7.2, 6.5.10]{Raven}, let 
    \begin{align*}
        i_H^G \colon & \mathrm{KK}_*^H((X,A),Y) \to  \mathrm{KK}_*^G(G \times_H (X,A), Y), && i_H^G[M,f,\xi] = [G \times_H M, \id_G \times_H f , \mathrm{Ind}_H^G \xi] , 
        \\
        i_H^G \colon & \mathrm{HH}_H^*((X,A),Y) \to  \mathrm{HH}_G^*(G \times_H (X,A), Y), && i_H^G (\phi) = (\id_{\mathbb{C}[G]} \otimes_{\mathbb{C}[H]} \phi) \circ \mu_Y.
    \end{align*}
    Here, the induction map 
    \[
    \operatorname{Ind}_H^G \colon \mathrm{K}_0^H(M \times Y,\partial M \times Y) \to \mathrm{K}_0^G(G \times _H (M \times Y)  , G \times_H (\partial M \times  Y) )
    \]
    is given by sending $[E,F,u] \in \mathrm{K}_0^H(M \times Y, \partial M \times Y)$, where $E,F$ are $H$-vector bundles over $M \times Y$ and $u \colon E|_{\partial M \times Y} \to F|_{\partial M \times Y} $ is an $H$-equivariant bundle isomorphism, to $[G \times_H E, G \times_H F, \id_G \times_H u]$. Note that, since $Y$ is assumed to be a $G$-space, we have an isomorphism of $G$-spaces
    \[
        G \times _H (M \times Y) \cong (G \times_H M ) \times Y, \quad [g,(m,y)] \mapsto ([g,m],gy).
    \]
    We emphasize once again, as already noted in Remark \ref{rmk:Raven.precise}, that the composition symbol for HH-classes is taken in the opposite direction from that for morphisms.
    As in \cite[Proposition 6.5.8]{Raven}, for a $H$-sheaf $F$ on $X \setminus A$, we canonically regard $\mathbb{C}[G] \otimes_{\mathbb{C}[H]} F$ as a sheaf on $G \times_H X$. 
    This gives the induction functor between the categories of nonnegative cochain complexes of $G$-equivariant sheaves on $X \setminus A$ as
    \[ 
    \mathbb{C}[G] \otimes_{\mathbb{C}[H]} \cdot \colon \mathbf{K}_H^+(X \setminus A, \mathbb{C}) \to \mathbf{K}_G^+(G \times_H (X \setminus A) ; \mathbb{C}).
    \]
    Since it preserves the $c$-softness, the tensor product morphism $\id_{\mathbb{C}[G]} \otimes_{\mathbb{C}[H]} \phi$ is regarded as an element of $\mathrm{HH}_G^*(G \times _H (X,A), G \times_H Y)$. 
    The element $\mu_Y \in \mathrm{HH}_G^*(G \times_H Y,Y)$ is given by  
    \[ 
    \mu_Y \colon R\Gamma_c\mathbb{C}_{G \times_H Y} \cong \mathbb{C}[G] \otimes_{\mathbb{C}[H]} R\Gamma_c\mathbb{C}_Y \to R\Gamma_c \mathbb{C}_Y, \quad \mu_Y(g,s) = g \cdot s \quad \text{ for $s \in R\Gamma_c\mathbb{C}_Y$. }
    \]
    The inverse is given by
    \begin{align*}
        r_G^H \colon & \mathrm{KK}_*^G(G \times _H (X,A), Y) \to \mathrm{KK}_*^H((X,A),Y), && r_G^H[M,f,\xi] = [f^{-1}(H\times_H X),f,\mathrm{Res}_G^H \xi] , 
        \\
        r_G^H \colon &  \mathrm{HH}_G^*(G \times_H (X,A), Y) \to \mathrm{HH}_H^*((X,A),Y), && r_G^H \psi  =\eta_{X\setminus A} \circ \mathrm{Res}_G^H \psi,
    \end{align*}
    where $\mathrm{Res}_G^H \colon \mathrm{K}^0_G(M \times Y, \partial M \times Y) \to \mathrm{K}^0_H(M \times Y, \partial M \times Y)$ and $\mathrm{Res}_G^H \colon \mathbf{D}^+(\mathbb{C}[G]) \to \mathbf{D}^+(\mathbb{C}[H])$ are the restriction functors, and 
    \[
    \eta_{X\setminus A} \in \Hom_{\mathbf{D}^+(\mathbb{C}[H])}(R\Gamma_c\mathbb{C}_{X\setminus A} , R\Gamma_c\mathbb{C}_{G \times_H (X \setminus A)})
    \]
    is given by 
    \[ 
    \eta_{X\setminus A} \colon R\Gamma_c \mathbb{C}_{X\setminus A} \to R\Gamma_c\mathbb{C}_{G \times_H (X \setminus A)} \cong \mathbb{C}[G] \otimes_{\mathbb{C}[H]} R\Gamma_c \mathbb{C}_{X\setminus A}, \quad \eta_{X \setminus A}(s)= \delta_e \otimes s.
    \]
    They are inverses to each other. Indeed, for the $\mathrm{KK}^G_*$-functors, both $r_G^H \circ i_H^G = \id$ and $i_H^G \circ r_G^H=\id$ are clear from the definition. To see the same for the $\mathrm{HH}_G^*$-functors, let us take a projective resolution $P^\bullet$ of $R\Gamma_c\mathbb{C}_{X \setminus A}$ and $Q^\bullet$ of $R\Gamma_c\mathbb{C}_Y$. Note that $\mathbb{C}[G] \otimes_{\mathbb{C}[H]} P^\bullet$ is a projective resolution of $R\Gamma_c\mathbb{C}_{G \times_H X} \cong \mathbb{C}[G] \otimes_{\mathbb{C}[H]} R\Gamma_c\mathbb{C}_X$. Now, for $\phi \in \mathrm{HH}_H^n((X,A),Y) = \Hom_{\mathbf{K}^+(\mathbb{C}[H])} (P,Q[n])$ and $\psi \in \mathrm{HH}_G^n(G \times_H (X,A),Y) = \Hom _{\mathbf{K}^+(\mathbb{C}[G])}(\mathbb{C}[G] \otimes_{\mathbb{C}[H]} P,Q[n])$, we have  
    \begin{gather*}
        (\eta_{X \setminus A} \circ  (\id_{\mathbb{C}[G]} \otimes_{\mathbb{C}[H]}\varphi) \circ \mu_Y) \cdot s =  \mu_Y (e \otimes \varphi(s)) = \varphi(s),\\
        \mu_Y \circ (\id_{\mathbb{C}[G]} \otimes_{\mathbb{C}[H]} (\psi \circ \eta_{X \setminus A}))\cdot (g \otimes s) = \mu_Y(g \otimes \psi(e \otimes s)) = g \cdot (\psi (e \otimes s)) = \psi(g \otimes s),
    \end{gather*}
    for any $s \in P^m$. They prove $r_G^H \circ i_H^G = \id$ and $i_H^G \circ r_G^H=\id$. Here, for the latter equality, we have used the fact that $\psi$ is by definition a $\mathbb{C}[G]$-module map.

    The isomorphism of the delocalized parts of the bivariant homology theory is proved from this case. 
    Indeed, for $\gamma \in G_{\mathrm{tor}}$, we have
    \begin{align*}
    (G \times_H X)^\gamma = {}&{} \{ [g,p] \in G \times_H X : g^{-1} \gamma g =: \gamma' \in H,  \gamma' \cdot p=p  \}\\
    = {}&{} \bigsqcup_{
    \substack{[\gamma'] \in \mathcal{C}(H_{\mathrm{tor}})\\ \gamma '=g^{-1}\gamma g}} gZ_G(\gamma') \times_{Z_H(\gamma')} X^{\gamma'}
    \end{align*}
    as $Z_G(\gamma)$-spaces (here $Z_H(\gamma)$ and $Z_G(\gamma)$ denotes the centralizer subgroups in $H$ and $G$ respectively). This shows that the following composition is bijective:
    \begin{align*}
    \bigoplus_{[\gamma] \in \mathcal{C}(G_{\mathrm{tor}})}\mathrm{HH}_{Z_G(\gamma) }^*((G \times_H X)^\gamma , Y^\gamma )
    \cong {}&{}
    \bigoplus_{[\gamma] \in \mathcal{C}(G_{\mathrm{tor}})} 
    \bigoplus_{
    \substack{[\gamma'] \in \mathcal{C}(H_{\mathrm{tor}})\\ \gamma '=g^{-1}\gamma g}} \mathrm{HH}_{Z_G(\gamma)}^*(gZ_G(\gamma') \times_{Z_H(\gamma')}X^{\gamma'}, Y^{\gamma})\\
    \cong {}&{}
    \bigoplus_{[\gamma] \in \mathcal{C}(G_{\mathrm{tor}})} 
    \bigoplus_{
    \substack{[\gamma'] \in \mathcal{C}(H_{\mathrm{tor}})\\ \gamma '=g^{-1}\gamma g}} \mathrm{HH}_{Z_G(\gamma')}^*(Z_G(\gamma') \times_{Z_H(\gamma')}X^{\gamma'}, Y^{\gamma'})\\
    \xrightarrow{\oplus r^{Z_H(\gamma')}_{Z_G(\gamma')}}{}&{} \bigoplus_{[\gamma'] \in \mathcal{C}(H_{\mathrm{tor}})} \mathrm{HH}^*_{Z_H(\gamma')}(X^{\gamma'},Y^{\gamma'}).  
    \end{align*}

   Finally, we verify that Raven's Chern character commutes with $r_G^H$. 
   According to \eqref{eqn:Raven.Chern}, for $\gamma \in H_{\mathrm{tor}}$, Raven's Chern character localized at $\gamma$ of an induced KK-cycle $i_H^G[M,f,\xi]$, namely
   \[
      \mathrm{ch}_R^\gamma \circ i_H^G[ M, f,  \xi]) =
   \mathrm{ch}_R^\gamma([G \times_H M,\id_G \times_H f, \mathrm{Ind}_H^G\xi]),
   \]
   is given by the composition of pull-backs, exterior products with characteristic forms, and the integration, as 
    \begin{align}
    R\Gamma_c\mathbb{C}_{(G \times_H (X \setminus A))^\gamma} \xrightarrow{f^* }{}&{} R\Gamma_c\mathbb{C}_{(G \times_H M)^\gamma} = \Omega_c^\bullet ((G \times _H M)^\gamma ) \xrightarrow{\pi_{M^\gamma} ^* } \Omega_c^\bullet ((G \times_H M)^\gamma ) \otimes \mathbb{C}[Y^\gamma] \\
    \xrightarrow{ \mathrm{ch}_G^\gamma (\mathrm{Ind}_H^G \xi)} {}&{} \Omega_c^{\bullet}((G \times_H M)^\gamma ) \otimes \mathbb{C}[Y^\gamma]  
    \xrightarrow{\mathrm{Td}_G^\gamma((G \times_H M)^\gamma ) \otimes \id} \Omega_c^\bullet((G \times_H M)^\gamma) \otimes \mathbb{C}[Y^\gamma] \\
    \xrightarrow{\mathrm{ch}_G^\gamma (\beta_{T(G \times M)^\natural \oplus n}) }{}&{} \Omega_c^{\bullet}((G \times_H (TM \oplus n))^\gamma ) \otimes \mathbb{C}[Y^\gamma] \xrightarrow{\int_{(TM \oplus n)^\gamma } \otimes \id} \mathbb{C}[Y^\gamma] .
    \end{align}
    Applying $r_G^H$ to this composition is nothing but restricting it to the identity component $(eH \times_H X)^\gamma = X^\gamma$. 
    Meanwhile, the restrictions of the characteristic forms $\mathrm{ch}_G^\gamma(\mathrm{Ind}_H^G \xi)$, $\mathrm{Td}_G^\gamma(G \times_H M)$, and $\mathrm{ch}_G^\gamma (\beta_{T(G \times M^\natural) \oplus n})$ to $M^\gamma$ are nothing but those defined from $[M,f,\xi]$, that is, $\mathrm{ch}_H^\gamma(\xi)$, $\mathrm{Td}_H^\gamma(M)$, and $\mathrm{ch}_H^\gamma (\beta_{TM^\natural \oplus n})$. This concludes that 
    \[
    r_G^H \circ \mathrm{ch}_{R,G}^\gamma \circ i_H^G [M,f,\xi] = \mathrm{ch}_{R,H}^\gamma [M,f,\xi], 
    \]
    where $\mathrm{ch}_{R,G}^\gamma$ and $\mathrm{ch}_{R,G}^\gamma$ denote Raven's $G$- and $H$- equivariant Chern character.

    Finally, we show the naturality of $i_H^G$. If one has a continuous $H$-map $F \colon (X_1,A_1) \to (X_2,A_2)$, which induces $\id_G \times_H F \colon G \times_H (X_1,A_1) \to G \times_H (X_2,A_2)$, then 
    \begin{align*}
        \KK^G_*(\id_G \times_H F,Y) \circ i_H^G[M,f,\xi] ={}&{} [G \times_H M,(\id_G \times_H F) \circ (\id_G \times_H f),\operatorname{Ind}_H^G\xi] = i_H^G \circ \KK^G_*(F,Y)[M,f,\xi], \\
        \mathrm{HH}_G^*(\id_G \times_H F,Y) \circ i_H^G(\varphi) ={}&{} (\id_{\mathbb{C}[G]} \otimes_{\mathbb{C}[H]} [F]) \circ (\id_{\mathbb{C}[G]} \otimes_{\mathbb{C}[H]}\phi) \circ \mu_Y = i_H^G \circ \mathrm{HH}_H^*(F,Y)(\varphi). 
    \end{align*}
    This shows that $i_H^G$ are natural transformations. We also have
        \begin{align*}
        i_H^G [[0,1] \times M, \id_{[0,1] } \times f,\xi] ={}&{} [[0,1] \times G \times_HM,\id_{[0,1]} \times \id_G \times_H f,1_{[0,1]} \otimes \operatorname{Ind}_H^G \xi] = i_H^G[M,f,\xi], \\
        i_H^G (\varphi \otimes \alpha)={}&{} (\id_{\mathbb{C}[G]} \otimes_{\mathbb{C}[H]}( \phi \otimes \alpha )) \circ (\mu_Y \otimes \id_{(0,1)}) = i_H^G(\phi) \otimes \alpha, 
    \end{align*}
    which shows that $i_H^G$ are compatible with the suspension isomorphisms, and hence with the boundary maps $\delta$ by Remark \ref{rmk:bivariant.homologies.3}. 
\end{proof}

\begin{lemma}\label{lem:induction.2}
Let $(X,A)$ be a free and proper $G$-CW-complex, let $H$ be a subgroup of $G$, and let $Y$ be an $H$-space. Then, there are isomorphisms 
\[
    \mathrm{KK}^H_*((X,A),Y) \cong \mathrm{KK}^G_*((X,A) , G \times_H Y), \quad \mathrm{HH}^*_H((X,A),Y) \cong \mathrm{HH}_G^*((X,A),G \times_H Y),
\]
that are morphisms of $G$-equivariant generalized homology functors on the category of free and proper $G$-CW-complexes with respect to the first variable, and compatible with the localized Chern character. 
\end{lemma}
The proof of \cite[Proposition 6.6]{ChaEch} indicates that the assumption of the freeness of $G \curvearrowright X$ is unnecessary. 
The general case, however, requires a more elaborate argument, which lies outside the scope of this paper, so we restrict ourselves to a simpler situation. 
\begin{proof}
For the definition of the isomorphisms, consider the diagram
\[
\xymatrix{
   \KK^{H}_*((X,A),Y) \ar[r]^{\mathrm{ch}_R^e} \ar[d]^{j_*} &
   \bigoplus_k \mathrm{HH}_{H}^{-*+2k}((X,A),Y) \ar[d]^{j_*  } \\
   \KK^{H}_*((X,A),G \times_H Y) \ar[r]^{\mathrm{ch}_R^e} \ar[d]^{i_H^G} &
   \bigoplus_k \mathrm{HH}_{H}^{-*+2k}((X,A),G \times_H Y) \ar[d]^{i_H^G} \\
   \KK^{G}_*(G \times_H (X,A) , G \times_H Y) \ar[r]^{\mathrm{ch}_R^e} \ar[d]^{ \mathrm{KK}_*^G(\alpha,G \times_HY)} &
   \bigoplus_k \mathrm{HH}_{G}^{-*+2k}(G \times_H (X,A) , G \times_H Y) \ar[d]^{\mathrm{HH}_G^*(\alpha,G \times_HY)}\\
   \KK^{G}_*((X,A) , G \times_H Y) \ar[r]^{\mathrm{ch}_R^e} &
   \bigoplus_k \mathrm{HH}_{G}^{-*+2k}((X,A) , G \times_H Y).
}
\]
 Here, $\alpha \colon G \times_H (X,A) \to (X,A)$ is the $G$-action on $X$ and $j \colon Y \to G \times _H Y$ denotes the open embedding to $eH \times _H Y \subset G \times_H Y$. For $\mathrm{F}^H_* = \mathrm{KK}^H_* $ and $\bigoplus_{k}\mathrm{HH}_H^{-*+2k}$, this $j$ induces the morphisms of KK- and HH- groups as the restriction
 \[
     j_* \coloneqq \big( \mathrm{F}^H_*((X,A),J) \colon \mathrm{F}_*^H((X,A), Y^+) \to \mathrm{F}_*^H((X,A), (G \times _H Y)^+) \big)\big|_{\mathrm{F}_*^H((X,A),Y)},
 \]
where $J \colon (G \times_H Y)^+ \to Y^+$ is the continuous $H$-map such that $J|_{j(Y)} =j^{-1}$ and $J|_{j(Y)^c}$ is the constant map to the basepoint $\mathrm{pt} \in Y^+$. 
 The middle square commutes by Lemma \ref{lemma:induction.tKK.HH}, and the top and bottom squares commute by the naturality of Raven's Chern character (\cite[Theorem 7.3.11]{Raven} and Remark \ref{rmk:Raven.natural.cohomological}). Moreover, for any a $G$-map $f \colon (X_1,A_1) \to (X_2, A_2)$, we have
 \begin{align*}
    {}&{}\mathrm{F}_*^G(\alpha, G \times_H Y) \circ i_H^G \circ \mathrm{F}_*^H((X,A),J) \circ \mathrm{F}_*^H(f,Y) \\
    ={}&{} \mathrm{F}_*^G(\alpha, G \times_H Y) \circ i_H^G \circ \mathrm{F}_*^H(f,G \times_H Y) \circ \mathrm{F}_*^H((X,A),J) \\
    ={}&{}\mathrm{F}_*^G(\alpha, G \times_H Y) \circ \mathrm{F}_*^G ( \id_G \times_H f,G \times_H Y) \circ i_H^G \circ \mathrm{F}_*^H((X,A),J) \\
= {}&{} \mathrm{F}_*^G(f,G \times_H Y) \circ \mathrm{F}_*^G(\alpha, G \times_H Y) \circ  i_H^G \circ \mathrm{F}_*^H((X,A),J),
 \end{align*}
 for both $\mathrm{F}_*^H=\mathrm{KK}^H_*$ and $\bigoplus_k \mathrm{HH}^{*+2k}_H$. This proves the naturality of the isomorphism. Here, the first equality follows from Remark \ref{rmk:bivariant.homologies}; the second, from the naturality of $i_H^G$ proved in Lemma \ref{lemma:induction.tKK.HH}; and the third, from the identity $f \circ \alpha = \alpha \circ (\id_G \times_H f)$, which holds by $G$-equivariance of $f$. 
 Similarly, since the suspension isomorphism $\mathrm{F}^H_*(X,Y) \to \mathrm{F}^H_{*+1}(\Sigma X,Y)$ is natural with respect to morphisms in both the first and second variables (Remarks \ref{rmk:bivariant.homologies.2} and \ref{rmk:bivariant.homologies.3}), as well as with respect to $i_H^G$ (Lemma \ref{lemma:induction.tKK.HH}), the above isomorphism intertwines the suspension isomorphism, and hence $\delta$, as well.

 We then show that the two vertical compositions are isomorphisms. The proof is the same as the one for Kasparov KK-theory, given in \cite[Proposition 6.6]{ChaEch}. First, by a standard Mayer--Vietoris argument comparing the long exact sequences, we reduce the isomorphism to the case that $(X,A)=(G,\emptyset)$. We first consider the case that $X$ be a free proper finite $G$-CW-complex and $A=\emptyset$. This $X$ is obtained from $X_0=\emptyset$ by iterating finitely many times the operation of attaching a $G$-cell $D_n = G \times D^{k_n}$ to $X_n$ along its boundary to form $X_{n+1}$.
 Since $\mathrm{F}_* \coloneqq \mathrm{KK}^H_*(\cdot, Y)$ and $\mathrm{F}_*' \coloneqq \mathrm{KK}^G_*(\cdot , G \times_H Y)$ (resp.\ $\mathrm{F}_* \coloneqq \mathrm{HH}_H^{-*}(\cdot, Y)$ and $\mathrm{F}_*' \coloneqq \mathrm{HH}_G^{-*}(\cdot , G \times_H Y)$) are $G$-equivariant homology functors (cf.\ Remark \ref{rmk:homology.functor.boundary}), it intertwines the boundary homomorphisms $\delta$, and hence the diagram 
    \[
    \xymatrix{
    \mathrm{F}_{*+1}(D^{k_n},\partial D^{k_n})  \ar[d]^\cong \ar[r]^{\delta} &
    \mathrm{F}_*(X_n) \ar[r] \ar[d]^\cong &  
    \mathrm{F}_*(X_{n+1}) \ar[r] \ar[d] & 
    \mathrm{F}_*(D^{k_n},\partial D^{k_n}) \ar[r]^\delta \ar[d]^\cong &
    \mathrm{F}_{*-1}(X_n) \ar[d]^\cong &  
 \\     
    \mathrm{F}_{*+1}'(D^{k_n},\partial D^{k_n}) \ar[r]^\delta &  
    \mathrm{F}_*'(X_n) \ar[r]  & 
    \mathrm{F}_*'(X_{n+1}) \ar[r]  & 
    \mathrm{F}_*'(D^{k_n},\partial D^{k_n}) \ar[r]^\delta &
    \mathrm{F}_{*-1}'(X_n)  &  
    }
    \]
    commutes by taking as the vertical maps those constructed above. An iterative use of the five lemma proves the isomorphism. For the case where $A$ is non-empty, applying the same comparison to the long exact sequence for the pair $(X,A)$ yields the desired isomorphism. Finally, for possibly infinite proper $G$-CW-pairs $(X,A)$, the groups $\mathrm{F}(X,A)$ and $\mathrm{F}'(X,A)$ are both defined as limits (Remark \ref{rmk:Raven.noncompact}), thus the isomorphism established above extends immediately.

    In the case of $(X,A)=(G,\emptyset)$, by Lemma \ref{lemma:induction.tKK.HH} we have
 \begin{gather}
    \mathrm{F}^H(G,Y) \cong  \bigoplus_{Hg \in H \backslash G} \mathrm{F}^H(Hg, Y) \cong \bigoplus_{H g \in H \backslash G} \mathrm{F}(\mathrm{pt}, Y), \\
    \mathrm{F}^G(G , G \times_H Y) \cong \mathrm{F}(\mathrm{pt},G \times_H Y) \cong \bigoplus_{gH \in G/H} \mathrm{F}(\mathrm{pt}, Y), 
 \end{gather}
 for $\mathrm{F}^G=\mathrm{KK}_*^G$ or $\mathrm{HH}^*_G$. 
We now prove that $\mathrm{F}^G(\alpha,G \times_H Y) \circ i_H^G \circ j_* \colon \mathrm{F}^H(G,Y) \to \mathrm{F}^G(G ,G \times_H Y)$ is identified, under the above isomorphisms, with the direct sum of the identity maps on the respective summands, where the left and the right coset corresponds by $Hg \mapsto g^{-1}H$. 

We first prove the claim for KK-theory. Let $[M,\pr,\xi] \in \KK(\mathrm{pt},Y)$, i.e., $M$ is a closed $\mathrm{Spin}^c$-manifold and $\xi \in \K^0(M \times Y^+,M \times \mathrm{pt})$. The corresponding element in $\KK^H(Hg,Y)$ is $[Hg \times M, \pr_{Hg}, \bar{\xi}]$, where  $\bar{\xi} \in \K^0_H(Hg \times M \times Y^+, Hg \times M \times \mathrm{pt})$ is $H$-invariant and $\bar{\xi}|_{\{g\} \times M \times Y} = \pr_{M \times Y}^*\xi$. This element is sent to 
\begin{align*}
    {}&{}\KK^G(\alpha,G \times_H Y) \circ i_H^G \circ \KK^H(\mathrm{pt},J) [Hg \times M, \pr_{Hg}, \bar{\xi}] \\
    ={}&{}\KK^G(\alpha,G \times_H Y) \circ i_H^G [Hg \times M, \pr_{Hg}, J^*\bar{\xi}] \\
    ={}&{} \KK^G(\alpha,G \times_H Y)  [G \times_H Hg \times M, \pr_{G \times_H Hg}, \Ind_H^G \circ  (\id_M \times J)^*\bar{\xi}] \\
    ={}&{} [G \times_H Hg \times M, \alpha \circ \pr_{G \times_H Hg}, \Ind_H^G \circ (\id_M \times J)^*\bar{\xi}] \\
    ={}&{} [G \times M, \pr_{G} , (\alpha^{-1} \times \id_{M} \times \id_{Y})^* \circ \Ind_H^G \circ (\id_M \times J)^*\bar{\xi}]. 
\end{align*}
The last equality is given by the identification of the topological KK-cycles given by the diffeomorphism $\alpha \times \id_M$. Now, the element 
\[
    (\alpha^{-1} \times \id_{M} \times \id_{Y})^* \circ \Ind_H^G \circ (\id_M \otimes J)^*(\bar{\xi}) \in \K^0_G(G \times M \times (G \times_H Y^+), G \times M \times (G \times_H \mathrm{pt}))
\]
is $G$-invariant, is supported on 
\[
    (\alpha \times \id_M \times \id_Y) \big( G \cdot ((H \times_H Hg) \times M \times j(Y)) =G \cdot (Hg \times M \times j(Y)),
\]
and restricts to $\xi$ on $\{g\} \times M \times j(Y)$. By the $G$-equivariance, its restriction to $\{e \} \times M \times (G \times _H Y)$ is supported on $\{e \} \times M \times g^{-1}j(Y)$ and the corresponding restriction is $g^*\xi \in \K^0(M \times g^{-1}j(Y))$. That is, 
\[
    r_G^{\{e\}} \circ \KK^G(\alpha,G \times_H Y) \circ i_H^G \circ j_*[Hg \times M, \pr_{Hg}, \bar{\xi}] = [\{e \} \times M, \pr, g^*\xi] \in \KK(\mathrm{pt}, g^{-1}H \times_H Y). 
\]
By the pull-back via $g^{-1} \colon gH \times_H Y \to H \times_H Y \cong Y$, we go back to the original element $[M,\pr,\xi] \in \KK(\mathrm{pt},Y)$.

We then show the claim for HH-theory. Let $S_Y^\bullet$ be a $c$-soft resolution of the $G$-sheaf $\mathbb{C}_Y$, and let $\varphi \in \Hom_{\mathbf{K}^+(\mathbb{C})} (\mathbb{C}, \Gamma_cS_Y[n] ) \cong \mathrm{HH} (\mathrm{pt},Y)$ (recall that $\mathbb{C}_{X}$ is a c-soft resolution as itself if $X$ is discrete). The corresponding element $\bar{\varphi} \in \Hom_{\mathbf{K}^+(\mathbb{C}[H])} (\mathbb{C}[Hg], \Gamma_cS_Y[n] ) \cong \mathrm{HH}^H(Hg,Y)  $ is characterized by its $H$-invariance and $\bar{\varphi}(\delta_g) =  \varphi (1)$. 
Then
\begin{align}
    \begin{split}
    \mathrm{HH}_G^n(\alpha,G \times_H Y) \circ i_H^G \circ j_*(\bar{\varphi}) 
    ={}&{}\mathrm{HH}^G(\alpha,G \times_H Y) \circ  i_H^G( \bar{\varphi} \circ [J]) \\
    ={}&{} \mathrm{HH}^G(\alpha,G \times_H Y) ( (\id \otimes \bar{\varphi} \circ [J]) \circ \mu_{G \times_H Y}) \\
    ={}&{}[\alpha] \circ (\id \otimes \bar{\varphi} \circ j_*) \circ \mu_{G \times_H Y}. 
    \end{split}\label{eqn:induction2.HH}
\end{align}
The right hand side, which is explicitly given as
\begin{align*}
    \mathbb{C}[G] \xrightarrow{[\alpha]} \mathbb{C}[G \times_H Hg] ={}&{}\mathbb{C}[G] \otimes_{\mathbb{C}[H]}\mathbb{C}[Hg] \xrightarrow{\id \otimes \bar{\varphi}} \mathbb{C}[G]\otimes_{\mathbb{C}[H]}\Gamma_cS_Y[n] \\
    {}&{}\xrightarrow{\id \otimes j_*}\mathbb{C}[G] \otimes_{\mathbb{C}[H]} \big( \mathbb{C}[G] \otimes_{\mathbb{C}[H]} \Gamma_cS_Y[n] \big) \xrightarrow{\mu_{G \times_HY}} \mathbb{C}[G] \otimes_{\mathbb{C}[H]} \Gamma_cS_Y[n],  
\end{align*}
sends $\delta_e \in \mathbb{C}[G]$ as
\[
    \delta_e \mapsto \delta_{g^{-1}} \otimes \delta_g \mapsto \delta_{g^{-1}} \otimes \bar{\varphi}(\delta_g) = \delta_{g^{-1}} \otimes \varphi(1) \mapsto \delta_{g^{-1}} \otimes \delta_e \otimes \varphi(1) \mapsto \delta_{g^{-1}} \otimes g^{-1}\varphi(1).
\]
This shows that 
\[
        \Big( r_G^{\{e\}} \circ \mathrm{HH}_G^n(\alpha,G \times_H Y) \circ i_H^G \circ j_*(\bar{\varphi})\Big) (1)  = \delta_{g^{-1}} \otimes g^{-1}\cdot \varphi(1) \in \Hom_{\mathbf{K}^+(\mathbb{C})} (\mathbb{C}, \mathbb{C}[g^{-1}H] \otimes_{\mathbb{C}[H]}\Gamma_cS_Y[n]).
\]
By the pull-back via $g^{-1} \colon gH \times_H Y \to H \times_H Y \cong Y$, corresponding to $g \colon \mathbb{C}[g^{-1}H] \otimes_{\mathbb{C}[H]}\Gamma_cS_Y[n] \to \mathbb{C}[H] \otimes_{\mathbb{C}[H]} \Gamma_cS_Y[n]$, we go back to the original element $\varphi$.
\end{proof}

\end{document}